\documentclass[a4paper,11pt]{article}
\usepackage[dvipsnames]{xcolor}
\usepackage[ top=1in,bottom=1in,left=1in,right=1in]{geometry}

\usepackage[intlimits]{amsmath}
\usepackage{amssymb}

\usepackage{mathrsfs}

\usepackage{tikz}
\usepackage{pgfplots}
\usetikzlibrary{arrows}

\usepackage{fancyhdr}
\usepackage{titlesec}

\titleformat{\section}
  {\center\normalfont\large\bfseries}{\thesection.}{0.1cm}{}

\titleformat{\subsection}
  {\center\normalfont\normalsize\bfseries}{\thesubsection.}{0.1cm}{}

\usepackage{setspace} 

\usepackage{tcolorbox}
\usepackage[bb=px]{mathalpha}
\usepackage{amsthm}
\usepackage{enumitem}
\usepackage{float}

\numberwithin{equation}{section}
\usepackage[colorlinks = true,
			urlcolor=blue,
			citecolor = Red,
linkcolor = blue]{hyperref}
\usepackage[capitalize,nameinlink]{cleveref}
\newtheorem{theorem}{Theorem}
\newtheorem{lemma}{Lemma}[section]

\newtheorem{remark}{\it Remark\/}

\newtheorem*{acknowledgements}{Acknowledgements}

\DeclareMathOperator*{\esssup}{ess\,sup}

\providecommand{\norm}[1]{\lVert#1\rVert}

\providecommand{\abs}[1]{\lvert#1\rvert}

\title{\Large{\textbf{\scshape{Unbounded  solutions for  the Muskat problem}}}}
\author{\small{\scshape{Omar S\'anchez }}}

\date{}

\AtBeginDocument{}
\begin{document}
\maketitle
\begin{abstract}
	We prove the local existence of solutions of the form $x^2+ct+g,$ with $g\in H^s(\mathbb R)$ and $s\geq 3,$ for the Muskat problem in the stable regime. We use energy methods to obtain a bound of $g$ in Sobolev spaces. In the proof we deal with the loss of the Rayleigh-Taylor condition at infinity and a new structure of the kernels in the equation. Remarkably, these solutions grow quadratically at infinity.

\end{abstract}
\thispagestyle{empty}
\tableofcontents

 \section{Introduction}
The Muskat problem  models the interaction of two  immiscibles  fluids with different densities in a porous medium.  The fluids are separated by an interface, which  splits the plane $\mathbb R^2$ in two fluid domains $\Omega_{+}$ and $\Omega_-.$ This problem was  originally introduced by Morris Muskat in \cite{Muskat} as a model for oil extraction and  has attracted great interest from mathematicians in  recent decades.  The equation governing the dynamic of the fluids is  Darcy's law
    \begin{equation}\label{dl}
		\begin{split}
	\frac{\mu}{\kappa}v^\pm=-\nabla p^\pm-\rho^\pm \mathrm{g}e_2&\quad \mbox{ in }\quad \Omega_{\pm},\\
		\end{split}
	\end{equation}
	where $v^{\pm}$ is the velocity, $\rho^{\pm}$ the density and  $p^{\pm}$ the pressure in the fluids domains $\Omega_{\pm}.$   The  viscosity $\mu$, the  permeability $\kappa$ and $\mathrm{g}$ the gravity are constants and we will assume that they are all equal to 1. The density 
	\[\rho(\mathbf{x},t)=\left\{\begin{array}{rl}
\rho^+(\mathbf{x},t), &  \mathbf{x}\in \Omega_+(t),\\
\rho^-(\mathbf{x},t), &   \mathbf{x}\in \Omega_-(t),
\end{array}\right.\]
 where $\mathbf{x}=(x,y)\in\mathbb R^2,$ satisfies the mass conservation equation
	 \begin{equation}\label{ms}
		\begin{split}
	\partial_t\rho+v\cdot \nabla \rho =0&\quad \mbox{ in }\quad \mathbb R^2,\\
		\end{split}
	\end{equation}
	in a weak sense. Here 
		\[v(\mathbf{x},t)=\left\{\begin{array}{rl}v^+(\mathbf{x},t), &  \mathbf{x}\in \Omega_+(t),\\
v^-(\mathbf{x},t),&   \mathbf{x}\in \Omega_-(t).
\end{array}\right.\]
	We will also assume the fluids are incompressible, \emph{i.e.}
		 \begin{equation}\label{in}
		\begin{split}
		\mathrm{div}(v^\pm)=0&\quad \mbox{ in }\quad \Omega_{\pm}.
		\end{split}
	\end{equation}
	In general,  the interface  could be   an arbitrary curve, in our case we will assume that it  is parameterized by the graph of a function $h$ (see figure \ref{fig1}). Thus 
	\[\partial\Omega_{\pm}(t)=\{(x,h(x,t)):x\in\mathbb R,t>0\}.\]
We assume  that the density $\rho^{\pm}(\mathbf{x},t)$   is a step function
	\[\rho^\pm(\mathbf{x},t)=\left\{\begin{array}{rl}
\rho^+, &  \{y>h(x,t)\},\\
\rho^-,&   \{y<h(x,t)\},
\end{array}\right.\]
where $\rho^{\pm}\in\mathbb R$ are two constant values.
\begin{figure}[H]
				\centering
		\definecolor{zzttqq}{rgb}{0.4,0.8,0.}
\begin{tikzpicture}[scale=1.8]
\fill[line width=1.pt,color=zzttqq,fill=zzttqq,fill opacity=0.10000000149011612] (0.,0.) -- (0.,3.) -- (4.,3.) -- (4.,0.) -- cycle;
\draw[line width=0.pt,color=zzttqq,fill=zzttqq,fill opacity=0.20000000298023224, smooth,samples=50,domain=0.0:4.0] plot(\x,{cos(((\x)-1)*180/pi)+1.5}) -- (4.,0.) -- (0.,0.) -- cycle;
\draw[line width=1.pt,smooth,samples=100,domain=0:4] plot(\x,{cos(((\x)-1)*180/pi)+1.5});
		\draw (1,0.6) node[anchor=north west] {$\Omega_-(t),\,\rho^-$};
		\draw (1.0,0.8) node[anchor=north west] {$\mbox{Fluid}-$};
		\draw (2.2,2.5) node[anchor=north west] {$\Omega_+(t),\,\rho^+$};
		\draw (2.2,2.7) node[anchor=north west] {$\mbox{Fluid}+$};
		\draw (4,0.66) node[anchor=north west] {$\leadsto (x,h(x,t))$};
	\end{tikzpicture}
	\caption{Interface $h(x,t).$}
	\label{fig1}
			\end{figure}
\noindent The equations (\ref{dl}), (\ref{ms}) and (\ref{in}) are known as the Incompressible Porous Media system (IPM) and they are supplemented by the  boundary conditions
		\begin{equation}\label{normalcontinuity}
			\begin{split}
					&(v^+-v^-)\cdot n=0\quad \mbox{in} \quad \partial \Omega_{\pm},
	 				\end{split}
		\end{equation}
		\[p^+=p^- \quad \mbox{in} \quad \partial \Omega_{\pm},\]
		where  $n$ denotes the unit normal vector to $\partial \Omega_{-},$ pointing out $\Omega_-$
		\[n=\frac{(-h'(x),1)}{\sqrt{1+h'(x)^2}}.\]
Notice that (\ref{normalcontinuity}) implies that $\nabla \cdot v=0$ in a weak sense. In addition, from (\ref{ms}), we can recover the kinematic boundary condition
		\[\partial_th= v^{+}(x,h(x,t))\cdot(-\partial_xh,1),\quad x\in\mathbb R.\]
		The mathematical formulation of this problem is the same as that  for two incompressible fluids in a Hele-Shaw cell, see \cite{MR0097227}. In \cite{DP}, C\'ordoba and Gancedo showed  that the Muskat problem  can  be reduced to an evolution equation for  the function $h$
		\begin{equation}\label{MEC}\frac{d}{dt}h(x,t)=\frac{\rho^--\rho^+}{2\pi}PV\int_{\mathbb R}\frac{\alpha\cdot (\partial_xh(x,t)-\partial_xh(x-\alpha,t))}{\alpha^2+(h(x,t)-h(x-\alpha,t))^2}d\alpha.\end{equation}
	The stability of (\ref{MEC}) strongly depends on the sign of the Rayleigh-Taylor function
\[\mathrm{RT}=-(\nabla p^-(\mathbf{x},t)-\nabla p^+(\mathbf{x},t))\cdot n,\quad \mathbf{x}\in\partial\Omega_{\pm},\] 
that in our case can be written as follows
	\[\mathrm{RT}=\frac{\rho^--\rho^+}{\sqrt{1+(\partial_xh)^2}}.\] 
When $\mathrm{RT}>0,$ this means the heaviest fluid is always below,  the problem is stable. 
In this regime, local existence of  solutions  is very well known as well as global existence for small initial data. However, if the heaviest fluid is above the situation is unstable and (\ref{MEC}) is ill-posed. We will review some of the literature dealing with these issues in section \ref{previous}.\\

\noindent In this paper we study the existence of non trivial solutions of (\ref{MEC}) of the form
\[h(x,t)=x^2+(\rho^--\rho^+)t+g(x,t),\]
where $g\in L^\infty((0,T): H^3(\mathbb R)).$ Thus, our  solutions grow quadratically at  infinity. As far as we know these are the solutions with the highest growth at  infinity that have been shown to exist. \\

\noindent Our main result reads as follows.
 \begin{theorem}\label{main_theorem} Let $s\geq 3$ and  $g_0\in H^s(\mathbb R).$ Then there exists a time $T_0=T(\norm{g_0}_{H^s})>0$ and a function $g\in L^\infty([0,T_0]:H^s(\mathbb R))\cap W^{1,\infty}([0,T_0]: H^{s-1}(\mathbb R))$ such that the function
 \[h(x,t)=x^2+(\rho^--\rho^+)t+g(x,t)\]
 solves  (\ref{MEC}) with $h(x,0)=x^2+g_0(x).$
    \end{theorem}
    \begin{remark} Let us remark that $T_0\to \infty$ when $\norm{g_0}_{H^s}\to 0$. 
\end{remark}
    The strategy  of the proof consists of two main steps:
    \begin{enumerate}
    	\item Firstly, we will check that $f(x,t)=x^2+(\rho^--\rho^+)t$ is actually a solution of (\ref{MEC}).
    	\item Secondly,  we will derive an equation for the function $g(x,t)=h(x,t)-x^2-(\rho^--\rho^+)t,$ (see equation (\ref{equation_g})). Then, we will prove the local existence of solutions for this equation  using energy estimates.
    \end{enumerate}
    
    Let us emphasize that the analysis of equation (\ref{equation_g}) for the evolution of $g(x,t)$ presents severals differences with respect to the analysis of (\ref{MEC}) in $H^s(\mathbb R)$ or $\dot{H}^k(\mathbb R)$ spaces, with $0\leq k\leq 2.$ Indeed, the quadratic growth at infinity introduces a degeneration of the kernels at  infinity that need to be understood. In addition, the explicit dependence  of $x$ leads to pseudodifferential operators, as opposed to the differential ones which occur in the classical Muskat problem. Notice that the kernel in (\ref{MEC}) is of the form $K(y,h(x),h(x-y))$  but in (\ref{equation_g}) we find two kernels of the form $K(x,y,g(x),g(x-y)).$ Finally, we find in (\ref{equation_g}) a new term which has no analogous in (\ref{MEC}).
    
    \begin{remark}
    In this paper we just deal with local existence of solutions. One could ask for global existence for small initial data,  as it is proven in the classical case. The reason why in our case  to prove global existence is more difficult than in the classical case, is that the Rayleigh-Taylor conditions breaks down at infinity and the parabolicity is lost. Same phenomenom causes that in \cref{main_theorem} the solutions $g\in L^\infty((0,T): H^s(\mathbb R))$ instead of $g\in C((0,T): H^s(\mathbb R)).$\end{remark}
    
The paper is organized as follows: In section \ref{previous} we will review some results concerning the existence of solutions for the Muskat problem. In section  \ref{notation} we will prove that $f(x,t)=x^2+(\rho^--\rho^+)t$ solves the Muskat equation and  we will derive  equation (\ref{equation_g}). Section \ref{energy} is dedicated to obtain the appropiate energy estimate for the function $g.$  All the necessary lemmas to prove the energy estimate are presented in section \ref{KernelBounds}. Finally,  section \ref{regularization} is devoted to the  study of the regularized system  in order to obtain existence of solutions.
    
 \subsection{Previous results}\label{previous}
	
		 The Muskat problem has been extensively  studied in the last decades. The first local existence result was established by Yi in \cite{MR2019452}, using  Newton's iteration method. Ambrose  in \cite{MR2128613},  using a  formulation for the tangent angle proved  local existence in $H^s(\mathbb R),s\geq 3.$ Caflish, Siegel and Howison proved in \cite{MR2070208} ill-posedness in the unstable case.	C\'ordoba and Gancedo in \cite{DP} proved  local existence in $H^s(\mathbb R),s\geq 3$, for the 2d case and $H^s(\mathbb R^2),s\geq 4$, for the 3d case, using energy methods.  Cheng,  Granero-Belich\'on  and Shkoller in \cite{MR3415681} established  global existence for a small initial data in  $H^2(\mathbb T)$ with different viscosities. Tofts in  \cite{MR3714494} by using a similar approach as Ambrose,  proved  global existence for small data in $H^s(\mathbb R),s\geq 6$ when surface tension is added.

	 Solutions of the Muskat equation (\ref{MEC}) satisfies a $L^\infty(\mathbb R)$ and $L^2(\mathbb R)$ maximum principles, see the work of   C\'ordoba and Gancedo   in \cite{MR2472040} and  the work of Constantin, C\'ordoba, Gancedo and Strain in \cite{MR2998834}.	In \cite{MR3661870}, Constantin, Gancedo, Shvydkoy and Vicol proved  local  existence for  initial data in  $W^{2,p}(\mathbb R)$ for $p\in(1,\infty].$ In the same paper, they proved  global existence when the slope $h'$  remains bounded. Later, in \cite{MR3869383}, Cameron  established  global existence in   $C^{1,\epsilon}(\mathbb R)$ using a criteria in terms of the product of the supremum and infimum of the slope of the initial data. For a small  data   Constantin, C\'ordoba, Gancedo and Strain in \cite{MR2998834}	proved   global existence for  initial data in  $H^3(\mathbb R)$  with a small derivative in the Wiener algebra $\mathcal A(\mathbb R).$ They also established the existence of global weak solutions for  $W^{1,\infty}(\mathbb R)$ initial data with the condition $\norm{h_0'}_{L^\infty}<1.$ In a subsequent paper \cite{MR3595492}, the same authors together with Rodr\'iguez-Piazza extended these results to the 3d case. 
	 	
	 		 	We observe that the Muskat equation (\ref{MEC}) is invariant  by the scale  $h_\lambda(x,t)=\lambda^{-1}h(\lambda x,\lambda t),$ \emph{i.e.} if  $h$ is a solution then $h_\lambda$ is also a solution. The spaces which are invariant under this scaling  are called critical spaces, for example both $\dot H^{3/2}(\mathbb R)$ and $\dot W^{1,\infty}(\mathbb R).$  In \cite{MR3861893}, Matioc proved  local existence for initial data $ H^s(\mathbb R)$ with $s\in (3/2,2).$ In a posterior work \cite{MR4390287},  Abels and Matioc established local existence for  initial data in $W^{s,p}(\mathbb R)$ with $p\in (1,\infty)$ and  $s\in (1+1/p,2),$ notice that $W^{1+1/p,p}(\mathbb R)$ is a critical space as well.
	 	
	 	In  \cite{MR4363243}, C\'ordoba and Lazar proved global existence  for initial data in  $\dot H^{3/2}(\mathbb R)\cap \dot H^{5/2}(\mathbb R)$ with a small assumption over $\dot H^{3/2}(\mathbb R),$  by using oscillatory integrals and a new formulation of the Muskat equation. Later, in order to get lower  regularity Alazard and Lazar established in \cite{MR4097324} local  existence for  initial data in $\dot H^1(\mathbb R)\cap \dot H^s(\mathbb R)$ with $s>3/2$. In a posterior work \cite{MR4242131}, Alazard and  Nguyen  proved  local existence for an initial data in the critical space $\dot W^{1,\infty}(\mathbb R)\cap H^{3/2}(\mathbb R),$ and the existence of global solutions for  small initial data. In  \cite{MR4313450} the same authors showed local and global  existence for   non-Lipchitz initial data. Recently, in  \cite{MR4541917}  they  proved local existence  for  initial data in $H^{3/2}(\mathbb R)$  and global existence in $H^{3/2}(\mathbb R)$ with a small condition over $\dot H^{3/2}(\mathbb R).$
	 	
	 In the 3d case, Gancedo and Lazar in   \cite{MR4487512}, proved  global existence for  the critical space $\dot H^2 (\mathbb R^2)\cap \dot W^{1,\infty}(\mathbb R^2).$  Alazard and  Nguyen proved in \cite{MR4387237}, using a different approach,  the same result of  \cite{MR4487512} and established the existence  of solutions for a non-Lipchitz initial data. Nguyen and Pausader proved  in \cite{MR4090462} the local existence for initial data in the subcritical space $H^s(\mathbb R^d),$ where $s>1+d/2.$   In \cite{MR4348695}  Nguyen established the global existence for small  initial data in the Besov space $\dot B_{\infty,1}^1(\mathbb R^d).$
	 
	  In \cite{MR3639321} Deng, Lei and Lin  constructed global weak solutions under  the assumptions that the initial interface is  monotonically decreasing with asymptotic behavior  at infinity \emph{i.e.} $f_0(x)\to a, x\to\infty$. Cameron in \cite{Cameron2020GlobalWF} proved  the existence of solutions  in the 3d case that are unbounded and has sublinear growth. In \cite{MR4385135}, Garc\'ia-Ju\'arez,  G\'omez-Serrano, Nguyen and Pausader proved the existence of  self-similar solutions. In  \cite{GarciaJuarez2023DesingularizationOS}, Garc\'ia-Ju\'arez,  G\'omez-Serrano,  Haziot and  Pausader proved  local existence when the initial interface  has multiple corners and linear growth at infinity.
	 
	 None of these results allow quadratic growth of the interface at infinity.
	 	
	 	In the unstable regime $\rho^+>\rho^-$ the Muskat equation is ill-posed,  see \cite{DP} and \cite{MR2070208}, then mixing solutions are used  to describe this scenario. In \cite{MR4309495}, Castro, C\'ordoba and Faraco studied this kind of solutions using convex integration and the theory of pseudodifferential operators  after the work of L. Sz\'ekelyhidi, see \cite{MR3014484}.  In the same direction see \cite{MR3921353}, \cite{MR4350274}, \cite{MR4193647} and \cite{MR3858828}. Mengual in \cite{MR4409883} studied the unstable case with different viscosities. Recently Castro, Faraco and Gebhard in \cite{castro2023entropy} studied maximal potential energy dissipation as a selection criterion for subsolutions. For others results concerning convex integration applied to IPM see \cite{MR2796131} and \cite{MR3479065}.
	 	
	 C\'ordoba, C\'ordoba and Gancedo proved,  in \cite{MR2753607},  local existence in $H^k(\mathbb T)$ with $k\geq 3,$ considering  different viscosities and positive RT. Later, the same authors treated in  \cite{MR3071395} the 3d case for a $H^4$ surface also in the case with different viscosities.  Gancedo, Garc\'ia-Ju\'arez, Patel and Strain in \cite{MR3899970} proved   global existence for small initial data in both 2d and 3d cases, also considering different viscosities. 
	 
	 For finite time singularities, in \cite{MR2993754}  Castro, C\'ordoba, Fefferman, Gancedo  and  L\'opez-Fern\'andez proved that there is an open subset of initial data in $H^4$ such that the Rayleigh-Taylor condition breaks down in finite time.  This means that the initial interface is  a graph $\mathrm{RT}>0,$ then in a finite time the interface is not a graph, $\mathrm{RT}<0.$ This is called \emph{turning singularity}.	In \cite{MR3048596} Castro, C\'ordoba, Fefferman and  Gancedo, proved that there exist solutions which lose the Rayleigh-Taylor condition and, after that, lose regularity in  finite time. These singular solutions have been extended over time as mixing solutions in \cite{MR4406890}.  C\'ordoba, 	G\'omez-Serrano and	 Zlato\v{s} proved  in \cite{MR3393318} the existence of solutions  that start in the unstable regime, then become stable and finally return to the unstable regime. The same authors in  \cite{MR3619874} established the existence of solutions  that start in the stable regime, then become unstable and finally return to the stable regime.

	\subsection{Notation and preliminaries }\label{notation}
	In this section, we derive the equation (\ref{equation_g})   and  introduce some notation that will be used throughout the paper.  The first step  is to prove that $f(x,t)=x^2+ct$ is an explicit solution of the Muskat equation. We have the following lemma.
		\begin{lemma}\label{par}
	The parabola $f(x,t)=x^2+ct$ solves the Muskat equation (\ref{MEC}) with $c=\rho^--\rho^+>0.$
	\end{lemma}
	\begin{proof} First we compute the differences
	\begin{equation*}
	\begin{split}
	&f(x)-f(x-\alpha)=\alpha(2x-\alpha),\\
	&\partial_xf(x)-\partial_xf(x-\alpha)=2\alpha,\\
	&\partial_tf=c.\\
	\end{split}
		\end{equation*}
	Then we substitute in the Muskat equation
	\begin{equation*}
		\begin{split}
			c&=\frac{\rho^--\rho^+}{2\pi}\int_{\mathbb R}\frac{2\alpha^2}{\alpha^2+\alpha^2(2x-\alpha)^2}d\alpha\\
			&=\frac{\rho^--\rho^+}{\pi}\int_{\mathbb R}\frac{1}{1+(2x-\alpha)^2}d\alpha\\
			&=\frac{\rho^--\rho^+}{\pi}\int_{\mathbb R}\frac{1}{1+u^2}du,\quad u=2x-\alpha\\
			&=\rho^--\rho^+.
			\end{split}			
	\end{equation*}
	\end{proof}
\noindent For renormalization we set $\rho^--\rho^+=2\pi.$ The function $f(x,t)=x^2+2\pi t$  solves the Muskat equation and is a parabola moving along the vertical axis as $t\to +\infty.$  We define  the difference $\delta_\alpha g$ and the slope $\Delta_\alpha g$ by 
\[\delta_\alpha g(x):=g(x)-g(x-\alpha)\quad\mbox{ and } \quad\Delta_\alpha g(x):=\frac{g(x)-g(x-\alpha)}{\alpha}.\] 
By substituting in the equation (\ref{MEC}) the function $h:=f+g,$ we see that $g$ satisfies 
    \begin{equation}\label{eq21}
        \begin{split}
    \frac{d}{d t}g(x)+2\pi&=PV\int_{\mathbb R} \frac{\partial_x\Delta_{\alpha}g(x)}{1+(\Delta_{\alpha} h(x))^2}\,d\alpha+PV\int_{\mathbb R}\frac{\partial_x\Delta_\alpha f(x)}{1+(\Delta_{\alpha}h(x))^2}\,d\alpha.\\
    \end{split}
    \end{equation}
By the definition of $f$ we have 
\[2\pi=\int_{\mathbb R}\frac{\partial_x\Delta_{\alpha}f}{1+(\Delta_{\alpha}f)^2}\,d\alpha.\]
Thus adding the term $2\pi$ to the right side of (\ref{eq21}), we obtain the following  equation
    \[\frac{d}{dt}g(x)=PV\int_{\mathbb R}\frac{\partial_x\Delta_\alpha g}{1+(\Delta_\alpha h)^2}\,d\alpha+PV\int_{\mathbb R} \Delta_\alpha g\frac{(-2)(\Delta_\alpha h+\Delta_\alpha f)}{(1+(\Delta_\alpha h)^2)(1+(\Delta_\alpha f)^2)}\,d\alpha.\]
 If we define the kernels
    \begin{equation}\label{kernelskg}  K(x,\alpha):=\frac{1}{1+(\Delta_\alpha h)^2},\quad  G(x,\alpha):=-2\frac{\Delta_\alpha h+\Delta_\alpha f}{(1+(\Delta_\alpha h)^2)(1+(\Delta_\alpha f)^2)}.\end{equation}
Then (\ref{eq21}) is equivalent to the equation
    \begin{equation}\label{equation_g}\frac{d}{dt}g(x,t)=PV\int_{\mathbb R}\partial_x\Delta_\alpha g(x) K(x,\alpha)\,d\alpha+PV\int_{\mathbb R}\Delta_\alpha g(x) G(x,\alpha)\,d\alpha .\end{equation}
    Thus, our task is proving local existence of (\ref{equation_g}) with an initial data $g(x,0)=g_0(x)\in H^s(\mathbb R).$ We observe that the kernels $K(x,\alpha)$ and $G(x,\alpha)$ explicitly depend on  the variable $x$ which represents a significant difference from the classic Muskat equation (\ref{MEC}). To control this type of terms, we deal with the Hilbert transforms of  rational functions.   We define  $Hf$ the Hilbert transform and $H_{\abs{\alpha}<1}f$ the truncated Hilbert transform  by
\[Hf(x)=\frac{1}{\pi}PV\int_{\mathbb R}\frac{f(x-y)}{y}dy,\quad H_{\abs{\alpha}<1}f(x)=\frac{1}{\pi}PV\int_{\abs{\alpha}<1}\frac{f(x-y)}{y}dy.\]
Additionally,  we will use the fact that the truncated Hilbert transform is a bounded operator from $L^2(\mathbb R)$ to $L^2(\mathbb R).$ We  also define the operator $\Lambda f:=H\partial_x f.$
\noindent Finally, we define the following norms
\[\norm{f}_{C^k}=\sup_{x\in \mathbb R}\max_{j\geq k}\big|\partial_x^kf(x)\big|,\]
\[\norm{f}_{L^\infty}=\esssup_{x\in \mathbb R}\abs{f(x)},\]
and denote by $D(x,\alpha)$ the difference of kernels
  \begin{equation}\label{diff_of_kernels}D(x,\alpha):= K(x,\alpha)-K(x,0).\end{equation}
\section{Energy estimates}\label{energy}
In this section we  obtain the energy estimate for the function $g.$ We present two main lemmas. \cref{mainlemma1} corresponds to the lower order  derivative terms, while  \cref{secondmain_lemma} deals with the highest derivative terms.  Let $s$ be an integer, we consider the energy of the function $g$   as the norm in the Sobolev space $H^s(\mathbb R),$ 
\[E(t)=\frac{1}{2}\norm{g}_{L^2}^2(t)+\frac{1}{2}\norm{\partial_x^sg}_{L^2}^2(t).\]
In order to prove local existence of solutions in $H^s(\mathbb R)$  we need an estimate for the evolution in time for the energy $E(t).$ In our case, the estimate will be in polynomial form, that is 
\[\frac{d}{dt}E(t)\leq c\,(E(t)+E(t)^2+\cdots+E(t)^\ell) \]
for a large  integer $\ell.$ This bound will suffice to prove that the energy of the solution is uniformly bounded in $H^s(\mathbb R)$ up to some time $T=T(\norm{g_0}_{H^s})>0.$  We start by controlling the evolution of the  $L^2(\mathbb R)$ norm of $g.$
\begin{lemma}\label{mainlemma1} Let $g\in H^s(\mathbb R)$  with $s\geq 3,$ then
      \begin{equation}
        \label{one_eq}
      \frac{1}{2}\frac{d}{dt}\norm{g}_{L^2}^2(t)\leq c\,\Big(\norm{g}_{H^s}^2+\dots+\norm{g}_{H^s}^5\Big).
    \end{equation}
    \end{lemma}
    \begin{proof}
    Taking the $L^2(\mathbb R)$ product of $g$ and $g_t,$ given by equation (\ref{equation_g}), we have 
      \begin{equation*}\setlength{\jot}{11pt}
        \begin{split}
            \frac{1}{2}\frac{d}{dt}\norm{g}_{L^2}^2(t)&=\int_{\mathbb R} g(x)\int_{\mathbb R}\partial_x\Delta_\alpha g(x)K(x,\alpha)\,d\alpha\,dx+ \int_{\mathbb R} g(x) \int_{\mathbb R} \Delta_\alpha g(x)G(x,\alpha)\,d\alpha\, dx\\
            &:=\mbox{I}+\mbox{II}.
        \end{split}
    \end{equation*}
   \noindent\emph{\color{blue}{Bound for \emph{I }}}: We use the definition of the slope $\Delta_\alpha g$ to split
   \begin{equation*}\setlength{\jot}{11pt}
                \begin{split}
                    \mbox{I}&=\int_{\mathbb R} g(x)\partial_xg(x)\int_{\mathbb R}\frac{1}{\alpha} K(x,\alpha )d\alpha dx-\int_{\mathbb R} g(x)\int_{\mathbb R}\frac{\partial_xg(x-\alpha)}{\alpha} K(x,\alpha) d\alpha dx\\
                    &:=A_1-A_2.
                \end{split}
            \end{equation*}
    Using  Cauchy-Schwarz inequality and then estimates  (\ref{first_estimates_order_one}) from \cref{aux_main_lemma1},   we find that
            \begin{equation}\label{b_A_1}
                \begin{split}   
                    \abs{A_1}&\leq\norm{g}_{L^2}\norm{\partial_xg}_{L^2}\bigg\lVert PV\int_{\mathbb R} \frac{1}{\alpha} K(\,\cdot\, ,\alpha)d\alpha \bigg\rVert_{L^\infty}\\
                    &\leq c\,(1+\norm{g}_{C^2})^3\norm{g}_{L^2}\norm{\partial_xg}_{L^2}.
                \end{split}
            \end{equation}
            To deal with the term  $A_2,$ we split the integral  in the \emph{in} and \emph{out} parts. For the \emph{in} part we have the following decomposition
    \begin{equation}\label{in_part_A_2}
        \begin{split}
            A_2^{in}&=\int_{\mathbb R}g(x) H_{\abs{\alpha}<1} \partial_xg(x)K(x,0)dx\\
            &+\int_{\mathbb R}g(x)\int_{\abs{\alpha}<1}\frac{\partial_xg(x-\alpha)}{\alpha}\big[K(x,\alpha)-K(x,0)\big] d\alpha dx,\\
        \end{split}
    \end{equation}
    where we use the truncated Hilbert transform $H_{\abs{\alpha}<1}\partial_xg,$ and  add and subtract the  kernel at zero 
  \[K(x,0)=\frac{1}{1+(\partial_xh(x))^2}.\]
  Then applying Cauchy-Schwarz inequality, we obtain that 
    \begin{equation}\label{l2_estimate_1}\bigg|\int_{\mathbb R}g(x)H_{\abs{\alpha}<1} \partial_xg(x) K(x,0)dx\bigg|\leq \norm{g}_{L^2}\norm{ \partial_xg}_{L^2}. \end{equation}
 By direct calculation, together with the Fundamental Theorem of Calculus we deduce an estimate for the difference (\ref{diff_of_kernels})
  \begin{equation*}
    \abs{D(x,\alpha)}=\big|K(x,\alpha)-K(x,0)\big|\leq c\,(1+\norm{\partial_x^2g}_{L^\infty})\,\abs{\alpha}.
    \end{equation*}
   Hence, for the second integral in (\ref{in_part_A_2}), we observe that applying Cauchy-Schwarz inequality we derive the following estimate
    \begin{equation}\label{l2_estimate_2}\setlength{\jot}{12pt}
        \begin{split}
    \bigg|\int_{\mathbb R}g(x)\int_{\abs{\alpha}<1}\frac{\partial_xg(x-\alpha)}{\alpha} D(x,\alpha) d\alpha dx\bigg|& \\
    &\hspace{-3cm}\leq\int_{\mathbb R}\abs{g(x)}\int_{\abs{\alpha}<1}\frac{\abs{\partial_xg(x-\alpha)}}{\abs{\alpha}}\abs{D(x,\alpha)}d\alpha dx\\
    &\hspace{-3cm}\leq c\,(1+\norm{g}_{C^2})\int_{\abs{\alpha}<1}\int_{\mathbb R}\abs{g(x)}\abs{\partial_xg(x-\alpha)}dx d\alpha\\
        &\hspace{-3cm}\leq c\,(1+\norm{g}_{C^2})\norm{g}_{L^2}\norm{\partial_xg}_{L^2}.
        \end{split}
    \end{equation}
   For the \emph{out} part, we apply  Cauchy-Schwarz  inequality  respect to $x$ 
    \begin{equation*}\setlength{\jot}{12pt}
        \begin{split}
            \abs{A_2^{out}}=\bigg|\int_\mathbb{ R}g(x)\int_{\abs{\alpha}>1}\frac{\partial_xg(x-\alpha)}{\alpha}K(x,\alpha)d\alpha dx\bigg|&\\
            &\hspace{-4cm}\leq \norm{g}_{L^2}\bigg(\int_{\mathbb R}\bigg|\int_{\abs{\alpha}>1}\frac{\partial_xg(x- \alpha)}{\alpha} K(x,\alpha)d\alpha\bigg|^2dx\bigg)^{1/2}.\\
              \end{split}
    \end{equation*}
    Now  we use Cauchy-Schwarz  inequality  respect to $\alpha$
    \begin{equation*}
      \begin{split}
             \abs{A_2^{out}}&\hspace{0cm}\leq \norm{g}_{L^2}\bigg(\int_{\mathbb R}\bigg(\int_{\abs{\alpha}>1 }\partial_xg(x-\alpha)^2d\alpha\bigg) \bigg(\int_{\abs{\alpha}>1}\frac{1}{\alpha^2}K(x,\alpha)^2d\alpha\bigg)dx\bigg)^{1/2}\\
            &\hspace{0cm}\leq \norm{g}_{L^2}\norm{\partial_xg}_{L^2}\bigg(\int_{\abs{\alpha}>1}\frac{1}{\alpha^2}\int_{\mathbb R}K(x,\alpha)^2dxd\alpha\bigg)^{1/2}.
\end{split}
\end{equation*}
   The    estimate (\ref{bound_for_k_2r}) in \cref{main_prop_9} states that 
   \[\int_{\mathbb{R}} K(x,\alpha)^2 dx \leq c\,(1+\norm{\partial_xg}_{L^\infty}).\]
Therefore putting together the estimates (\ref{b_A_1}), (\ref{l2_estimate_1}), (\ref{l2_estimate_2})  and the inequalities for the \emph{out} part  we obtain the following bound
\[\abs{\mbox{ I }} \leq c\,(1+\norm{g}_{C^2})^3\norm{g}_{L^2}\norm{\partial_xg}_{L^2}.\]

\noindent\emph{\color{blue}{Bound for \emph{II }}}: For the \emph{in} part,  using  the Fundamental Theorem of Calculus  we have  the following formula for the slope
        \begin{equation}\label{formula_slope_g}\Delta_\alpha g=\int_0^1\partial_xg(x+(s-1)\alpha)ds,\end{equation}
        hence we obtain that 
        \begin{equation*}
            \begin{split}
                \mbox{II}^{in}&=\int_0^1\int_{\mathbb R}g(x)\int_{\abs{\alpha}<1}\partial_xg(x+(s-1)\alpha)G(x,\alpha)d\alpha dxds.
                \end{split}
        \end{equation*}
From the definition (\ref{kernelskg}) we deduce that 
    \[\abs{G(x,\alpha)}=\bigg|-2\frac{\Delta_\alpha h+\Delta_\alpha f}{(1+(\Delta_\alpha h)^2)(1+(\Delta_\alpha f)^2)}\bigg|\leq 2.\]
    Now applying the Cauchy-Schwarz inequality yields
    \begin{equation}\label{in_part_second_kernel}\abs{\mbox{II}^{in}}\leq 2\,\norm{g}_{L^2}\norm{\partial_xg}_{L^2}.\end{equation}
For the \emph{out} part, expanding $\Delta_\alpha g$  we split the integral  in two terms 
    \begin{equation*}    \setlength{\jot}{11pt}
        \begin{split}
            \mbox{II}^{out}&=\int_{\mathbb R}g(x)^2\int_{\abs{\alpha}>1}\frac{1}{\alpha}G(x,\alpha)d\alpha dx-\int_{\mathbb R}g(x)\int_{\abs{\alpha}>1}\frac{g(x-\alpha)}{\alpha}G(x,\alpha)d\alpha dx\\
            &:=A_3+A_4.
        \end{split}
    \end{equation*}
    In $A_3,$ we control the inner integral using the estimates (\ref{first_estimates_order_3})  from \cref{aux_main_prop3}, hence   
        \begin{equation*}\setlength{\jot}{10pt}
            \begin{split}
                \abs{A_3}&\leq \bigg|\int_{\mathbb R} g(x)^2\,\int_{\abs{\alpha}>1} \frac{1}{\alpha} G(x,\alpha)d\alpha \,dx\bigg|\\
                &\leq \norm{g}_{L^2}^2\bigg\lVert PV\int_{\abs{\alpha}>1}\frac{1}{\alpha } G(\,\cdot\,,\alpha)d\alpha\bigg\rVert_{L^\infty}\leq c\,(1+\norm{g}_{L^\infty})^2\norm{g}_{L^2}^2.
            \end{split}
        \end{equation*}       
Now for $A_4,$ we follow  the same technique used in $A_2^{out}.$ First, applying Cauchy-Schwarz inequality first respect to $x$  and  then respect to $\alpha,$ we deduce that 
    \begin{equation*}
        \begin{split}
            \abs{A_4}&\leq\norm{g}_{L^2}^2\bigg(\int_{\abs{\alpha}>1}\frac{1}{\alpha^2}\int_{\mathbb R}G(x,\alpha)^2dxd\alpha\bigg)^{1/2}.\\
        \end{split}
    \end{equation*}
    The estimate  (\ref{bound_for_g_2r}) in \cref{main_prop_10},  says that 
    \[\int_{\mathbb R}G(x,\alpha)^2dx\leq c\,(1+\norm{\partial_xg}_{L^\infty})^3.\]
 Then the  last inequality and (\ref{in_part_second_kernel}) conclude the proof
    \[\abs{\mbox{ II }}\leq  c\,(1+\norm{g}_{C^1})^2\norm{g}_{H^1}^2. \]\end{proof}
Now we move to the second part of the energy, which involves the  derivative of order $s$ of $g.$  We will prove  the following lemma.
\begin{lemma}\label{secondmain_lemma}
          Let $g\in H^s(\mathbb R)$ with $s\geq 3,$ then
        \begin{equation}\label{main_ineq_2}\frac{1}{2}\frac{d}{dt}\norm{\partial_x^3g}_{L^2}^2(t)\leq c\,\Big(\norm{g}_{H^s}^2+ \dots +\norm{g}_{H^s}^5\Big).\end{equation}
\end{lemma}
\begin{proof}
 We  take $s=3$ and  compute  $\partial_x^3g_t$  from the equation (\ref{equation_g}). We have two terms
    \begin{equation*}\setlength{\jot}{3ex}
        \begin{split}
    \frac{1}{2}\frac{d}{dt}\norm{\partial_x^3g}_{L^2}^2&=\int_\mathbb R\partial_x^3g(x)\partial_x^3\int_{\mathbb R} \partial_x\Delta_\alpha g(x) K(x,\alpha)d\alpha\,dx+\int_\mathbb R\partial_x^3g(x)\partial_x^3\int_\mathbb R \Delta_\alpha g (x)G(x,\alpha)d\alpha\,dx\\
    &:=\mbox{III}+\mbox{IV}.
        \end{split}
    \end{equation*}
    We use the Leibniz product rule to get the  next decomposition 
    \begin{equation*}
        \begin{split}
      \mbox{III}
            &:=J_1+3J_2+3J_3+J_4.
        \end{split}
    \end{equation*}
       The goal  is obtain a polynomial bound for each $J_i.$ We start by getting a bound for $J_1.$ \\

\noindent\emph{\color{blue}{Bound for $J_1$}}: This term is  the most singular  because four derivatives acting on $g$.  We expand   $\partial_x^4\Delta_\alpha g$ and   add and subtract the kernel at zero $K(x,0),$ we have 
  \begin{equation*}\setlength{\jot}{3ex}
    \begin{split}
      J_1&=\int_{\mathbb R}\partial_x^3g(x)\partial_x^4g(x)\int_{\mathbb R}\frac{1}{\alpha}K(x,\alpha)d\alpha dx-\int_{\mathbb R}K(x,0)\partial_x^3g(x)\int_{\mathbb R}\frac{\partial_x^4g(x-\alpha)}{\alpha}d\alpha dx\\
      &\hspace{0cm}-\int_{\mathbb R}\partial_x^3g(x-\alpha)\int_{\mathbb R}\frac{\partial_x^4g(x-\alpha)}{\alpha}\big[K(x,\alpha)-K(x,0)\big]d\alpha dx.
    \end{split}
   \end{equation*}
    Recall that the  kernel at zero is given by 
    \[K(x,0)=\frac{1}{1+(\partial_xh(x))^2}.\]
    Using 
      \[\partial_x^3g(x)\partial_x^4g(x)=\frac{1}{2}\partial_x[\partial_x^3g(x)]^2\] 
    and   integration by parts, we obtain that
      \[\int_{\mathbb R}\partial_x^3g(x)\partial_x^4g(x)\int_{\mathbb R}\frac{1}{\alpha}K(x,\alpha)d\alpha dx=-\frac{1}{2}\int_{\mathbb R}[\partial_x^3g(x)]^2\partial_x\int_{\mathbb R}\frac{1}{\alpha}K(x,\alpha)d\alpha dx.\]
      The fact that $H\partial_x=\Lambda$ implies 
  \begin{equation}\setlength{\jot}{3ex}\label{term_J11}
    \begin{split}
      J_1&=-\frac{1}{2}\int_{\mathbb R}\big[\partial_x^3g(x)\big]^2\int_{\mathbb R}\frac{1}{\alpha}\partial_x K(x,\alpha )d\alpha dx-\int_{\mathbb R}K(x,0)\partial_x^3g(x)\Lambda \partial_x^3g(x)dx\\
      &-\int_{\mathbb R}\partial_x^3g(x-\alpha)\int_{\mathbb R}\frac{\partial_x^4g(x-\alpha)}{\alpha}D(x,\alpha)d\alpha dx,
    \end{split}
  \end{equation}
  where  $D(x,\alpha)$ is  the difference  $K(x,\alpha)-K(x,0).$ Now we use  the C\'ordoba-C\'ordoba  pointwise inequality, see \cite{MR2032097}, then we obtain that
  \[\partial_x^3g(x)\Lambda \partial_x^3g(x)\geq \frac{1}{2}\Lambda [\partial_x^3g(x)]^2.\]
 Due to $K(x,0)>0,$ we get that 
  \begin{equation*}\setlength{\jot}{10pt}
    \begin{split}
      J_1&\leq -\frac{1}{2}\int_{\mathbb R}\big[\partial_x^3g(x)\big]^2\int_{\mathbb R}\frac{1}{\alpha}\partial_x K(x,\alpha )d\alpha dx-\frac{1}{2}\int_{\mathbb R}\Lambda K(x,0)[\partial_x^3g(x)]^2dx\\
      &-\int_{\mathbb R}\partial_x^3g(x-\alpha)\int_{\mathbb R}\frac{\partial_x^4g(x-\alpha)}{\alpha}D(x,\alpha)d\alpha dx.
    \end{split}
  \end{equation*}
  Using   the inequalities (\ref{second_estimates_order_one}) and (\ref{lambdak_0}) in \cref{aux_main_prop4} and \cref{lambdak0}, we conclude that the first two terms above are bounded. That is,  the $L^\infty(\mathbb R)$ norms  of the inner integral  and the  operator  $\Lambda K(x,0)$ are bounded
  \begin{equation*}
      \begin{split}
        &\bigg\lVert PV\int_{\mathbb R}\frac{1}{\alpha}\partial_x K(\,\cdot\,,\alpha)d\alpha\bigg\rVert_{L^\infty}\leq c\,(1+\norm{g}_{C^{2,\delta}})^2\\
      \end{split}\end{equation*}
and
  \begin{equation*}
      \begin{split}
        &\norm{\Lambda K(x,0)}_{L^\infty}\leq c\,(1+\norm{g}_{C^{2,\delta}}).
      \end{split}\end{equation*}
  In the last two inequalities we take $\delta=1/2$ because $H^3(\mathbb R)\hookrightarrow C^{2,1/2}(\mathbb R).$ \\
  
\noindent  It remains to get the bound for the  term with the difference $D(x,\alpha)$ in (\ref{term_J11}).  First, we note that by the chain rule  $\partial_\alpha\partial_x^3g(x-\alpha)=-\partial_x^4g(x-\alpha).$  From here, integration by parts yields 
  \begin{equation}\label{singular_term}
          \begin{split}
            -\int_{\mathbb R} \partial_x^3g(x)\int_{\mathbb R}\frac{\partial_x^4g(x-\alpha)}{\alpha}D(x,\alpha)d\alpha dx&=-\int_{\mathbb R}\partial_x^3g(x)\int_{\mathbb R}\partial_x^3g(x-\alpha)\partial_\alpha \bigg(\frac{D(x,\alpha)}{\alpha}\bigg)d\alpha dx.
        \end{split}
    \end{equation}
    Denote  \[\Phi(x,\alpha):= \partial_\alpha [D(x,\alpha)/\alpha]. \] 
We split the  integral  (\ref{singular_term}) into the \emph{in} and \emph{out} parts. For the \emph{in} part,  we observe that 
    \begin{equation*}\setlength{\jot}{10pt}
        \begin{split}
            \bigg|  \int_{\mathbb R}\partial_x^3g(x)\int_{\abs{\alpha}<1}\partial_x^3g(x-\alpha)\Phi(x,\alpha)d\alpha dx\bigg|&\\
            &\hspace{-5cm}\leq \int_{\mathbb R}\int_{\abs{\alpha}<1}\abs{\partial_x^3g(x)}\abs{\partial_x^3g(x-\alpha)}\big|\Phi(x,\alpha )\big|d\alpha dx\\
            &\hspace{-5cm}\leq \frac{1}{2}\int_{\mathbb R}\int_{\abs{\alpha }<1}\Big[\abs{\partial_x^3g(x)}^2+\abs{\partial_x^3g(x-\alpha )}^2\Big]\big|\Phi(x,\alpha)\big|d\alpha dx\\
            &\hspace{-5cm}\leq \frac{1}{2}\int_{\mathbb R}\abs{\partial_x^3g(x)}^2\int_{\abs{\alpha }<1}\big|\Phi(x,\alpha)\big|d\alpha dx+\frac{1}{2}\int_{\mathbb R}\int_{\abs{\alpha }<1}\abs{\partial_x^3g(x-\alpha)}^2\big|\Phi(x,\alpha)\big|d\alpha dx,
        \end{split}
    \end{equation*}
    where we have used Young's inequality 
    \[ab\leq \frac{1}{2}\abs{a}^2+\frac{1}{2}\abs{b}^2.\]
    Using estimate (\ref{second_order_k}) from \cref{aux_main_prop5}, we get 
    \begin{equation}\label{control_over_thediff}\Big|\Phi(x,\alpha)\chi_{\abs{\alpha}<1}(\alpha)\Big|\leq c\,(1+\norm{g}_{C^{2,\delta}})^2\abs{\alpha}^{\delta-1},\end{equation}
    which is  integrable near to the origin. For the second integral we change variables $\beta=\alpha$ and $y=x-\alpha$ to get 
    \[\int_{\mathbb R}\int_{\abs{\alpha }<1}\abs{\partial_x^3g(x-\alpha)}^2\big|\Phi(x,\alpha)\big|d\alpha dx=\int_{\mathbb R}\abs{\partial_y^3g(y)}^2\int_{\abs{\beta}<1}\abs{\Phi(y+\beta,\beta)}d\beta dy\]
and  we have the same control (\ref{control_over_thediff}) over $\abs{\Phi(y+\beta,\beta)}.$   Hence the \emph{in} part  is bounded
		\begin{equation}\label{bound_in_part_kernel_3}
			\bigg|-\int_{\mathbb R}\partial_x^3g(x-\alpha)\int_{\abs{\alpha}<1}\frac{\partial_x^4g(x-\alpha)}					{\alpha}D(x,\alpha)d\alpha dx\bigg|\leq c\,(1+\norm{g}_{C^{2,\delta}})^2\norm{\partial_x^3g}_{L^2}^2.
		\end{equation}
Now we focus on  the \emph{out} part. First, we note that 
\[\Big|\Phi(x,\alpha)\chi_{\abs{\alpha}>1}(\alpha)\Big|\leq \frac{\abs{D(x,\alpha)}}{\alpha^2}+\frac{\abs{\partial_\alpha K(x,\alpha)}}{\abs{\alpha}}.\]
Then we split in two parts. For the first part we have 
       \begin{equation}\label{J2_out_2}\setlength{\jot}{12pt}
        \begin{split}
            \bigg|\int_{\mathbb R}\partial_x^3g(x)\int_{\abs{\alpha}>1} \partial_x^3g(x-\alpha)\Phi(x,\alpha) d\alpha dx\bigg|\\
            &\hspace{-4cm}\leq \int_{\mathbb R}\abs{\partial_x^3g(x)}\int_{\abs{\alpha}>1}\abs{\partial_x^3g(x-\alpha)}\bigg|\frac{\partial_{\alpha}K(x,\alpha)}{\alpha}\bigg|d\alpha dx \\
            &\hspace{-4cm}+ \int_{\mathbb R}\abs{\partial_x^3g(x)}\int_{\abs{\alpha}>1}\abs{\partial_x^3g(x-\alpha)}\bigg|\frac{D(x,\alpha)}{\alpha^2}\bigg|d\alpha dx.
        \end{split}
    \end{equation}
    The second line of (\ref{J2_out_2}) can be bounded by applying  Cauchy-Schwarz inequality, first with respect to $x,$  and  then with respect to $\alpha.$  Then we obtain that 
    \begin{equation*}\setlength{\jot}{12pt}
        \begin{split}
        &\int_{\mathbb R}\abs{\partial_x^3g(x)}\int_{\abs{\alpha}>1}\abs{\partial_x^3g(x-\alpha)}\bigg|\frac{\partial_{\alpha}K(x,\alpha)}{\alpha}\bigg|d\alpha dx \\
        &\hspace{2cm}\leq \norm{\partial_x^3g}_{L^2}\bigg(\int_{\mathbb R}\norm{\partial_x^3g}_{L^2}^2\bigg(\int_{\abs{\alpha}>1}\frac{\partial_\alpha K(x,\alpha)^2}{\alpha^2}d\alpha\bigg)dx\bigg)^{1/2}\\
        &\hspace{2cm}\leq  \norm{\partial_x^3g}_{L^2}^2\bigg(\int_{\abs{\alpha}>1}\frac{1}{\alpha^2}\int_{\mathbb R}\partial_\alpha K(x,\alpha)^2 dxd\alpha\bigg)^{1/2}\\
        &\hspace{2cm}\leq c\,\norm{\partial_x^3g}_{L^2}^2(1+\norm{\partial_xg}_{L^\infty})^3.
        \end{split}
    \end{equation*}
In the last inequality we applied  the estimates   (\ref{bound_for_k_2rr}) of \cref{main_prop_11}.
 For the second term in the right hand side of (\ref{J2_out_2}),  we apply Cauchy-Schwarz inequality with respect to $x$ and then use  Minkowski's integral inequality.  Also we  note that the difference satisfies \[\abs{D(x,\alpha)}\leq 2.\] Thus
\begin{equation*}\setlength{\jot}{11pt}
        \begin{split}
        &\int_{\mathbb R}\abs{\partial_x^3g(x)}\int_{\abs{\alpha}>1}\abs{\partial_x^3g(x-\alpha)}\bigg|\frac{D(x,\alpha)}{\alpha^2}\bigg|d\alpha dx\\
        &\hspace{2cm}\leq \norm{\partial_x^3g}_{L^2}\bigg(\int_{\mathbb R}\bigg|\int_{\abs{\alpha}>1}\frac{\partial_x^3g(x-\alpha)}{\alpha^2} D(x,\alpha)d\alpha \bigg|^2dx\bigg)^{1/2}\\
        &\hspace{2cm}\leq \norm{\partial_x^3g}_{L^2}\int_{\abs{\alpha}>1}\bigg(\int_{\mathbb R}\bigg[\frac{\partial_x^3g(x-\alpha)}{\alpha^2}\bigg]^2dx\bigg)^{1/2}d\alpha\\
        &\hspace{2cm}\leq c\,\norm{\partial_x^3g}_{L^2}^2.
        \end{split}
    \end{equation*}
Therefore, by  joining the estimates for the \emph{out} part and (\ref{bound_in_part_kernel_3}), we deduce that 
    \begin{equation}\label{J1_bound}\abs{J_1}\leq c\,(1+\norm{g}_{C^{2,1/2}})^3\norm{g}_{H^3}^2.\end{equation}
    \emph{\color{blue}{Bound for $J_2$}}: The second term $J_2$ is similar to $J_1,$ expanding $\partial_x^3\Delta g$ we have 
        \begin{equation*}\setlength{\jot}{12pt}
            \begin{split}
                J_2&=\int_{\mathbb R}\partial_x^3g(x)^2\int_{\mathbb R}\frac{1}{\alpha}\partial_x K(x,\alpha)d\alpha dx-\int_{\mathbb R}\partial_x^3g(x)\int_{\mathbb R}\frac{\partial_x^3g(x-\alpha)}{\alpha}\partial_x K(x,\alpha) d\alpha dx.\\
            \end{split}
        \end{equation*}
For the last integral in $J_2$ the  change of variable $x-\alpha=y$ leads to
    \begin{equation*}\setlength{\jot}{12pt}
        \begin{split}
        \int_{\mathbb R}\partial_x^3g(x)\int_{\mathbb R}\frac{\partial_x^3g(x-\alpha)}{\alpha}\partial_x K(x,\alpha) d\alpha dx& =\int_{\mathbb R}\int_{\mathbb R}\partial_x^3g(x)\partial_y^3g(y)\frac{\partial_x K(x,x-y)}{x-y} dy dx.\end{split}
    \end{equation*}  
Define $\zeta(x,y)$ as the kernel
  \begin{equation*}
    \begin{split}
      \zeta(x,y)&:=\frac{\partial_xK(x,x-y)}{x-y}=\frac{-2}{x-y}\frac{\bigg(\displaystyle{\frac{h(x)-h(y)}{x-y}\bigg)\bigg(\frac{\partial_xh(x)-\partial_xh(y)}{x-y}}\bigg)}{\bigg(1+\bigg(\displaystyle{\frac{h(x)-h(y)}{x-y}}\bigg)^2\bigg)^2}.
  \end{split}
  \end{equation*}
We observe that $\zeta(x,y)=-\zeta(y,x),$ then change of variables $x=y$ and $y=x$ implies that 
    \begin{equation*}
        \begin{split}
        \int_{\mathbb R}\int_{\mathbb R}\partial_x^3g(x)\partial_y^3g(y)\zeta(x,y) dx dy&\\
        &\hspace{-2cm}=\int_{\mathbb R}\int_{\mathbb R}\partial_y^3g(y)\partial_x^3g(x)\zeta(y,x) dy dx\\
        &\hspace{-2cm}=-\int_{\mathbb R}\int_{\mathbb R}\partial_x^3g(x)\partial_y^3g(y)\zeta(x,y) dx dy.\\
        \end{split}
        \end{equation*}
Therefore the second integral  in $J_2$ is zero and the first integral  has the same  bound (\ref{J1_bound}) of $J_1$. That is
    \begin{equation}\label{bound_j2}
        \abs{J_2}\leq c\,(1+\norm{g}_{C^{2,1/2}})^3\norm{g}_{H^3}^2.
    \end{equation}
\noindent\emph{\color{blue}{Bound for $J_3$}}: We split in the \emph{in} and \emph{out} parts
	\begin{equation*}
		\begin{split}		
			J_3&=\int_{\mathbb R}\partial_x^3g(x)\int_{\abs{\alpha}<1}\partial_x^2\Delta_{\alpha}									g\partial_x^2K(x,\alpha)d\alpha dx+\int_{\mathbb R}\partial_x^3g(x)\int_{\abs{\alpha}>1}							\partial_x^2\Delta_{\alpha}g\partial_x^2K(x,\alpha)d\alpha dx\\
				&:=J_3^{in}+J_3^{out}.
		\end{split}
	\end{equation*}
Using the estimate   (\ref{second_order_K_2}) we have the following bound
\[\abs{\partial_x^2K(x,\alpha)}\leq c\,\big(1+\norm{\partial_x^2g}_{L^\infty}\big)^2+c\,\norm{g}_{C^{2,\delta}}\cdot\abs{\alpha}^{\delta-1},\]
where $\delta\in (0,1).$ We use the Fundamental Theorem of Calculus to obtain the following formula
		\begin{equation}\label{formula_taylor_third}
				\partial_x^2\Delta_\alpha g=\int_{0}^1\partial_x^3g(x+(s-1)\alpha)ds.
		\end{equation}
Then using (\ref{formula_taylor_third}) and Cauchy-Schwarz inequality with respect to $x$  we obtain that 
	\begin{equation}\setlength{\jot}{10pt}\label{j_3_in_bound}
        \begin{split}
            \abs{J_3^{in}}&\leq \int_0^1\int_{\mathbb R}\abs{\partial_x^3g(x)}\int_{\abs{\alpha}<1}								\abs{\partial_x^3g(x+(s-1)\alpha)}\abs{\partial_x^2K(x,\alpha)}d\alpha dx  ds\\
            &\leq c\,(1+\norm{g}_{C^{2,\delta}})^2\int_0^1\int_{\abs{\alpha}<1}(1+\abs{\alpha}^{\delta-1})						\int_{\mathbb R}\abs{\partial_x^3g(x)}\abs{\partial_x^3g(x+(s-1)\alpha)}dxd\alpha ds\\
            &\leq c\,(1+\norm{g}_{C^{2,\delta}})^2\norm{\partial_x^3g}_{L^2}^2.
        \end{split}
    \end{equation}
where as in the previous term $J_2$, we take $\delta=1/2.$  Now,  for the \emph{out} part expanding $\partial_x^2\Delta_\alpha g$ we have
    \begin{equation}\setlength{\jot}{12pt}\label{J3_out}
        \begin{split}
            J_3^{out}&=\int_{\mathbb R}\partial_x^3g(x)\partial_x^2g(x)\int_{\abs{\alpha}>1}\frac{1}{\alpha}\partial_x^2K(x,\alpha)d\alpha dx-\int_{\mathbb R}\partial_x^3g(x)\int_{\abs{\alpha}>1}\frac{\partial_x^2g(x-\alpha)}{\alpha}\partial_x^2K(x,\alpha)d\alpha dx.\\
        \end{split}
    \end{equation}
For the first integral in the right hand side of (\ref{J3_out})  we apply Cauchy-Schwarz inequality and use the estimate (\ref{second_estimates_order_one_far_zero}) of \cref{aux_main_prop7}. We deduce that 
    \begin{equation}\label{j3_out_bound}\setlength{\jot}{10pt}
        \begin{split}
            \bigg|\int_{\mathbb R}\partial_x^3g(x)\partial_x^2g(x)\int_{\abs{\alpha}>1}\frac{1}{\alpha}\partial_x^2K(x,\alpha)d\alpha dx\bigg|&\\
            &\hspace{-2cm}\leq \norm{\partial_x^3g}_{L^2}\norm{\partial_x^2g}_{L^2}\bigg\lVert \int_{\abs{\alpha}>1}\frac{1}{\alpha}\partial_x^2K(x,\alpha)d\alpha \bigg\rVert_{L^\infty}\\
            &\hspace{-2cm}\leq c\,(1+\norm{g}_{C^2})^2\norm{\partial_x^3g}_{L^2}\norm{\partial_x^2g}_{L^2}.
            \end{split}
        \end{equation}
For the second integral in the right hand side of (\ref{J3_out}), we observe
\begin{equation}\label{second_K}
    \begin{split}
        \partial_x^2K(x,\alpha)&=[\partial_x\Delta_{\alpha}h]^2B_1(x,\alpha)+ \partial_x^2\Delta_{\alpha} g B_2(x,\alpha),
    \end{split}
\end{equation}
where
\[B_1(x,\alpha):=-2K(x,\alpha)^2+8(\Delta_\alpha h)^2K(x,\alpha)^3\quad \mbox{and}\quad B_2(x,\alpha):=-2\Delta_{\alpha }hK(x,\alpha)^2.\]
Then expanding the sum in (\ref{second_K}) we obtain that
\[\int_{\mathbb R}\partial_x^3g(x)\int_{\abs{\alpha}>1}\frac{\partial_x^2g(x-\alpha)}{\alpha}\partial_x^2K(x,\alpha)d\alpha 				dx:=J_{3,1}^{out}+J_{3,2}^{out},\]
where
\[ J_{3,1}^{out}=\int_{\mathbb R}\partial_x^3g(x)\int_{\abs{\alpha}>1}\frac{\partial_x^2g(x-\alpha)}{\alpha}					[\partial_x\Delta_\alpha h]^2B_1(x,\alpha)d\alpha dx\]
and
\[ J_{3,2}^{out}=\int_{\mathbb R}\partial_x^3g(x)\int_{\abs{\alpha}>1}\frac{\partial_x^2g(x-\alpha)}{\alpha}					[\partial_x^2\Delta_\alpha g]B_2(x,\alpha)d\alpha dx.\]
 We notice that 
	  \[\abs{B_1(x,\alpha)}\leq 10\, K(x,\alpha),\]
which is square integrable with respect to $x,$ by  \cref{main_prop_9}. Using the Fundamental Theorem of Calculus  we deduce that 
	\begin{equation}\label{bound_for_partial_h}
		\abs{\partial_x\Delta_\alpha h}\leq 2\,(1+\norm{\partial_x^2g}_{L^\infty}).
	\end{equation}
Hence applying Cauchy-Schwarz inequality,  first respect to $x,$ then respect to $\alpha.$ We find that 
  \begin{equation*}\setlength{\jot}{12pt}
    \begin{split}
\abs{J_{3,1}^{out}}
      &\hspace{0cm}\leq \norm{\partial_x^3g}_{L^2}\bigg(\int_{\mathbb R}\bigg|\int_{ \abs{\alpha}>1} 					\frac{\partial_x^2g(x-\alpha)}{\alpha}[\partial_x\Delta_\alpha h]^2B_1(x,\alpha)d\alpha\bigg|						^2dx\bigg)^{1/2}\\
      &\hspace{0cm}\leq c\,(1+\norm{\partial_x^2g}_{L^\infty})^2\norm{\partial_x^3g}_{L^2}\norm{\partial_x^2g}			_{L^2}\bigg(\int_{\abs{\alpha}>1} \frac{1}{\alpha^2}\int_{\mathbb R} B_1(x,										\alpha)^2dxd\alpha\bigg)^{1/2}\\
		& \hspace{0cm}\leq c\,(1+\norm{g}_{C^2})^3\norm{\partial_x^2g}_{L^2}\norm{\partial_x^3g}_{L^2}.
    \end{split}
  \end{equation*}
The second  term $J_{3,2}^{out}$  has a similar bound.  In that case we use the following bounds
 		 \[\abs{B_2(x,\alpha)}\leq 2K(x,\alpha)\quad\mbox{ and }\quad\abs{\partial_x^2\Delta_{\alpha}g}\leq 2\norm{\partial_x^2g}_{L^\infty}\abs{\alpha}^{-1}.\]
 Therefore by joining the estimates (\ref{j_3_in_bound}) and (\ref{j3_out_bound}) and the inequalities for $J_{3,1}^{out}$ and $J_{3,2}^{out}$ we conclude that 
    \begin{equation}\label{J3_bound}
        \abs{J_3}\leq c\,(1+\norm{g}_{C^{2,1/2}})^3\norm{g}_{H^3}^2.
    \end{equation}
\noindent\emph{\color{blue}{Bound for $J_4$}}: We notice that

	\begin{equation}\label{third_derivative}
		\partial_x^3K(x,\alpha)=\partial_x\Delta_{\alpha}hB_3(x,\alpha)+															\partial_x^2\Delta_{\alpha}gB_4(x,\alpha)+\partial_x^3\Delta_{\alpha}gB_5(x,\alpha),			
	\end{equation}
where
	\begin{equation}\label{terms_bs}\setlength{\jot}{10pt}
		\begin{split}
			&B_3(x,\alpha):=\Big[24(\Delta_\alpha h)K(x,\alpha)^3-48(\Delta_\alpha h)^3 K(x,\alpha)^4\Big](\partial_x \Delta_\alpha h)^2,\\
			&B_4(x,\alpha):=3\big[-2K(x,\alpha)^3+8(\Delta_{\alpha}h)^2K(x,\alpha)^4\big]					(\partial_x\Delta_{\alpha}h),\\
 &B_5(x,\alpha):=-2\Delta_{\alpha }h K(x,\alpha)^2.
		\end{split}
	\end{equation}
Then expanding the sum in (\ref{third_derivative}) we decompose $J_4:= J_{4,1}+J_{4,2}+J_{4,3}$ with

	\begin{equation*}
		\begin{split}
			J_{4,1}&=\int_{\mathbb R}\partial_x^3g(x)\int_{\mathbb R}(\partial_x\Delta_{\alpha}g) 
				\partial_x\Delta_\alpha h B_3(x,\alpha) d\alpha dx,\\
			J_{4,2}&=\int_{\mathbb R}\partial_x^3g(x)\int_{\mathbb R}(\partial_x\Delta_{\alpha}g)
					\partial_x^2\Delta_\alpha g B_4(x,\alpha) d\alpha dx,\\
			J_{4,3}&=\int_{\mathbb R}\partial_x^3g(x)\int_{\mathbb R}(\partial_x\Delta_{\alpha}g )
					\partial_x^3\Delta_\alpha g B_5(x,\alpha) d\alpha dx.
		\end{split}
 	\end{equation*}
Using  the Fundamental Theorem of Calculus we have the following formula
	\begin{equation}\label{taylor_bound_order_1}\partial_x\Delta_{\alpha}g=\int_{0}^{1}\partial_x^2g(x+(s-1)\alpha)ds.\end{equation}
Notice 
	 \begin{equation}\label{bound_b3}
		\begin{split}
			&\abs{B_3(x,\alpha)}\leq c\,(1+\norm{\partial_x^2g}_{L^\infty})^2,
		\end{split}	
	\end{equation}
then the estimate (\ref{bound_for_partial_h})  together with   the Cauchy-Schwarz inequality yields to the following bound
  \begin{equation}\setlength{\jot}{10pt}\label{bound_j41in}
    \begin{split}
    \abs{ J_{4,1}^{in}}& \leq c\,\int_0^1\int_{\abs{\alpha}<1}\int_{\mathbb R}\abs{\partial_x^3g(x)}\abs{\partial_x^2g(x+(s-1)\alpha)}\abs{\partial_x\Delta_\alpha  h}\abs{B_3(x,\alpha)}dxd\alpha ds\\
      &\hspace{0cm}\leq c\,(1+\norm{\partial_x^2g}_{L^\infty})^3\norm{\partial_x^3g}_{L^2}\norm{\partial_x^2 g}_{L^2}.
    \end{split}
  \end{equation}
 For the \emph{out} part, we expand $\partial_x\Delta_\alpha g$ and  take $J_{4,1}^{out}:=L_1+L_2,$ where
	\begin{equation*}
		\begin{split}		
			L_1&=\int_{\mathbb R}\partial_x^3g(x)\partial_xg(x)\int_{\abs{\alpha}>1}\frac{1}{\alpha}											\partial_x\Delta_\alpha h B_3(x,\alpha)d\alpha dx,\\
			L_2&=-\int_{\mathbb R}\partial_x^3g(x)\int_{\abs{\alpha}>1}\frac{\partial_xg(x-\alpha)}{\alpha}									\partial_x\Delta_{\alpha}h B_3(x,\alpha)d\alpha dx.
		\end{split}
	\end{equation*}
Now  we expand the sum $\partial_x\Delta_{\alpha}h=2+\partial_x\Delta_\alpha g$ and  decompose further   $L_1:=S_1+S_2$ for  
	\begin{equation*}\setlength{\jot}{10pt}
		\begin{split}		
			S_1&=2\int_{\mathbb R}\partial_x^3g(x)\partial_xg(x)\,\bigg\{\,\int_{\abs{\alpha}>1}\frac{1}{\alpha}B_3(x,						\alpha)d\alpha\bigg\} \,dx=2\int_{\mathbb R}\partial_x^3g(x)\partial_xg(x)\eta(x)dx,\\
			S_2&=\int_{\mathbb R}\partial_x^3g(x)\partial_xg(x)\int_{\abs{\alpha}>1}											\frac{\partial_x\Delta_{\alpha }g}{\alpha} B_3(x,\alpha)d\alpha dx.\\			
		\end{split}
	\end{equation*}
In order to get a bound of $S_1$ we need an estimate for $\eta(x)$. First we  observe from (\ref{terms_bs}) that 
	\[B_3(x,\alpha)= \gamma(x,\alpha)\Big[4+4\partial_x\Delta_{\alpha}g+(\partial_x\Delta_\alpha g)^2\Big],\]
where 
	\begin{equation}
		\label{def_gamma}	
		\gamma(x,\alpha):=24(\Delta_\alpha h)K(x,\alpha)^3-48(\Delta_\alpha h)^3 K(x,\alpha)^4.
	\end{equation}
We expand $B_3(x,\alpha )$ and decompose $\eta(x):=4\eta_1(x)+\eta_2(x)$ for
		\begin{equation}
		\begin{split}\label{control_psi}
			&\eta_1(x)=PV\int_{\abs{\alpha}>1}\frac{1}{\alpha}\gamma(x,\alpha)d\alpha,\\
			&\eta_{2}(x)=PV\int_{\abs{\alpha}>1}\frac{1}{\alpha}\gamma(x,\alpha)(4\partial_x\Delta_\alpha g+								(\partial_x\Delta_\alpha g)^2)d\alpha.\\
		\end{split}
	\end{equation}
	We derive  the bound for $\eta_2$ from the estimate  $\abs{\gamma(x,\alpha)}<c$ and the following inequality
	\[\abs{4\partial_x\Delta_{\alpha}g+(\partial_x\Delta_{\alpha}g)^2}\leq 8\frac{\norm{\partial_xg}_{L^\infty}}{\abs{\alpha}}+4\frac{\norm{\partial_xg}_{L^\infty}^2}{\abs{\alpha}^2}.\]
Hence 
	\[\abs{\eta_2(x)}\leq c\,\,(\norm{\partial_xg}_{L^\infty}+\norm{\partial_x g}_{L^\infty}^2).\]
While for $\eta_1,$ the estimate (\ref{bound_outside}) in \cref{bounds_outside_kn} states that 
	\begin{equation*}
		\abs{\eta_1(x)}\leq c\,(1+\norm{g}_{L^\infty})^3.
	\end{equation*}
By joining the inequalities for $\eta_{1}$ and $\eta_{2}$  we obtain the next estimate 
	\begin{equation*}
\norm{\eta}_{L^\infty}\leq c\,(1+\norm{g}_{C^1})^3.\end{equation*}
Thus, applying the Cauchy-Schwarz inequality we complete the estimate for $S_1.$ We have that 
	\begin{equation*}
		\begin{split}
			\abs{S_{1}}&\leq 4\int_{\mathbb R}\abs{\partial_x^3g(x)}\abs{\partial_xg(x)}\abs{\eta(x)} dx\leq c\,(1+\norm{g}_{C^1})^3\norm{\partial_xg}_{L^2}							\norm{\partial_x^3g}				_{L^2}.
		\end{split}
	\end{equation*}
The inner integral in $S_2$ is easily bounded by using  the estimate (\ref{bound_b3}), we conclude that 
	\begin{equation*}\setlength{\jot}{12pt}
		\begin{split}
			\bigg|\int_{\abs{\alpha}>1}\frac{\partial_x\Delta_{\alpha}g}{\alpha}B_3(x,					\alpha)d\alpha \bigg|&\leq c\,(1+\norm{\partial_x^2g}												_{L^\infty})^3\int_{\abs{\alpha}>1}\abs{\alpha}^{-2}d\alpha\\
			&\leq c\,(1+	\norm{\partial_x^2g}_{L^\infty})^3.
		\end{split}	
	\end{equation*}
Then, similarly to $S_1,$ we apply the Cauchy-Schwarz inequality and use the previous bound to obtain that
	\[\abs{S_2}\leq c\,(1+\norm{g}_{C^2})^3\norm{\partial_xg}_{L^2}\norm{\partial_x^3g}_{L^2}.\]
The last inequality completes the estimate for $L_1.$  Now we move to $L_2,$ analogously we take $L_2:=S_3+S_4,$ where
	\begin{equation*}\setlength{\jot}{10pt}
		\begin{split}		
			S_3&=\int_{\mathbb R}\partial_x^3g(x)\int_{\abs{\alpha}>1}\frac{2\partial_xg(x-\alpha)}{\alpha} B_3(x,\alpha)d\alpha dx,\\
			S_4&=\int_{\mathbb R}\partial_x^3g(x)\int_{\abs{\alpha}>1}											\frac{\partial_xg(x-\alpha)}{\alpha}\partial_x\Delta_{\alpha }g B_3(x,\alpha)d\alpha dx.\\			
		\end{split}
	\end{equation*}
Notice  that $\abs{\gamma(x,\alpha)}<c\,K(x,\alpha).$ Using the bound  (\ref{bound_b3})  we  derive the following estimate
		\[\abs{B_3(x,\alpha)}\leq 4cK(x,\alpha)+4c\frac{\norm{\partial_xg}_{L^\infty}}{\abs{\alpha}}										+c\frac{\norm{\partial_xg}_{L^\infty}^2}{\abs{\alpha}^2}.\]
The last bound  together with  the Cauchy-Schwarz inequality with respect to $x$ and  Minkowski's integral inequality  leads to 
		\begin{equation*}\setlength{\jot}{11pt}
			\begin{split}
				\abs{S_3}&\leq c\,\norm{\partial_x^3g}_{L^2}\norm{\partial_x g}_{L^2}\bigg(\int_{\abs{\alpha}>1}\frac{1}							{\alpha^2}\int_{\mathbb R}K(x,\alpha)^2dxd\alpha\bigg)^{1/2}\\
&				+c\, \norm{\partial_x g}_{L^\infty}\norm{\partial_x^3g}_{L^2}\norm{\partial_x g}_{L^2}\int_{\abs{\alpha}>1}						\frac{1}{\abs{\alpha}^2}d\alpha+c\,\norm{\partial_xg}_{L^\infty}^2\norm{\partial_x^3g}_{L^2}\norm{\partial_x g}_{L^2}\int_{\abs{\alpha}>1}\frac{1}{\abs{\alpha}^3}d\alpha.
			\end{split}
		\end{equation*}
Then the estimate (\ref{bound_for_k_2r}) in \cref{main_prop_9} implies that 
\[\abs{S_3}\leq c\,(1+\norm{g}_{C^1})^2\norm{\partial_xg}_{L^2}\norm{\partial_x^3g}_{L^2}.\]
For $S_4,$ we use  the  bound (\ref{bound_b3}) to  obtain that
\[\abs{\partial_x\Delta_{\alpha}g B_3(x,\alpha)}\leq c\,(1+\norm{\partial_x^2g}_{L^\infty})^2\norm{\partial_xg}_{L^2}\abs{\alpha}^{-1}.\] 
Now, we apply  the Cauchy-Schwarz and  Minkoswki's integral inequalities. Then we conclude the following bound
	\[\abs{S_4}\leq c\,(1+\norm{g}_{C^2})^3\norm{\partial_xg}_{L^2}\norm{\partial_x^3g}_{L^2}. \]
The last inequality completes the estimate for the \emph{out} part $L_2^{out}.$  Hence estimate (\ref{bound_j41in}) and bounds for $L_1$ and $L_2$  implies that 
		\begin{equation}\label{bound_j41}
			\abs{J_{4,1}}\leq c\,(1+\norm{g}_{C^2})^3\norm{g}_{H^3}^2.
		\end{equation}
For $J_{4,2}^{in}$ using the formula (\ref{formula_taylor_third}) we find  that 
	\[J_{4,2}^{in}=\int_0^1\int_{\mathbb R}\partial_x^3g(x)\int_{\abs{\alpha}<1}\partial_x^3g(x+(s-1)\alpha)				\partial_x\Delta_{\alpha}g B_4(x,\alpha)d\alpha dxds.\]
From the identities  (\ref{terms_bs})  and  the estimate   (\ref{bound_for_partial_h}) we deduce that 
		\begin{equation}\label{bound_b4}
			\abs{B_4(x,\alpha)}\leq c\,(1+\norm{\partial_x^2g}_{L^\infty}).
		\end{equation}
The  bound (\ref{bound_b4}) together with  formulas (\ref{taylor_bound_order_1}), (\ref{formula_taylor_third}) and by applying the Cauchy-Schwarz inequality allow us conclude that 
	\begin{equation}\label{bound_j42in}
			\abs{J_{4,2}^{in}}\leq c\,(1+\norm{\partial_x^2g}_{L^\infty})									\norm{\partial_x^2g}_{L^\infty}\norm{\partial_x^3g}_{L^2}^2. 
	\end{equation}
For the \emph{out} part, expanding $\partial_x^2 \Delta_\alpha g$ we split $J_{4,2}^{out}:=L_3+L_4$ for 
	\begin{equation*}\setlength{\jot}{12pt}
		\begin{split}
			L_3&=\int_{\mathbb R}\partial_x^3g(x)\partial_x^2g(x)\int_{\abs{\alpha}>1}					\frac{1}{\alpha}\partial_x\Delta_{\alpha}g B_4(x,\alpha) d\alpha dx,\\
			L_4&=-\int_{\mathbb R}\partial_x^3g(x)\int_{\abs{\alpha}>1}										\frac{\partial_x^2g(x-\alpha)}{\alpha}\partial_x\Delta_{\alpha}g B_4(x,						\alpha)d\alpha dx.
		\end{split}
	\end{equation*}
Recall that $\abs{\partial_x\Delta_\alpha g}\leq 2\norm{\partial_xg}_{L^\infty}\abs{\alpha}^{-1}.$
We use  the estimate (\ref{bound_b4}) and  analogously to $S_2$ we obtain that 
		\[\abs{L_3}\leq c\,(1+\norm{\partial_x^2g}_{L^\infty})\norm{\partial_xg}_{L^\infty}\norm{\partial_x^2g}_{L^2}\norm{\partial_x^3g}_{L^2}.\]
The estimate for $L_4$ is easy, because is similar to $S_4,$ then  we have  the following bound
		\[\abs{L_4}\leq c\,(1+\norm{\partial_x^2g}_{L^\infty})\norm{\partial_xg}_{L^\infty}\norm{\partial_x^2g}_{L^2}\norm{\partial_x^3g}_{L^2}.\]
Using the estimates  for $L_3,L_4$ and the bound (\ref{bound_j42in}) we obtain that
	\begin{equation}\label{bound_j42}
		\abs{J_{4,2}}\leq c\,(1+\norm{g}_{C^2})^2\norm{g}_{H^3}^2.
	\end{equation}
Finally for $J_{4,3},$ expanding  $\partial_x^3\Delta_{\alpha}g,$ we split $J_{4,3}:=L_5+L_6$ for 
  \begin{equation*}\setlength{\jot}{12pt}
    \begin{split}
        L_5&=\int_{\mathbb R}[\partial_x^3g(x)]^2\int_{\mathbb R}\frac{1}{\alpha}[\partial_x\Delta_{\alpha}			g] B_5(x,\alpha)d\alpha dx,\\
        L_6&=\int_{\mathbb R}\partial_x^3g(x)\int_{\mathbb R}\frac{\partial_x^3g(x-\alpha)}{\alpha}\partial_x\Delta_{\alpha}g B_5(x,\alpha)d\alpha dx.\\
    \end{split}
    \end{equation*}
Using the indentities (\ref{terms_bs}) we have the next bound
	\begin{equation*}
		\begin{split}
			\abs{\partial_x \Delta_{\alpha}gB_5(x,\alpha)}\leq 2\norm{\partial_x^2g}_{L^\infty}.
		\end{split}
	\end{equation*}
The last bound, the  estimates (\ref{second_estimates_order_one}) from \cref{aux_main_prop4} together with the   Cauchy-Schwarz  and Minkowski's integral inequalities leads to 
   \begin{equation}\label{bound_L5}
		\abs{L_5}\leq c\, (1+\norm{g}_{C^{2,1/2}})^2\norm{\partial_x^3g}_{L^2}^2.				\end{equation}
For $L_6:=L_6^{in}+L_6^{out},$ the \emph{out} part is easy controlled by using Cauchy-Schwarz inequality and  Minkowski's integral inequality
	\begin{equation}\label{bound_L6_out}
		\begin{split}
			\abs{L_6^{out}} &\leq 2\norm{\partial_xg}_{L^\infty}\int_{\mathbb R}\abs{\partial_{x}^3g(x)}							\int_{\abs{\alpha}>1}\frac{\abs{\partial_x^3g(x-\alpha)}}{\abs{\alpha}^2}d\alpha dx\\
			&\leq c\,\norm{\partial_xg}_{L^\infty}\norm{\partial_x^3g}_{L^2}^2.
		\end{split}
	\end{equation}
For the \emph{in} part, we add and subtract $\partial_x^2g(x)$ and $B_5(x,0)$ in order to get  $L_6^{in}:=N_{1}+N_{2}+N_{3}$ for
	\begin{equation*}
		\begin{split}		
			N_{1}&=\int_{\mathbb R}\partial_x^3g(x)\int_{\abs{\alpha}<1}\frac{\partial_x^3g(x-\alpha)}{\alpha}				\big(\partial_x\Delta_{\alpha}g-\partial_x^2g(x)\big)B_5(x,\alpha)d\alpha dx,\\
			N_{2}&=\int_{\mathbb R}\partial_x^3g(x)\partial_x^2g(x)\int_{\abs{\alpha}<1}\frac{\partial_x^3g(x-\alpha)}{\alpha}\big(B_5(x,\alpha)-B_5(x,0)\big)d\alpha dx,\\
			N_{3}&=\int_{\mathbb R}\partial_x^3g(x)\partial_x^2g(x)B_5(x,0)H_{\abs{\alpha}<1}\partial_x^3g(x)dx.
		\end{split}
	\end{equation*}
For $N_1$  we use 
	\[\abs{\partial_x\Delta_\alpha g-\partial_x^2g(x)}\leq \abs{g}_{C^{2,\delta}}\abs{\alpha}^{\delta},\]
for $\delta\in (0,1).$  Due to $\abs{B_5(x,\alpha)}\leq c,$ and applying  Cauchy-Schwarz inequality with respect to $x$ followed by the Minkoswki integral inequality yields to
	\begin{equation*}	
		\begin{split}
			\abs{N_1}&\leq c\,\norm{g}_{C^{2,\delta}}\norm{\partial_x^3g}_{L^2}							\bigg(\int_{\mathbb R}\bigg|\int_{\abs{\alpha}<1}\abs{\partial_x^3g(x-\alpha)}\abs{\alpha}^{\delta-1}d\alpha\bigg|^2dx\bigg)^{1/2}\\
			&\leq c\,\norm{g}_{C^{2,\delta}}\norm{\partial_x^3g}_{L^2}\int_{\abs{\alpha}					<1}\abs{\alpha}^{\delta-1}\bigg(\int_{\mathbb R}\abs{\partial_x^3g(x-\alpha)}^2dx\bigg)^{1/2}d\alpha\leq  c\,\norm{g}_{C^{2,\delta}}\norm{\partial_x^3g}_{L^2}^2.
		\end{split}
	\end{equation*}
We estimate the second term $N_2$ using the next inequality
	\[\abs{B_5(x,\alpha)-B_5(x,0)}\leq c\,\abs{\Delta_{\alpha}h-\partial_xh(x)}\leq c\, 		(1+\norm{\partial_x^2g}_{L^\infty}).\]
Then, by applying the Cauchy-Schwarz and Minkoswki integral we obtain that
		\begin{equation*}
		\abs{N_2}\leq c\,(1+\norm{g}_{C^{2}})^2\norm{\partial_x^3g}_{L^2}^2.			\end{equation*}
Finally using   $\abs{B_5(x,0)}\leq c$ and the fact that the  truncated Hilbert transform is bounded operator in $L^2(\mathbb R),$ we obtain that 
		\[\abs{N_3}\leq c\,\norm{\partial_x^2g}_{L^\infty}\norm{\partial_x^3g}_{L^2}			^2.\] 
The estimates for $N_1,N_2,N_3,$ the bound (\ref{bound_L6_out}) and the estimate  (\ref{bound_L5})  allow us  conclude that 
	\[\abs{J_{4,3}}\leq c\,(1+\norm{g}_{C^{2,1/2}})^3\norm{g}_{H^3}^2.\]
By joining the estimates  (\ref{bound_j41}), (\ref{bound_j42}) and the last one, we complete the estimate for $J_4.$ We obtain that 
	\begin{equation}\label{bound_j4}\abs{J_4}\leq c\,(1+\norm{g}_{C^{2,1/2}})^3\norm{g}_{H^3}^2.\end{equation}
We conclude from inequalities (\ref{J1_bound}), (\ref{bound_j2}), (\ref{J3_bound}) and (\ref{bound_j4}) that 
	\begin{equation}\label{bound_iii}
		\abs{\mbox{ III }}\leq c\,(1+\norm{g}_{C^{2,1/2}})^{3}\norm{g}_{H^3}^2.
	\end{equation}

\noindent\emph{\color{blue}{Bound for  \emph{IV}}} : Notice
    \[\mbox{IV}=\int_{\mathbb R}\partial_x^3g(x)\partial_x^3\int_{\mathbb R}\Delta_\alpha g \,G(x,\alpha)d\alpha dx=2\int_{\mathbb R}\partial_x^3g(x)\int_{\mathbb R}\partial_x^3K(x,\alpha)d\alpha dx.\]
Using (\ref{third_derivative}) we decompose $\mbox{IV}:=J_5+J_6+J_7$  for 
	\begin{equation*}
		\begin{split}
			J_5&=2\int_{\mathbb R}\partial_x^3g(x)\int_{\mathbb R}\partial_x\Delta_{\alpha}h B_3(x,								\alpha)d\alpha dx,\\
			J_6&=2\int_{\mathbb R}\partial_x^3g(x)\int_{\mathbb R}\partial_x^2\Delta_{\alpha}g B_4(x,								\alpha)d\alpha dx,\\
			J_7&=2\int_{\mathbb R}\partial_x^3g(x)\int_{\mathbb R}\partial_x^3\Delta_{\alpha}g B_5(x,								\alpha)d\alpha dx.\\
		\end{split}
	\end{equation*}
From the identities (\ref{terms_bs}) we see that   $B_5(x,\alpha)=-2\Delta_{\alpha }h K(x,\alpha)^2.$ Then we estimate 
$J_7$ is the same way to $J_2.$ Thus
	\begin{equation}\label{bound_j7}
		\abs{J_7}\leq c\,(1+\norm{g}_{C^{2,1/2}})^2\norm{g}_{H^3}^2.
	\end{equation}
Using the formula (\ref{formula_taylor_third}) and the inequality (\ref{bound_b4}) together with the Cauchy-Schwarz inequality  we find that 
 \begin{equation*}\setlength{\jot}{12pt}
    \begin{split}
    \abs{J_6^{in}}&\leq \int_{0}^1\int_{\abs{\alpha}<1}\int_{\mathbb R}\abs{\partial_x^3g(x)}\abs{\partial_x^3g(x+(s-1)\alpha)}\abs{ B_4(x,\alpha)}d\alpha dxds\\
    &\leq c\,(1+\norm{\partial_x^2g}_{L^\infty})\norm{\partial_x^3g}_{L^2}^2.
    \end{split}
  \end{equation*}
For the \emph{out} part expanding $\partial_x^2\Delta_{\alpha }g$  we decompose $J_{6}^{out}:=L_5+L_6$ for 
    \begin{equation*}
        \begin{split}
 L_5&= \int_{\mathbb R}\partial_x^3g(x)\partial_x^2g(x)\,\bigg\{\,\int_{\abs{\alpha}>1}\frac{1}{\alpha} B_4(x,\alpha)d\alpha\bigg\}\, dx,\\
L_6&= - \int_{\mathbb R}\partial_x^3g(x)\int_{\abs{\alpha}>1}\frac{\partial_x^2g(x-\alpha)}{\alpha}B_4(x,\alpha)d\alpha dx.
        \end{split}
    \end{equation*}
We denote the inner integral by 
	\[\nu(x):=PV\int_{\abs{\alpha}>1}\frac{1}{\alpha}B_4(x,\alpha)d\alpha.\]
Now in order to get a bound for $\nu,$ we proceed in similar way to $\eta$ in (\ref{control_psi}). Using  estimates (\ref{bound_outside_b4}) in  \cref{bounds_outside_b4} we obtain that 
	\[\norm{\nu}_{L^\infty}\leq c\,(1+\norm{g}_{C^1})^2.\]
Then Cauchy-Schwarz inequality yields to 
	\[\abs{L_5}\leq c\,(1+\norm{g}_{C^1})^3\norm{\partial_x^2g}_{L^2}\norm{\partial_x^3g}_{L^2}.\]	
From identities (\ref{terms_bs}) we have the next bound
	\[\abs{B_4(x,\alpha)}\leq c\,(1+\norm{\partial_x^2g}_{L^\infty})K(x,\alpha).\,\]
For $L_6,$ we apply the Cauchy-Schwarz inequality, first with respect to $x$ and then respect to $\alpha,$ also we  use  \cref{main_prop_11}, then we deduce that 
	\begin{equation*}
		\begin{split}	
			\abs{L_6}&\leq c\,(1+\norm{\partial_xg}_{L^\infty})\norm{\partial_x^3 g}_{L^2}
			\norm{\partial_x^2g}_{L^2}\bigg(\int_{\abs{\alpha}>1}\frac{1}{\abs{\alpha}^2}\int_{\mathbb R}
			 K(x,\alpha)^2dx d\alpha\bigg)^{1/2}\\
			&\leq c\,(1+\norm{\partial_xg}_{L^\infty})^2\norm{\partial_x^3 g}_{L^2}\norm{\partial_x^2g}_{L^2}.
		\end{split}
	\end{equation*}
The bounds for $L_5$ and $L_6$ complete the estimate for the \emph{out} part. We conclude
	\begin{equation}\label{bound_j6}
		\abs{J_6}\leq c\,(1+\norm{g}_{C^2})^3\norm{g}_{L^2}\norm{\partial_x^3g}_{L^2}.
	\end{equation}
We now estimate $J_5,$ first we note 
	\[\partial_x\Delta_{\alpha}h\,B_3(x,\alpha)=\gamma(x,\alpha)\big[8+12\partial_x\Delta_\alpha g+6(\partial_x\Delta_\alpha g)^2+(\partial_x\Delta_\alpha g)^3\big],\]
where $\gamma(x,\alpha)$ is given by (\ref{def_gamma}). Then we decompose $J_5:=J_{5,1}+J_{5,2}+J_{5,3}+J_{5,4}$ for 
	\begin{equation*}
		\begin{split}
			J_{5,1}&=8\int_{\mathbb R}\partial_x^3g(x)\int_{\mathbb R}\gamma(x,\alpha)d\alpha 					dx,\\
			J_{5,2}&=12\int_{\mathbb R}\partial_x^3g(x)\int_{\mathbb R}\gamma(x,								\alpha)\partial_x\Delta_\alpha gd\alpha dx,\\
			J_{5,3}&=6\int_{\mathbb R}\partial_x^3g(x)\int_{\mathbb R}\gamma(x,									\alpha)(\partial_x\Delta_\alpha g)^2gd\alpha dx,\\
			J_{5,4}&=\int_{\mathbb R}\partial_x^3g(x)\int_{\mathbb R}\gamma(x,									\alpha)(\partial_x\Delta_\alpha g)^3d\alpha dx.\\
		\end{split}
	\end{equation*}
The bound $\abs{\gamma(x,\alpha)}<c,$ the formula (\ref{taylor_bound_order_1})  and the Cauchy-Schwarz inequality  imply that 
	\[\abs{J_{5,2}^{in}}+\abs{J_{5,3}^{in}}+\abs{J_{5,4}^{in}}\leq c\,(1+\norm{\partial_x^2g}							_{L^\infty})^2\norm{\partial_x^2g}_{L^2}\norm{\partial_x^3g}_{L^2}.\]
The \emph{out} part  $J_{5,3}^{out}$ is  easily bounded. By expanding $\partial_x\Delta_\alpha g$ we have $J_{5,3}^{out}:=S_5+S_6$ for
	\begin{equation*}
		\begin{split}
			S_5&=6\int_{\mathbb R}\partial_x^3g(x)\partial_xg(x)\int_{\abs{\alpha}>1}\frac{1}{\alpha}\gamma(x,						\alpha)(\partial_x\Delta_\alpha g)d\alpha dx,\\
			S_6&=-6\int_{\mathbb R}\partial_x^3g(x)\int_{\abs{\alpha}>1}\frac{\partial_xg(x-\alpha)}{\alpha}						\gamma(x,\alpha)(\partial_x\Delta_\alpha g)d\alpha dx.\\
		\end{split}
	\end{equation*}
The bound $\abs{\partial_x\Delta_{\alpha}g}\leq 2\norm{\partial_xg}_{L^\infty}\abs{\alpha}^{-1}$ and the Cauchy-Schwarz inequality yields 
\[\abs{S_5}\leq c\,\norm{\partial_xg}_{L^\infty}\norm{\partial_xg}_{L^2}\norm{\partial_x^3g}_{L^2}.\]
While for $S_6$ we use Minkowski's integral inequality and we obtain that  
\[\abs{S_6}\leq c\,\norm{\partial_xg}_{L^\infty}\norm{\partial_xg}_{L^2}\norm{\partial_x^3g}_{L^2},\]
and this completes the estimate for $J_{5,3}^{out}.$  For $J_{5,4}^{out}$ we proceed similarly to $J_{5,3}^{out}$, then we conclude that 
	\[\abs{J_{5,4}^{out}}\leq c\,\norm{\partial_xg}_{L^\infty}^2\norm{\partial_xg}_{L^2}\norm{\partial_x^3g}_{L^2}.\]
For $J_{5,2}^{out},$ expanding $\partial_x\Delta_{\alpha} g$ we get
	\begin{equation*}
		\begin{split}
			J_{5,2}^{out}&=12\int_{\mathbb R}\partial_x^3g(x)\partial_xg(x)\eta_1(x)dx-12\int_{\mathbb R}\partial_x^3g(x)\int_{\abs{\alpha}>1}\frac{\partial_xg(x-\alpha)}{\alpha}\gamma(x,\alpha)d\alpha dx,
		\end{split}
	\end{equation*}
where $\eta_1(x)$ is given as in (\ref{control_psi}). We use  $\abs{\gamma(x,\alpha)}\leq c\,K(x,\alpha)$ and  the estimate (\ref{bound_outside}) for $\norm{\eta_1}_{L^\infty}$ followed by applying the Cauchy-Schwarz inequality  we obtain that 
\[\abs{J_{5,2}^{out}}\leq c\,(1+\norm{g}_{L^\infty})^3\norm{\partial_xg}_{L^2}\norm{\partial_x^3g}_{L^2}.\]
Finally, from the definition (\ref{def_gamma}) we split  $\gamma(x,\alpha)$ and  decompose  $J_{5,1}:=S_7+S_8$ for
	\begin{equation*}
		\begin{split}
			S_7&=24\int_{\mathbb R}\partial_x^3g(x)\int_{\mathbb R}\Delta_{\alpha} hK(x,\alpha)^3d\alpha dx,\\
			S_8&=-48\int_{\mathbb R}\partial_x^3g(x)\int_{\mathbb R}(\Delta_{\alpha}h)^3K(x,\alpha)^4d\alpha dx.\\
		\end{split}
	\end{equation*}
We define
\[\gamma_f(x,\alpha):=\frac{\Delta_{\alpha}f}{(1+(\Delta_{\alpha}f)^2)^3}\]
and observe
\[\int_{\mathbb R}\gamma_f(x,\alpha)d\alpha=0.\]
Expanding $\Delta_{\alpha}h$ and adding $\gamma_f,$ we decompose
	\begin{equation*}
		\begin{split}		
			S_7&=24\int_{\mathbb R}\partial_x^3g(x)\big(\Delta_{\alpha} fK(x,\alpha)^3-\gamma_f(x,\alpha)\big)d\alpha 							dx+24\int_{\mathbb R}\partial_x^3g(x)\int_{\mathbb R}\Delta_{\alpha}g K(x,\alpha)^3d\alpha dx\\
			&:=S_{7,1}+S_{7,2}.
		\end{split}
	\end{equation*}
Using the formula (\ref{formula_slope_g}) and the  Cauchy-Schwarz inequality we obtain that 
	\begin{equation*}
		\begin{split}		
			\abs{S_{7,2}^{in}}&\leq c\, \int_0^1\int_{\mathbb R}\int_{\abs{\alpha}<1}\abs{\partial_x^3g(x)}\abs{\partial_xg(x+(s-1)				\alpha)} K(x,\alpha)^3d\alpha dxds\\
			&\leq c\,\norm{\partial_xg}_{L^2}\norm{\partial_x^3g}_{L^2}.
		\end{split}
	\end{equation*}
Anagolously to $J_6^{out},$ we expand $\Delta_{\alpha}g$ and decompose
\[S_{7,2}^{out}=\int_{\mathbb R}\partial_x^3g(x)g(x)\int_{\abs{\alpha}>1}\frac{1}{\alpha}K(x,\alpha)^3d\alpha dx-\int_{\mathbb R}\partial_x^3g(x)\int_{\abs{\alpha}>1}\frac{g(x-\alpha)}{\alpha}K(x,\alpha)^3d\alpha dx.\]
Applying the Cauchy-Schwarz inequality and using the estimates (\ref{bound_for_k_2r}) and  (\ref{bound_k3s}) from  \cref{main_prop_9} and \cref{bound_k3}  we deduce that
	\begin{equation*}
		\abs{S_{7,2}^{out}}\leq c\,(1+\norm{g}_{L^\infty})\norm{ g}_{L^2}\norm{\partial_x^3g}_{L^2}.
	\end{equation*}
For the term $S_{7,1},$ we observe
	\begin{equation*}
		\begin{split}
			\Delta_{\alpha}f K(x,\alpha)^3-\gamma_f(x,\alpha)&=\Delta_{\alpha }g\,\Gamma(x,\alpha),
		\end{split}
	\end{equation*}
where 
	\begin{equation}\label{Gamma_dfn}
		\Gamma(x,\alpha):=-\Delta_{\alpha }f(\Delta_\alpha f+\Delta_\alpha h)\bigg[\frac{K(x,\alpha)^3}{1+(\Delta_\alpha f)^2}+\frac{K(x,			\alpha)^2}{(1+(\Delta_\alpha f)^2)^2}+\frac{K(x,\alpha)}{(1+(\Delta_\alpha f)^2)^3}\bigg].
	\end{equation}
Notice $\abs{\Gamma(x,\alpha)}\leq c,$ we obtain a bound for  the \emph{in} part 
	\[\abs{S_{7,1}^{in}}\leq c\,\norm{\partial_xg}_{L^2}\norm{\partial_x^3g}_{L^2}.\]
Now expanding $\Delta_\alpha g$ we have 
	\begin{equation}\label{term_s7_out}
		\begin{split}
			S_{7,1}^{out}&=\int_{\mathbb R}\partial_x^3g(x)g(x)\int_{\abs{\alpha}>1}\frac{1}{\alpha}\Gamma(x,\alpha)d\alpha dx
				-\int_{\mathbb R}\partial_x^3g(x)\int_{\abs{\alpha}>1}\frac{g(x-\alpha)}{\alpha}\Gamma(x,\alpha)dxd\alpha.
		\end{split}
	\end{equation}
Using the estimate (\ref{bound_outside_gamma}) from  \cref{bounds_outside_gamma_n}  and the Cauchy-Schwarz inequality we find that
	\begin{equation*}
		\begin{split}
			\bigg|\int_{\mathbb R}\partial_x^3g(x)g(x)\int_{\abs{\alpha}>1}\frac{1}{\alpha}\Gamma(x,\alpha)d\alpha dx\bigg|       &\leq c\,(1+\norm{g}_{L^\infty})\norm{ g}_{L^2}\norm{\partial_x^3g}_{L^2}.
		\end{split}
	\end{equation*}
For the second term  in (\ref{term_s7_out}) we use the next bound
\[\abs{\Gamma(x,\alpha)}\leq c\,K(x,\alpha)+2\norm{g}_{L^\infty}\abs{\alpha}^{-1}.\]
Then we apply the  Cauchy-Schwarz and Minkowski's integral inequalities to obtain that
\[\bigg|\int_{\mathbb R}\partial_x^3g(x)\int_{\abs{\alpha}>1}\frac{g(x-\alpha)}{\alpha}\Gamma(x,\alpha)dxd\alpha\bigg|\leq c\,(1+\norm{g}_{L^\infty})\norm{g}_{L^2}\norm{\partial_x^3g}_{L^2},\]
and this completes the estimate for $S_{7,1}^{out}.$ Therefore
	\begin{equation}\label{bound_S7}\abs{S_7}\leq c\,(1+\norm{g}_{L^\infty})\norm{g}_{L^2}\norm{\partial_x^3g}_{L^2}.\end{equation}
Finally,  for $S_8$ we expand $(\Delta_{\alpha}h)^3$ 
and decompose $S_8:= S_{8,1}+S_{8,2}+S_{8,3}+S_{8,4}$ for 
		\begin{equation*}
			\begin{split}
				S_{8,1}&=\int_{\mathbb R}\partial_x^3g(x)\int_{\mathbb R}(\Delta_{\alpha}f)^3K(x,\alpha)^4d\alpha dx,\\
		S_{8,2}&=3\int_{\mathbb R}\partial_x^3g(x)\int_{\mathbb R}(\Delta_{\alpha}f)^2\Delta_{\alpha}gK(x,\alpha)^4d\alpha dx,\\
				S_{8,3}&=3\int_{\mathbb R}\partial_x^3g(x)\int_{\mathbb R}\Delta_{\alpha}f(\Delta_{\alpha}g)^2K(x,\alpha)^4d\alpha dx,\\
				S_{8,4}&=\int_{\mathbb R}\partial_x^3g(x)\int_{\mathbb R}(\Delta_{\alpha}g)^3K(x,\alpha)^4d\alpha dx.
			\end{split}
		\end{equation*}
Repeating the same argument as in $S_7,$  we find that 
	\[\abs{S_{8,2}+S_{8,3}+S_{8,4}}\leq c\,(1+\norm{g}_{C^1})^2\norm{g}_{H^3}^2.\]
For $S_{8,1}$ we consider the function
	\[\theta_f(x,\alpha):=\frac{(\Delta_\alpha f)^3}{(1+(\Delta_{\alpha}f)^2)^4},\]
we observe $\int_{\mathbb R}\theta_fd\alpha=0.$ Using the formula (\ref{formula_slope_g}) and adding $\theta_f$ we decompose $S_{8,1}$ in the next way 
	\begin{equation*}
		\begin{split}
			S_{8,1}&=\int_{\mathbb R}\partial_x^3g(x)\int_{\mathbb R}\Big[(\Delta_{\alpha}f)^3K(x,\alpha)^4-\theta_f(x,					\alpha)\Big]d\alpha dx\\
				&=\int_{\mathbb R}\partial_x^3g(x)\int_{\mathbb R}\Delta_{\alpha}g\,\Theta(x,\alpha)d\alpha dx\\
				&=\int_0^1\int_{\abs{\alpha}<1}\int_{\mathbb R}\partial_x^3g(x)\partial_xg(x+(s-1)									\alpha)\Theta(x,\alpha)dxd\alpha ds\\
				&+\int_{\mathbb R}\partial_x^3g(x)g(x)\int_{\abs{\alpha}>1}\frac{1}{\alpha}\Theta(x,								\alpha)d\alpha dx-\int_{\mathbb R}\partial_x^3g(x)\int_{\abs{\alpha}>1}\frac{g(x-\alpha)}{\alpha}						\Theta(x,\alpha)d\alpha dx,\\
		\end{split}
	\end{equation*}
where
		\begin{equation}\label{Teta_dfn}\setlength{\jot}{12pt}			
			\begin{split}
				\Theta(x,\alpha)&:=-(\Delta_{\alpha}f)^3(\Delta_{\alpha}f+\Delta_\alpha h)\bigg[\frac{K(x,\alpha)^4}{1+							(\Delta_\alpha f)^2}+\frac{K(x,\alpha)^3}{(1+(\Delta_\alpha f)^2)^2}\\
					&\hspace{6cm}+\frac{K(x,\alpha)^2}{(1+(\Delta_\alpha f)^2)^3}+\frac{K(x,\alpha)}{(1+(\Delta_\alpha f)^2)^4}\bigg].
			\end{split}
		\end{equation}
We use a similar bound as in (\ref{bounds_in_the_numerator}) to obtain that  $\abs{\Theta(x,\alpha)}<c\,(1+\norm{\partial_xg}_{L^\infty})^2.$	 Then  using Cauchy-Schwarz inequality  we find a bound for the \emph{in} part
\begin{equation*}
\abs{S_{8,1}^{in}}\leq c\,(1+\norm{\partial_xg}_{L^\infty})^2\norm{\partial_xg}_{L^2}\norm{\partial_x^3g}_{L^2}.\end{equation*}
For the last two terms we use the estimate (\ref{bound_outside_teta}) in \cref{bounds_outside_teta_n}, to control  the inside integral, and the estimate
	\[\abs{\Theta(x,\alpha)}\leq c\,(1+\norm{\partial_xg}_{L^\infty})^2 K(x,\alpha).\]
Then Cauchy-Schwarz inequality implies 
	\begin{equation*}
		\begin{split}
				\bigg|\int_{\mathbb R}\partial_x^2g(x)g(x)\int_{\abs{\alpha}>1}\frac{1}{\alpha}\Theta(x,\alpha)d\alpha dx\bigg|&\leq c\,(1+\norm{g}_{C^1})^2\norm{g}_{L^2}\norm{\partial_x^3g}_{L^2}.
		\end{split}
	\end{equation*}
Now we apply the Cauchy-Schwarz inequality first with respect to $x$ and then respect to $\alpha$ 
		\begin{equation*}
		\begin{split}
				\bigg|\int_{\mathbb R}\partial_x^2g(x)\int_{\abs{\alpha}>1}\frac{g(x-\alpha)}{\alpha}\Theta(x,\alpha)d\alpha dx\bigg|&\\&\hspace{-4cm}\leq c\,(1+\norm{g}_{C^1})^2\norm{g}_{L^2}\norm{\partial_x^3g}_{L^2}\bigg(\int_{\abs{\alpha}>1}\frac{1}{\abs{\alpha}^2}\int_{\mathbb R}K(x,\alpha)^2dxd\alpha\bigg)^{1/2}.
		\end{split}
	\end{equation*}
The estimate (\ref{bound_for_k_2r}) in \cref{main_prop_9} leads to 
	\[\abs{S_{8,1}^{out}}\leq c\,(1+\norm{g}_{C^1})^2\norm{g}_{L^2}\norm{\partial_x^3g}_{L^2}.\]
By bringing together  the inequalities for $S_{8,1},S_{8,2},S_{8,3}, S_{8,4} $ and the bound (\ref{bound_S7}) we complete the estimate for  $J_{5,1},$ and  we obtain that 
\begin{equation*}
		\abs{J_{5,1}}\leq c\,(1+\norm{g}_{C^1})^2\norm{g}_{H^3}^2.
\end{equation*} 
The previous estimates for $J_{5,2},J_{5,3},J_{5,4},$ and the last one, lead us  to conclude
that
\begin{equation}\label{bound_j5} \abs{J_5}\leq c\,(1+\norm{g}_{C^2})^3\norm{g}_{H^3}^2.\end{equation}
Using the estimates  (\ref{bound_j7}), (\ref{bound_j6}) and (\ref{bound_j5}) we deduce 
\[\abs{\,\mbox{IV}\,}\leq c\,(1+\norm{g}_{C^{2,1/2}})^3\norm{g}_{H^3}^2.\]
Finally, using the estimate (\ref{bound_iii})  we obtain 
	\[\abs{\,\mbox{III}\,}+ \abs{\,\mbox{IV}\,}\leq  c\,(1+\norm{g}_{C^{2,1/2}})^3\norm{g}_{H^3}^2. \]
The Sobolev embedding $H^3(\mathbb R)\hookrightarrow C^{2,1/2}(\mathbb R)$   completes the proof of the lemma.
\end{proof}

\noindent From the inequalities (\ref{one_eq}) and (\ref{main_ineq_2})  in \cref{mainlemma1}  and \cref{secondmain_lemma} we get 
\[\frac{d}{dt}(1+\norm{g(t)}_{H^3})\leq c\,\big(1+\norm{g(t)}_{H^3}\big)^4.\]
We integrate in time to obtain that 
\[\norm{g(t)}_{H^3}\leq \frac{\norm{g_0}_{H^3}}{\Big(1-c[\phi(0)]^3t\Big)^{1/3}},\]
where $\phi(0)=1+\norm{g_0}_{H^3}.$ Then the solution belongs to $H^3(\mathbb R)$ up to a time 
\[t<\frac{\phi(0)^{-3}}{c}=T^\star.\]

 \section{Bounds on the Kernels}\label{KernelBounds}
This section is devoted to the necessary lemmas used in the energy estimates.
 More precisely, we study   the  integrability and decay properties of  the kernels $K$ and $G$ defined in (\ref{kernelskg}). Throughout the section, we will  use often  the auxiliary globally Lipschitz  function  $F\colon \mathbb R\to\mathbb R$ defined by 
\[F(x):=\frac{1}{1+x^2}.\]
We start with the following lemma.
\begin{lemma}\label{hilbert_truncated}The truncated Hilbert transform of the  rational function
			\[r(x)=\frac{x^{m}}{(1+x^2)^{n}},\] 
for $m,n\in\mathbb N_+$  and  $m< 2n$ is  bounded. That is 
\[\abs{H_{\abs{y}<1}r(x)}<c\quad\mbox{and}\quad \abs{H_{\abs{y}>1}r(x)}<c.\]
\end{lemma}
\begin{proof} Using the definition of the Hilbert transform we have 
	\[Hr(x)=\frac{1}{\pi}PV\int_{\mathbb R}\frac{1}{y}\frac{(x-y)^m}{(1+(x-y)^2)^n}dy.\]
We know that the  Hilbert transform of rational function  is again a rational function and ${\abs{Hr(x)}\leq c.}$ Firstly, we estimate the \emph{in} part. We decompose the integrand using partial fractions as follows
\[\frac{1}{x-y}\frac{y^m}{(1+y^2)^n}=\frac{b(x)}{x-y}+\sum_{k=1}^n\frac{a_k(x)y+c_k(x)}{(1+y^2)^k},\]
where $b(x), a_k(x)$  and $c_k(x)$ are bounded terms. We obtain that 
\begin{equation*}
	\begin{split}
		H_{\abs{y}<1}r(x)&=\frac{1}{\pi}b(x)\int_{\abs{x-y}<1}\frac{1}{x-y}dy+\frac{1}{\pi}\sum_{k=1}^na_k(x)					\int_{\abs{x-y}<1}\frac{y}{(1+y^2)^k}dy\\
		&+\frac{1}{\pi}\sum_{k=1}^nc_k(x)\int_{\abs{x-y}<1}\frac{1}{(1+y^2)^k}dy.\\
	\end{split}
\end{equation*}
We deduce that $\abs{H_{\abs{y}<1}r(x)}<c.$ The bound  for  the \emph{out} part is easy because 
\[H_{\abs{y}>1}r(x)=Hr(x)-H_{\abs{y}<1}r(x),\]
thus $\abs{H_{\abs{y}>1}r(x)}<c$ which completes the proof.
\end{proof}

\begin{lemma} \label{aux_main_lemma1}Let $g\in H^s(\mathbb R)$  with $s\geq 3,$ then
    \begin{equation}\label{first_estimates_order_one}
        \begin{split}
            \bigg\lVert PV\int_{\mathbb R}\frac{1}{\alpha}K(\,\cdot\,,\alpha)d\alpha\bigg\rVert_{L^\infty}&\leq c\,(1+\norm{g}_{C^2})^3.
        \end{split}
    \end{equation}
\end{lemma}
\begin{proof} Notice  that  by definition  
      \[K(x,\alpha)=F(\Delta_\alpha h).\] 
We  decompose the integral in the next way 
   \begin{equation}\label{bound_k1_L_infinity}
    \begin{split}
       PV \int_{\mathbb R}\frac{1}{\alpha}K(x,\alpha) d\alpha&= \int_{\mathbb R}\frac{1}{\alpha} F(\Delta_\alpha f)d\alpha+\int_{\mathbb R}\frac{1}{\alpha}\big[F(\Delta_\alpha h)-F(\Delta_\alpha f)\big] d\alpha\\
            \end{split}
    \end{equation}
      where the first term is the  Hilbert transform of  $F,$ that is   
    \[PV\int_{\mathbb R}\frac{1}{\alpha }F(\Delta_{\alpha}f)d\alpha =PV\int_{\mathbb R}\frac{1}{\alpha}\frac{1}{1+(2x-\alpha)^2}d\alpha =\pi HF(2x),\]
   this  Hilbert transform is a rational function and is bounded.  To deal with the second term in (\ref{bound_k1_L_infinity}) we split it in the \emph{in} and \emph{out} parts. We compute  the difference  and observe 
        \begin{equation*}
            F(\Delta_\alpha h)-F(\Delta_{\alpha }f)=\Delta_{\alpha }g\,B(x,\alpha)
          \end{equation*}
        where 
        \begin{equation*}
          \begin{split}B(x,\alpha)
          &=-2\Delta_{\alpha }f F(\Delta_\alpha f)F(\Delta_\alpha h)-\Delta_{\alpha }g F(\Delta_\alpha f)F(\Delta_\alpha h)
          \end{split}
          \end{equation*}
        is a bounded term $\abs{B(x,\alpha)}\leq 2$.  Adding and subtracting  $\partial_xg(x)$   we have the next decomposition
          \begin{equation}\label{bounds_for_b1}
                \begin{split}
                  \int_{\abs{\alpha}<1}\frac{1}{\alpha}\big[F(\Delta_\alpha h)-F(\Delta_\alpha f)\big] d\alpha&=\int_{\abs{\alpha}<1}\frac{1}{\alpha}(\Delta_\alpha g-\partial_xg(x))B(x,\alpha)d\alpha \\
&+ \partial_xg(x)\int_{\abs{\alpha}<1}\frac{1}{\alpha}B(x,\alpha )d\alpha.
                \end{split}
            \end{equation}
  Now, from the Fundamental Theorem of Calculus we have the next bound
    \begin{equation}\label{second_order_g}\abs{\Delta_\alpha g- \partial_xg(x)}\leq c\, \norm{\partial_x^2g}_{L^\infty}\abs{\alpha}.\end{equation}
    Using the bound for $B(x,\alpha)$ and the last inequality we obtain that 
    \[\bigg|\int_{\abs{\alpha}<1}\frac{1}{\alpha}(\Delta_\alpha g-\partial_xg(x)) B(x,\alpha)d\alpha \bigg|\leq c\,\norm{\partial_x^2g}_{L^\infty}.\]
    For the second integral in (\ref{bounds_for_b1}), adding and subtracting  the terms $\partial_xg(x)$ and  $F(\partial_x h(x))$ we obtain the next decomposition
    \begin{equation*}
      \begin{split}
        \int_{\abs{\alpha}<1}\frac{1}{\alpha} B(x,\alpha)d\alpha
&=-2\int_{\abs{\alpha}<1}\frac{1}{\alpha}\Delta_{\alpha}f F(\Delta_\alpha f)\big[F(\Delta_\alpha h)- F(\partial_xh(x))\big]d\alpha\\
& -F(\partial_xh(x))\int_{\abs{\alpha}<1}\frac{\Delta_{\alpha}f}{\alpha} F(\Delta_\alpha f)d\alpha\\
        &- \int_{\abs{\alpha}<1 }\frac{1}{\alpha}\big(\Delta_\alpha g-\partial_xg(x)\big) F(\Delta_\alpha f) F(\Delta_{\alpha }h)d\alpha\\
&-\partial_x g(x)\int_{\abs{\alpha}<1}\frac{1}{\alpha}F(\Delta_{\alpha}f)\big[ F(\Delta_{\alpha}h)-F(\partial_x h(x))\big]d\alpha\\
        &-\partial_x g(x)F(\partial_x h(x))\int_{\abs{\alpha}<1}\frac{1}{\alpha}F(\Delta_{\alpha}f)d\alpha. 
      \end{split}
    \end{equation*}
Using the Lipschitz condition of $F$ and the Fundamental Theorem of Calculus we deduce that 
	\begin{equation}\label{second_order_h}\abs{F(\Delta_\alpha h)-F(\partial_xh(x))}\leq c\,\abs{\Delta_{\alpha}h-\partial_xh(x)}\leq c\, \big(1+\norm{\partial_x^2g}_{L^{\infty}}\big)\,\abs{\alpha}\end{equation} and from   \cref{hilbert_truncated}  we find that 
\[\bigg|\int_{\abs{\alpha}<1}\frac{1}{\alpha}F(\Delta_{\alpha}f)d\alpha\bigg|,\,\bigg|\int_{\abs{\alpha}<1}\frac{1}{\alpha}\Delta_\alpha fF(\Delta_\alpha f)d\alpha \bigg|<c,\]
in the last integral  we recall the  definition of $f$. Therefore we conclude from (\ref{second_order_g}) and (\ref{second_order_h}) that
 \[\bigg|\partial_xg(x)\int_{\abs{\alpha}<1}\frac{1}{\alpha} B(x,\alpha)d\alpha\bigg|\leq c\,(1+\norm{g}_{C^2})^3.\]
 The bound for the \emph{out} part in the second term of (\ref{bounds_for_b1}) can be deduced from the Lipschitz condition of $F$
 \begin{equation}\label{lip_f}\abs{F(\Delta_\alpha h)-F(\Delta_\alpha f)}\leq 2\norm{g}_{L^\infty}\abs{\alpha}^{-1}.\end{equation}
Then using the fact that $B(x,\alpha)$ is bounded, we conclude that 
 \[\bigg|\int_{\abs{\alpha}>1}\frac{1}{\alpha}\big[ F(\Delta_\alpha h)-F(\Delta_{\alpha }f)\big]d\alpha\bigg|\leq c\,\norm{g}_{L^\infty}.\]
 and this completes the proof.
\end{proof}
The following  result  presents  a similar estimate to the previous lemma, but now for the kernel $G$. 
\begin{lemma}\label{aux_main_prop3}Let $g\in H^s(\mathbb R)$  with $s\geq 3,$ then
        \begin{equation}\label{first_estimates_order_3}
          \bigg\lVert PV\int_{\mathbb R}\frac{1}{\alpha } G(\,\cdot\,,\alpha)d \alpha\bigg\rVert_{L^\infty}\leq c\,(1+\norm{g}_{C^2})^2.
        \end{equation}
\end{lemma}
\begin{proof}
Using the   function $F$ we  rewrite the integral as
            \begin{equation*}
                \begin{split}
                    PV\int_{\mathbb R}\frac{1}{\alpha}G(x,\alpha)d\alpha &=-4\int_{\mathbb R}\frac{1}{\alpha}\Delta_\alpha fF(\Delta_\alpha f)F(\Delta_\alpha h)d\alpha\\
                    &\hspace{0cm}-2\int_{\mathbb R}\frac{1}{\alpha}\Delta_\alpha gF(\Delta_\alpha f)F(\Delta_\alpha h)d\alpha:=-4G_1-2G_2.
                \end{split}
            \end{equation*}
    We start with the  bound for  the  \emph{in} part in $G_1$. Notice that adding and subtracting $F(\partial_xh(x)),$ we obtain the next decomposition
        \begin{equation*}
            \begin{split}
                G_1^{in}& =\int_{\abs{\alpha}<1}\frac{1}{\alpha}\Delta_\alpha fF(\Delta_\alpha f)\big[F(\Delta_\alpha h)-F(\partial_xh(x))\big]d\alpha +F(\partial_xh(x))\int_{\abs{\alpha}<1}\frac{1}{\alpha}\Delta_\alpha fF(\Delta_\alpha f)d\alpha.
            \end{split}
        \end{equation*}
        Then, in a similar way to the \cref{aux_main_lemma1},  we use  the Lipschitz condition of $F$ to obtain that
        \[\abs{G_1^{in}}\leq c\,(1+\norm{g}_{C^2}).\]
Now for the \emph{out} part we add and subtract $F(\Delta_{\alpha}f).$ We find that 
    \begin{equation*}
        \begin{split}
        G_1^{out}&=\int_{\abs{\alpha}>1}\frac{1}{\alpha}\Delta_\alpha fF(\Delta_\alpha f)\big[F(\Delta_\alpha h)-F(\Delta_\alpha f)\big]d\alpha+\int_{\abs{\alpha}>1}\frac{1}{\alpha}\Delta_\alpha f[F(\Delta_\alpha f)]^2d\alpha.
        \end{split}
    \end{equation*}
    Using  the Lipschitz condition (\ref{lip_f})  we obtain that
     \begin{equation*}
		\abs{F(\Delta_\alpha h)-F(\Delta_\alpha f)}\leq  2\norm{ g}_{L^\infty}\abs{\alpha}^{-1}
	\end{equation*} 
	and from \cref{hilbert_truncated} we have 
\begin{equation}\label{bound_Hf_2}\bigg|\int_{\abs{\alpha}>1}\frac{1}{\alpha}\Delta_\alpha f[F(\Delta_{\alpha}f)]^2d\alpha\bigg|<c,\end{equation}
hence 
    \[\abs{G_1^{out}}\leq c\,(1+\norm{g}_{L^\infty}).\] In order to estimate  $G_2,$ for the \emph{in} part  we  add and subtract the terms $\partial_x g(x)$ and $F(\partial_xh(x))$ to obtain the next decomposition
        \begin{equation*}
        \begin{split}
        G_2^{in}&=\int_{\abs{\alpha}<1}\frac{1}{\alpha}(\Delta_\alpha g-\partial_xg(x))F(\Delta_\alpha f)F(\Delta_\alpha h)d\alpha \\
        &+\partial_xg(x)\int_{\abs{\alpha}<1}\frac{1}{\alpha}F(\Delta_\alpha f)
        \big[F(\Delta_\alpha h)-F(\partial_xh(x))\big]d\alpha\\
        &+\partial_xg(x)F(\partial_x h(x))\int_{\abs{\alpha}<1}\frac{1}{\alpha}F(\Delta_{\alpha}f) d\alpha,
        \end{split}
    \end{equation*}
   which are the terms appearing in $G_1^{in}$ and (\ref{bounds_for_b1}). Hence
    \[\abs{G_2^{in}}\leq c\,(1+\norm{g}_{C^2})^2.\]
    Finally, for the \emph{out} part $G_2^{out}$, we observe
    \[\abs{G_2^{out}}\leq  \int_{\abs{\alpha}>1}\frac{\abs{g(x)-g(x-\alpha)}}{\alpha^2} \abs{F(\Delta_{\alpha }f)F(\Delta_{\alpha}h)}d\alpha\leq 2\norm{g}_{L^\infty}\int_{\abs{\alpha}>1}\abs{\alpha}^{-2}d\alpha\]
    and this completes the proof.
      \end{proof}
      In the next lemma we prove similar estimates now for the derivative in $x$ of the kernel $K(x,\alpha).$
\begin{lemma}\label{aux_main_prop4}Let $g\in H^s(\mathbb R)$  with $s\geq 3,$ then 
      \begin{equation}\label{second_estimates_order_one}
            \begin{split}    
                \bigg\lVert PV\int_{\mathbb R} \frac{1}{\alpha}\partial_x K(\,\cdot\,,\alpha)d\alpha\bigg\rVert_{L^\infty}&\leq c\,(1+\norm{g}_{C^{2,\delta}})^2\quad \mbox{for}\quad\delta\in(0,1).
                \end{split}
                \end{equation}
      \end{lemma}
\begin{proof}
First we  note that
      \[\partial_xK(x,\alpha)=F'(\Delta_{\alpha}h)\partial_x\Delta_\alpha h.\]
For the \emph{in}  part we add and subtract $\partial_x^2h(x)$ and $F'(\partial_xh(x))$ and decompose in the following  way
      \begin{equation*}
        \begin{split}
          PV\int_{\abs{\alpha}<1}\frac{1}{\alpha}\partial_x K(x,\alpha)d\alpha&=\int_{\abs{\alpha}<1}\frac{F'(\Delta_\alpha h)}{\alpha}\big[\partial_x \Delta_\alpha h-\partial_x^2h(x)\big]d\alpha\\
          &+\partial_x^2h(x)\int_{\abs{\alpha}<1}\frac{1}{\alpha}\big[F'(\Delta_{\alpha}h)-F'(\partial_xh(x))\big]d\alpha.
        \end{split}
      \end{equation*}
      Using the  inequalities 
      \begin{equation*}
        \begin{split}
          &\abs{\partial_x\Delta_\alpha h-\partial_x^2h(x)}\leq c\abs{\partial_x^2g}_{C^\delta}\cdot\abs{\alpha}^\delta,\quad\mbox{for}\quad\delta\in(0,1),\\
          &\abs{F'(\Delta_\alpha h)-F'(\partial_xh(x))}\leq c\, \abs{\Delta_\alpha h-\partial_xh(x)},
        \end{split}
      \end{equation*}
      it follows the next bound
    \begin{equation}\label{bound_kx_int_part}\bigg|\int_{\abs{\alpha}<1}\frac{1}{\alpha}\partial_x K(x,\alpha)d\alpha \bigg|\leq c\,(1+\norm{g}_{C^{2,\delta}})^2.\end{equation}
    For the \emph{out} part, by  adding  and subtracting the term $F'(\Delta_\alpha f)$ we obtain that 
    \begin{equation}\label{split_out_part}
      \begin{split}
        PV\int_{\abs{\alpha}>1}\frac{1}{\alpha}\partial_x K(x,\alpha)d\alpha &=\int_{\abs{\alpha}>1}\frac{1}{\alpha}\big[F'(\Delta_\alpha h)-F'(\Delta_{\alpha}f)\big]\partial_x\Delta_\alpha hd\alpha \\
        &+\int_{\abs{\alpha}>1}\frac{1}{\alpha} F'(\Delta_{\alpha }f)(\partial_x\Delta_{\alpha}h)d\alpha.
      \end{split}
    \end{equation}
   Notice that 
    	\begin{equation*}
    		\begin{split}
    			&\abs{F'(\Delta_\alpha h)-F'(\Delta_{\alpha }f)}\leq c\,\abs{\Delta_\alpha g},\\
			&\abs{\partial_x\Delta_\alpha h}\leq c\,(1+\norm{\partial_x^2g}_{L^\infty}).
		\end{split}
	\end{equation*}
    Hence  the following bound is automatic
    	\begin{equation*}
    		\begin{split}
    		\bigg|\int_{\abs{\alpha}>1}\frac{1}{\alpha}\big[F'(\Delta_\alpha h)-F'(\Delta_{\alpha}f)\big]\partial_x\Delta_\alpha h         d\alpha \bigg|&\leq c\,(1+\norm{\partial_x^2g}_{L^\infty})\int_{\abs{\alpha}>1}\frac{\abs{g(x)-g(x-\alpha)}}{\alpha^2}d\alpha\\
    		&\leq c\,(1+\norm{\partial_x^2g}_{L^\infty})\norm{g}_{L^\infty}.
    		\end{split}
    	\end{equation*}
    For the second integral in the right hand side of (\ref{split_out_part}) we expand
    \[\partial_x\Delta_\alpha h=2+\partial_x\Delta_\alpha g,\]
    and decompose 
    \[\int_{\abs{\alpha}>1}\frac{1}{\alpha} F'(\Delta_{\alpha }f)(\partial_x\Delta_{\alpha}h)dx=2\int_{\abs{\alpha}>1}\frac{1}			{\alpha} F'(\Delta_\alpha f)d\alpha +\int_{\abs{\alpha}>1}\frac{\partial_xg(x)-\partial_xg(x-\alpha)}{\alpha^2}F'(\Delta_\alpha f)d\alpha. \]
    Notice that  \[F'(\Delta_\alpha f)=-2\Delta_\alpha fF(\Delta_\alpha f)^2\quad\mbox{and}\quad\abs{F'(\Delta_{\alpha }f)}<2.\] 
Using  the estimate (\ref{bound_Hf_2}) we  obtain that 
     \[\bigg|\int_{\abs{\alpha}>1}\frac{1}{\alpha}\partial_x K(x,\alpha)d\alpha \bigg|\leq c\,(1+\norm{g}_{C^{1}}).\]
The last bound together with the estimate (\ref{bound_kx_int_part})     completes the proof.
      \end{proof}
  In  the following lemma  we obtain  a bound for    $\Lambda K(x,0)$   where $K(x,0)$ is the kernel at zero.
\begin{lemma}\label{lambdak0}
    Let $g\in H^s(\mathbb R)$ with $s\geq 3,$ then  we have the next bound
    \begin{equation}\label{lambdak_0}\norm{\Lambda K(x,0)}_{L^\infty}\leq c\,(1+\norm{g}_{C^{2,\delta}})\quad\mbox{ for }\quad\delta\in(0,1).\end{equation}
    \end{lemma}
\begin{proof}
    By definition of the operator $\Lambda$ we have 
    \[\Lambda K(x,0)=\frac{1}{\pi}PV\int_{\mathbb R}\frac{K(x,0)-K(y,0)}{(x-y)^2}dy,\]
    where 
    \[K(x,0)=\frac{1}{1+(\partial_xh(x))^2}.\]
    We denote  $K(x,0):=K(x)$ and  split in the \emph{in} and \emph{out} parts. We  change variables $y=x-y$ to obtain that
    \[\Lambda K(x)=\frac{1}{\pi}\int_{\mathbb R}\frac{K(x)-K(x-y)}{y^2}dy\]
and write as
	\[\int_{\abs{y}<1}\frac{K(x)-K(x-y)}{y^2}dy=\frac{1}{2}\int_{\abs{y}<1}\frac{K(x)-K(x-y)}{y^2}dy+\frac{1}{2}\int_{\abs{y}<1}			\frac{K(x)-K(x+y)}{y^2}dy,\]
	hence 
    \begin{equation*}\setlength{\jot}{12pt}
        \begin{split}
        \Lambda K(x,0)&=\frac{1}{2\pi}\int_{\abs{y}<1}\frac{2K(x)-K(x+y)-K(x-y)}{y^2}dy+\frac{1}{\pi}\int_{\abs{x-y}>1}\frac{K(x)-K(y)}{(x-y)^2}dy\\
        &=I^{in}+I^{out}.
        \end{split}
    \end{equation*}
     Using  the Fundamental Theorem of Calculus we obtain the formulas
    \[K(x)-K(x-y)=\int_{0}^1K'(x+(1-s)y)\partial_x^2h(x+(1-s)y)ds\cdot y\]
    and
    \[K(x)-K(x+y)=-\int_{0}^1K'(x+sy)\partial_x^2h(x+sy)ds\cdot y.\]
    Let us recall that  $\partial_x^2h(x)=2+\partial_x^2g(x).$ Thus, we have the next estimate 
    \begin{equation*}
        \begin{split}\setlength{\jot}{14pt}
            \abs{2K(x)-K(x-y)-K(x+y)}&\leq \norm{\partial_x K}_{L^\infty}\int_{0}^1\abs{\partial_x^2g(x+(1-s)y)-\partial_x^2g(x+sy)}ds\cdot\abs{y}\\
            &\leq  c\,\abs{y}\sup_{x\neq y}\abs{\partial_x^2g(x+(1-s)y)-\partial_x^2g(x+sy)}\\
            &\leq c\abs{y}^{1+\delta}\abs{\partial_x^2g}_{C^\delta},
        \end{split}
    \end{equation*}
    where $\delta\in(0,1).$ Hence the \emph{in} part on  $\Lambda K$ is bounded by 
    \[\abs{I^{in}}\leq c\,\norm{g}_{C^{2,\delta}}\int_{\abs{y}<1}\abs{y}^{\delta-1}dy.\]
    Now, for the \emph{out} part is enough to see that $0< K(x)\leq 1$, for all $x\in\mathbb R.$ Hence
    \begin{equation*}\setlength{\jot}{12pt}
        \begin{split}
            \abs{I^{out}}&\leq \frac{1}{\pi}\int_{\abs{x-y}>1}\frac{\abs{K(x)-K(y)}}{(x-y)^2}dy\\
            &\leq \frac{2}{\pi}\int_{\abs{x-y}>1}\frac{1}{\abs{x-y}^2}dy<\infty,
        \end{split}
    \end{equation*}
    and  this completes the proof. \end{proof}
In the following lemma we recall that $\Phi(x,\alpha)$ is the derivative of the difference
	\[\Phi(x,\alpha):=\partial_\alpha\bigg[\frac{K(x,\alpha)-K(x,0)}{\alpha}\bigg].\]
	\begin{lemma}\label{aux_main_prop5}Let $g\in H^s(\mathbb R)$  with $s\geq 3,$ then for every $\delta \in (0,1)$ we have 
      \begin{equation}\label{second_order_k}
            \begin{split}    
            \big|\Phi(x,\alpha)\big|\leq c\,(1+\norm{g}_{C^{2,\delta}})^2\abs{\alpha}^{\delta-1}, \quad \mbox{for}\quad x\in\mathbb R.
            \end{split}
        \end{equation}
                \end{lemma}
	\begin{proof}
		Recall that $F(\Delta_\alpha h)= K(x,\alpha)$ and write 
                  \[F(\alpha)=K(x,\alpha)\,\,\mbox{ and } \,\,F(0)=K(x,0).\]
	            Then we integrate in the next way 
                \begin{equation*}
                   \begin{split}
                      \Phi(x,\alpha)&=-\frac{F(\alpha)-F(0)}{\alpha^2}+\frac{1}{\alpha}F'(\alpha)\\
                      &=-\frac{1}{\alpha^2}\int_{0}^\alpha F'(z)dz+\frac{F'(\alpha)}{\alpha}\\
                      &=\frac{1}{\alpha^2}\int_0^\alpha\int_z^\alpha F''(w)dwdz.\\
                    \end{split}
                  \end{equation*}
                  Hence 
                  \[\abs{\Phi(x,\alpha)}\leq c\,\abs{\partial_\alpha^2 K(x,\alpha)}\]
                 A direct computation yields to
        \begin{equation}\label{second_derivative_in_alpha}
            \begin{split}
                \partial_\alpha^2K(x,\alpha)=
                F''(\Delta_{\alpha}h)[\partial_\alpha \Delta_\alpha h]^2+F'(\Delta_\alpha h)\partial_\alpha^2\Delta_\alpha g.
            \end{split}
            \end{equation}
            Using the Fundamental Theorem of Calculus  we obtain the next indentities
          \begin{equation*}\setlength{\jot}{12pt}
            \begin{split}
              &\partial_\alpha^2\Delta_\alpha g(x)=\frac{1}{\alpha}\int_{0}^1\int_0^1\Big[\partial_x^2g(x+(rs-1)\alpha)-                   			\partial_x^2g(x-\alpha)\Big](2s)drds,\\ 
              &\partial_\alpha \Delta_{\alpha}h(x)=\int_{0}^1(s-1)\partial_x^2h(x+(s-1)\alpha)ds,
              \end{split}
            \end{equation*}
      where the integrands are  bounded by
            \begin{equation*}
            	\begin{split}
	            	\big|\partial_x^2g(x+(rs-1)\alpha)-\partial_x^2g(x-\alpha)\big|&\leq c\,\norm{g}_{C^{2,\delta}}\abs{\alpha}						^\delta,\\
     	       	\abs{\partial_x^2h(x+(s-1)\alpha)}&\leq c\,(1+\norm{\partial_x^2g}_{L^\infty}).
            	\end{split}
           \end{equation*}
          It follows from  equation (\ref{second_derivative_in_alpha}) that 
        \begin{equation}\label{second_order_K_2}
        		\begin{split}
        			\abs{\Phi(x,\alpha)}\leq  c\,\abs{\partial_\alpha^2 K(x,\alpha)}
						&\leq c\,(1+\norm{g}_{C^{2,\delta}})^2\abs{\alpha}^{\delta-1},
        		\end{split}
        	\end{equation}
        	which completes the proof.
        \end{proof}
      \begin{lemma}\label{main_prop_9}
	        Let $g\in H^s(\mathbb R)$ with $s\geq 3,$    The kernel $K(x,\alpha)$ belongs to  $L_x^2(\mathbb R),$  that is 
      \begin{equation}
          \label{bound_for_k_2r}
          \int_{\mathbb R} K(x,\alpha)^2dx \leq c\,(1+\norm{\partial_xg}_{L^\infty}).
          \end{equation}
    \end{lemma} 
 \begin{proof} 
Notice that 
		\[K(x,\alpha)^2< K(x,\alpha)<1\]
 and the lower bound
	\[\Delta_{\alpha}h\geq 2x-\alpha-\norm{\partial_xg}_{L^\infty}.\]
Using the last lower bound, we have   \[K(x,\alpha)\leq \frac{1}{1+(2x-\alpha)^2}\quad\mbox{if}\quad x\geq \norm{\partial_xg}_{L^\infty}\] and $K(x,\alpha)<1$ for any $x\in\mathbb R.$ Then we split
	\[\int_{0}^\infty K(x,\alpha)dx\leq \int_0^{\norm{\partial_xg}_{L^\infty}}dx+\int_{\norm{\partial_xg}_{L^\infty}}^\infty\frac{1}{1+(2x-\alpha)^2}dx.\]
The first integral is bounded by $\norm{\partial_x g}_{L^\infty},$ while for the second one, the change of variable $z=2x-\alpha$
    implies that
	\[\int_{0}^\infty\frac{1}{1+(2x-\alpha)^2}dx\leq \frac{1}{2}\int_{\mathbb R}\frac{dz}{1+z^2}<\infty.\]
    which completes the proof. 
		
	\end{proof}
 \begin{lemma}\label{main_prop_10}
	        Let $g\in H^s(\mathbb R)$ with $s\geq 3,$    The kernel $G(x,\alpha)$ belongs to  $L_x^2(\mathbb R),$  that is 
      \begin{equation}
          \label{bound_for_g_2r}
          \int_{\mathbb R} G(x,\alpha)^2dx \leq c\,(1+\norm{\partial_xg}_{L^\infty})^3.
          \end{equation}
    \end{lemma} 
	\begin{proof}
	From  the definition (\ref{kernelskg}) we have  that
		\[G(x,\alpha)=-\frac{2\Delta_{\alpha }f+\Delta_{\alpha}g}{(1+(\Delta_{\alpha }f)^2)(1+(\Delta_{\alpha}h)^2)}=-(2\Delta_{\alpha}f+\Delta_{\alpha }g)K(x,\alpha)F(\Delta_{\alpha }f).\] 
	We decompose the sum and observe 
		\[\abs{G(x,\alpha)}\leq 2\abs{\Delta_{\alpha}f} F(\Delta_{\alpha}f)K(x,\alpha)+\norm{\partial_xg}_{L^\infty}K(x,\alpha)\leq (2+\norm{\partial_xg}_{L^\infty})K(x,\alpha). \]
	Then 
		\[G(x,\alpha)^2\leq (2+\norm{\partial_xg}_{L^\infty})^2K(x,\alpha)^2.\]
	Now we integrate
		\[\int_{\mathbb R}G(x,\alpha)^2dx\leq (2+\norm{\partial_xg}_{L^\infty})^2\int_{\mathbb R}K(x,\alpha)^2dx\]
	then the proof follows from  \cref{main_prop_9}.  
	\end{proof}

 \begin{lemma}\label{main_prop_11}
        Let $g\in H^s(\mathbb R)$ with $s\geq 3,$   The second derivate  respect to $\alpha $ of the kernel $K(x,\alpha)$ belongs to  $L_x^2(\mathbb R),$  that is 
      \begin{equation}
          \label{bound_for_k_2rr}
          \int_{\mathbb R}\partial_\alpha K(x,\alpha)^2dx \leq c\,(1+\norm{\partial_x g}_{L^\infty})^3.
      \end{equation}
  \end{lemma}
  \begin{proof}  Recall that  $K(x,\alpha)=F(\Delta_\alpha h),$ then the derivative with respect to $\alpha$ is given by
        \[\partial_\alpha K(x,\alpha)=F'(\Delta_\alpha h)\partial_\alpha \Delta_\alpha h.\]
    Now we observe   
	\[F'(\Delta_{\alpha}h)\leq 2K(x,\alpha)\]
and from the Fundamental Theorem of Calculus we have 
    \[\abs{\partial_\alpha \Delta_{\alpha} h} \leq 2+ \norm{\partial_x^2g}_{L^\infty},\]
    which implies that 
        \[\abs{\partial_\alpha K(x,\alpha)}^2\leq c\,(1+\norm{\partial_x^2g}_{L^\infty})^2 K(x,\alpha).\]
then the estimate  follows from  \cref{main_prop_9}.
  \end{proof}
 \begin{lemma}\label{aux_main_prop7}Let $g\in H^s(\mathbb R)$  with $s\geq 3,$ we have 
      \begin{equation}\label{second_estimates_order_one_far_zero}
            \begin{split}    
                \bigg\lVert PV\int_{\abs{\alpha}>1} \frac{1}{\alpha}\partial_x^2 K(\,\cdot\,,\alpha)d\alpha\bigg\rVert_{L^\infty}&\leq c\, (1+\norm{g}_{C^{2,\delta}})^2\quad\mbox{for}\quad\delta\in(0,1).
                \end{split}
                \end{equation}
      \end{lemma}
 \begin{proof}
      Using the indentity (\ref{second_K}) we have 
         \begin{equation*}
            \begin{split}
                \partial_x^2K(x,\alpha)=&(\partial_x^2\Delta_\alpha g )B_1(x,\alpha )+(\partial_x\Delta_\alpha h)^2B_2(x,\alpha)
            \end{split}
        \end{equation*}
  where $B_1(x,\alpha)$ and $B_2(x,\alpha)$ are bounded terms. We decompose the integral in the next way 
        \begin{equation*}\setlength{\jot}{12pt}
          \begin{split}
              \int_{\abs{\alpha}>1}\frac{1}{\alpha}\partial_x^2K(x,\alpha)d\alpha&=\int_{\abs{\alpha}>1}\frac{1}{\alpha^2}\big(\partial_x^2g(x)-\partial_x^2g(x-\alpha)\big) B_1(x,\alpha)d\alpha\\
              &+\int_{\abs{\alpha}>1}\frac{1}{\alpha}(2+\partial_x\Delta_\alpha g)^2B_2(x,\alpha)d\alpha.
          \end{split}
        \end{equation*}
        For the first integral  in the right hand side, we note that $\abs{B_1(x,\alpha)}=\abs{F'(\Delta_\alpha h)}\leq 2$ and 
        \[\abs{\partial_x^2g(x)-\partial_x^2g(x-\alpha)}\leq \abs{\partial_x^2 g}_{C^\delta}\cdot\abs{\alpha}^\delta,\quad\mbox{
for}\quad \delta\in(0,1).\]
        Therefore 
        \[\bigg|\int_{\abs{\alpha}>1}\frac{1}{\alpha^2}\big(\partial_x^2g(x)-\partial_x^2g(x-\alpha)\big) B_1(x,\alpha)d\alpha\bigg|\leq c\,\norm{g}_{C^{2,\delta}}\int_{\abs{\alpha}>1}\abs{\alpha}^{2-\delta}d\alpha ,\]
        which is integrable. To get the bound for the second integral we observe 
      \[B_2(x,\alpha)=-2F(\Delta_\alpha h)^2+ 8(\Delta_{\alpha}h)^2F(\Delta_\alpha h)^3,\]        
  then we  proceed as in  \cref{aux_main_prop4} to obtain that 
      \[\bigg| \int_{\abs{\alpha}>1}\frac{1}{\alpha}(\partial_x\Delta_\alpha h )^2 B_2(x,\alpha)d\alpha \bigg|\leq c\,(1+\norm{\partial_xg}_{L^\infty})^2,\]
      and this completes the proof.
      \end{proof}
\begin{lemma}\label{bound_k3} Let $g\in H^s(\mathbb R)$ with $s\geq 3$  then
	\begin{equation}\label{bound_k3s}
		\bigg|PV\int_{\abs{\alpha}>1}\frac{1}{\alpha}K(x,\alpha)^3d\alpha\bigg|\leq c\,(1+\norm{g}_{L^\infty}),
	\end{equation}
\end{lemma}
\begin{proof}
Using $K(x,\alpha)=F(\Delta_{\alpha}h)$ and  adding a subtracting  $F(\Delta_{\alpha}f)$ we have the next decomposition
	\begin{equation*}
		\begin{split}
			K(x,\alpha)^3&=F(\Delta_{\alpha}h)^2\big[F(\Delta_\alpha f)-F(\Delta_{\alpha}f)\big]
                                 +F(\Delta_{\alpha}h)\big[F(\Delta_\alpha f)-F(\Delta_{\alpha}f)\big]F(\Delta_\alpha f )\\
              			&+\big[F(\Delta_\alpha f)-F(\Delta_{\alpha}f)\big]F(\Delta_\alpha f )^2+F(\Delta_\alpha f)^3:=\Xi(x,\alpha)+F(\Delta_\alpha f)^3
		\end{split}
	\end{equation*}
Using the Lipschitz condition (\ref{lip_f}) we see that $\abs{\Xi(x,\alpha)}/\alpha$ is integrable for $\abs{\alpha}>1.$ Finally using \cref{hilbert_truncated} we obtain that 
	\[\bigg|\int_{\abs{\alpha}>1}\frac{1}{\alpha}K(x,\alpha)^3d\alpha\bigg|\leq c\,\norm{g}_{L^\infty}+\bigg|\int_{\abs{\alpha}>1}\frac{1}{\alpha}F(\Delta_{\alpha}f)^3d\alpha\bigg|\leq c\,(1+\norm{g}_{L^\infty}),\]
and this completes the proof.
\end{proof}
In the next lemma we recall the definition (\ref{def_gamma})
	\begin{equation*}
		\gamma(x,\alpha)=24(\Delta_\alpha h)K(x,\alpha)^3-48(\Delta_\alpha h)^3 K(x,				\alpha)^4.
	\end{equation*}

\begin{lemma}\label{bounds_outside_kn} Let $g\in H^s(\mathbb R)$ with $s\geq 3,$  we have  the next bound
	\begin{equation}\label{bound_outside}
		\bigg|PV\int_{\abs{\alpha}>1}\frac{1}{\alpha}\gamma(x,\alpha)d\alpha\bigg|\leq c\,(1+\norm{g}_{L^\infty})^3.
	\end{equation}
\end{lemma}
\begin{proof} Recall that $K(x,\alpha)=F(\Delta_{\alpha}h).$ Using the definition (\ref{def_gamma}) we expand $\Delta_{\alpha}h$ and $(\Delta_\alpha h)^3$ to obtain that 
	\begin{equation}\label{gamma_decomposition}
		\begin{split}
			\gamma(x,\alpha)
			&=24\Delta_{\alpha}f K(x,\alpha)^3+ 24\Delta_{\alpha}g K(x,\alpha)^3-48(\Delta_{\alpha}f)^3K(x,\alpha)^4\\
			&-48\cdot 3(\Delta_{\alpha}f)^2\Delta_{\alpha}g K(x,\alpha)^4-48\cdot 3(\Delta_{\alpha}f)(\Delta_{\alpha}g)^2 K(x,\alpha)^4\\
			&-48(\Delta_{\alpha}g)^3 K(x,\alpha)^4.
		\end{split}
	\end{equation}
The second and last terms in (\ref{gamma_decomposition}) are easily bounded by
\[\abs{24\Delta_{\alpha}g K(x,\alpha)^3-48\Delta_{\alpha}g K(x,\alpha)^4}\leq c\,\frac{\norm{g}_{L^\infty}}{\abs{\alpha}}+c\frac{\norm{g}_{L^\infty}^3}{\abs{\alpha}^3}.\]
For the fourth term, by adding and subtracting $\Delta_{\alpha}g,$  we obtain the next decomposition 
	\begin{equation}\label{bounds_in_the_numerator}
		\begin{split}
			(\Delta_{\alpha}f)^2\Delta_{\alpha}gK(x,\alpha)^4&=(\Delta_{\alpha}h)^2\Delta_{\alpha}gK(x,\alpha)^4-2\Delta_{\alpha}h						(\Delta_{\alpha}g)^2K(x,\alpha)^4+(\Delta_{\alpha}g)^3K(x,\alpha)^4.
		\end{split}		
	\end{equation}
Hence the fourth term in (\ref{gamma_decomposition}) is bounded by 
\[	\abs{(\Delta_{\alpha}f)^2\Delta_{\alpha}gK(x,\alpha)^4}\leq c\,\frac{\norm{g}_{L^\infty}}{\abs{\alpha}}+c\,\frac{\norm{g}_{L^\infty}^2}{\abs{\alpha}^2}+c\,\frac{\norm{g}_{L^\infty}^3}{\abs{\alpha}^3}.\]
In a similar way the fifth term  is bounded by
	\[\abs{\Delta_{\alpha}f(\Delta_{\alpha}g)^2K(x,\alpha)^4}\leq c\,\frac{\norm{g}_{L^\infty}^2}{\abs{\alpha}^2}+c\,\frac{\norm{g}_{L^\infty}^3}{\abs{\alpha}^3}.\]
For the first term  adding and subtracting $F(\Delta_{\alpha}f)$ we have the next decomposition
		\begin{equation}\label{good_decomposition_outside}\setlength{\jot}{12pt}
			\begin{split}
				\Delta_{\alpha} f K(x,\alpha)^3&=\Delta_{\alpha}f F(\Delta_\alpha h)^2\big[F(\Delta_\alpha h)-F(\Delta_\alpha f)\big]\\
&+\Delta_{\alpha}f F(\Delta_\alpha h)\big[F(\Delta_\alpha h)-F(\Delta_\alpha f)\big]F(\Delta_\alpha f)\\
					&+\Delta_{\alpha}f \big[F(\Delta_\alpha h)-F(\Delta_\alpha f)\big]F(\Delta_{\alpha}f)^2+\Delta_{\alpha}f F(\Delta_\alpha f)^3.
			\end{split}
		\end{equation}
Using the Lipschitz condition  (\ref{lip_f}) and estimates from \cref{hilbert_truncated} we obtain that
	\begin{equation*}\setlength{\jot}{12pt}
		\begin{split}
			\bigg|\int_{\abs{\alpha}}\frac{1}{\alpha}\Delta_{\alpha }fK(x,\alpha)^3d\alpha\bigg|&\leq c\,\norm{g}_{L^\infty}+\bigg|\int_{\abs{\alpha}>1}\frac{1}{\alpha}\Delta_{\alpha}f F(\Delta_{\alpha}f)^3d\alpha\bigg|\\
			&\leq c\,(1+\norm{g}_{L^\infty}).
		\end{split}
	\end{equation*} 
Similarly we find that 
		\begin{equation*}
		\begin{split}
			\bigg|\int_{\abs{\alpha}}\frac{1}{\alpha}(\Delta_{\alpha }f)^3K(x,\alpha)^4d\alpha\bigg|
			&\leq c\,(1+\norm{g}_{L^\infty}).
		\end{split}
	\end{equation*} 
We conclude the proof by  using the decay at infinity for the remaining  terms.
\end{proof}
\begin{lemma}\label{bounds_outside_b4} Let $g\in H^s(\mathbb R)$ with $s\geq 3,$  we have  the next bound
	\begin{equation}\label{bound_outside_b4}
		\bigg|PV\int_{\abs{\alpha}>1}\frac{1}{\alpha}B_4(x,\alpha)d\alpha\bigg|\leq c\,(1+\norm{g}_{C^1})^2.
	\end{equation}
\end{lemma}
\begin{proof}
Using the definition (\ref{terms_bs})
\[B_4(x,\alpha)=3\big[-2K(x,\alpha)^3+8(\Delta_{\alpha}h)^2K(x,\alpha)^4\big]					\partial_x\Delta_{\alpha}h.\]
We expand the terms $\partial_x\Delta_{\alpha}h$ and $(\Delta_{\alpha}h)^2$ in $B_4(x,\alpha)$ to obtain the following decomposition
\[B_4(x,\alpha)=\Psi(x,\alpha)-12K(x,\alpha)^3+48(\Delta_{\alpha}f)^2K(x,\alpha)^4\]
where  
\[\Psi(x,\alpha):=96\Delta_\alpha f\Delta_\alpha g K(x,\alpha)^4-6\partial_x\Delta_\alpha g K(x,\alpha)^3+24\partial_x\Delta_\alpha g(\Delta_\alpha h)^2 K(x,\alpha)^4.\]
We note that 
\[\abs{\Psi(x,\alpha)}\leq c\,\big(\norm{g}_{C^1}+\norm{g}_{C^1}^2\big)\,\abs{\alpha}^{-1},\]
hence $\abs{\Psi(x,\alpha)}/\alpha$ is integrable for $\abs{\alpha}>1.$ For the  remaining terms in the decomposition,  we follow the proofs of \cref{bound_k3} and \cref{bounds_outside_kn}.
\end{proof}
\begin{lemma}\label{bounds_outside_gamma_n} Let $g\in H^s(\mathbb R)$ with $s\geq 3,$  then
	\begin{equation}\label{bound_outside_gamma}
		\bigg|PV\int_{\abs{\alpha}>1}\frac{1}{\alpha}\Gamma(x,\alpha)d\alpha\bigg|\leq c\,(1+\norm{g}_{L^\infty}).
	\end{equation}
\end{lemma}
\begin{proof}Using  the identity (\ref{Gamma_dfn}), we decompose the integral in two terms
		\begin{equation}\label{Gamma_decomposition}\int_{\abs{\alpha}>1}\frac{1}{\alpha}\Gamma(x,\alpha)d\alpha=\int_{\abs{\alpha}>1}\frac{1}{\alpha}\Gamma_1(x,					\alpha)d\alpha+\int_{\abs{\alpha}>1}\frac{1}{\alpha}\Gamma_2(x,\alpha)d\alpha,\end{equation}
for 
	\begin{equation*}
		\begin{split}
				\Gamma_1(x,\alpha)&:=-2(\Delta_{\alpha}f)^2\big[F(\Delta_{\alpha}h)^3F(\Delta_{\alpha}f)+F(\Delta_{\alpha}h)^2F(\Delta_{\alpha}f)^2+F(\Delta_{\alpha}h)F(\Delta_{\alpha}f)^3\big],\\
				\Gamma_2(x,\alpha)&:=-\Delta_{\alpha}g\Delta_{\alpha}f\big[F(\Delta_{\alpha}h)^3F(\Delta_{\alpha}f)+F(\Delta_{\alpha}h)^2F(\Delta_{\alpha}f)^2+F(\Delta_{\alpha}h)F(\Delta_{\alpha}f)^3\big].	
		\end{split}
	\end{equation*}
Notice 
	\[\abs{\Gamma_2(x,\alpha)}\leq 2\norm{g}_{L^\infty}\abs{\alpha}^{-1},\]
then the second  integral in (\ref{Gamma_decomposition}) is bounded. While for the first one, we proceed in a similar way to (\ref{good_decomposition_outside}) by adding and  subtracting  $F(\Delta_{\alpha}f).$ Then we have
	\begin{equation*}
		\begin{split}		
			(\Delta_\alpha f)^2 F(\Delta_{\alpha }h)^3 F(\Delta_{\alpha}f)&=(\Delta_\alpha f)^2F(\Delta_\alpha f)^2\big[F(\Delta_\alpha                                                        h)-F(\Delta_\alpha f)\big]F(\Delta_\alpha f)\\
			&+(\Delta_\alpha f)^2F(\Delta_\alpha f)\big[F(\Delta_\alpha h)-F(\Delta_\alpha f)\big]F(\Delta_\alpha f)^2\\
			&+(\Delta_\alpha f)^2\big[F(\Delta_\alpha                                                        h)-F(\Delta_\alpha f)\big]F(\Delta_\alpha f)^3+	(\Delta_\alpha f)^2F(\Delta_\alpha f)^4.
		\end{split}
	\end{equation*}
Using the estimate (\ref{lip_f}) and \cref{hilbert_truncated} we obtain that 
		\[\bigg|\int_{\abs{\alpha}>1}\frac{1}{\alpha}(\Delta_{\alpha}f)^2 F(\Delta_{\alpha}h)^3F(\Delta_\alpha f)d\alpha\bigg|\leq c\,\norm{g}_{L^\infty}+\bigg|\int_{\abs{\alpha}>1}\frac{1}{\alpha}(\Delta_{\alpha}f)^2F(\Delta_\alpha f)^4d\alpha\bigg|\leq c\,(1+\norm{g}_{L^\infty}).\]
The remaining terms in $\Gamma_1$ are bounded similarly and  this finishes the proof.
\end{proof}
\begin{lemma}\label{bounds_outside_teta_n} Let $g\in H^s(\mathbb R)$ with $s\geq 3,$  then
	\begin{equation}\label{bound_outside_teta}
		\bigg|PV\int_{\abs{\alpha}>1}\frac{1}{\alpha}\Theta(x,\alpha)d\alpha\bigg|\leq c\,(1+\norm{g}_{L^\infty})^2.
	\end{equation}
\end{lemma}
\begin{proof}
Using the identity (\ref{Teta_dfn}) we decompose in the next way 
\begin{equation}\label{Theta_decomposition}\int_{\abs{\alpha}>1}\frac{1}{\alpha}\Theta(x,\alpha)d\alpha:=\int_{\abs{\alpha}>1}\frac{1}{\alpha}\Theta_1(x,					\alpha)d\alpha+\int_{\abs{\alpha}>1}\frac{1}{\alpha}\Theta_2(x,\alpha)d\alpha,\end{equation}
for 
	\begin{equation*}\setlength{\jot}{10pt}
		\begin{split}
				\Theta_1(x,\alpha)&:=-2(\Delta_{\alpha}f)^4\big[F(\Delta_{\alpha}h)^3F(\Delta_{\alpha}f)+F(\Delta_{\alpha}h)^3F(\Delta_{\alpha}f)^2\\
&\hspace{4cm}+F(\Delta_{\alpha}h)^2F(\Delta_{\alpha}f)^3+ F(\Delta_{\alpha}h)F(\Delta_{\alpha}f)^4\big],\\
				\Theta_2(x,\alpha)&:=-\Delta_{\alpha}g(\Delta_{\alpha}f)^3\big[F(\Delta_{\alpha}h)^3F(\Delta_{\alpha}f)+F(\Delta_{\alpha}h)^3F(\Delta_{\alpha}f)^2\\
&\hspace{4cm}+F(\Delta_{\alpha}h)^2F(\Delta_{\alpha}f)^3+ F(\Delta_{\alpha}h)F(\Delta_{\alpha}f)^4\big].\\
		\end{split}
	\end{equation*}
Notice 
	\[\abs{\Theta_2(x,\alpha)}\leq c\,(1+\norm{\partial_xg}_{L^\infty})\norm{g}_{L^\infty}\abs{\alpha}^{-1},\]
then the second  integral in (\ref{Theta_decomposition}) is bounded. While for $\Theta_1$ we proceed in a similar way to $\Gamma_1$ in the previous lemma. By  adding and  subtracting  $F(\Delta_{\alpha}f),$ we find that 
\[\Theta_1(x,\alpha)=-2(\Delta_\alpha f)^4F(\Delta_\alpha h)^4\big[F(\Delta_\alpha h)-F(\Delta_\alpha f)\big]F(\Delta_\alpha f)+\tilde \Theta(x,\alpha)+c\,(\Delta_{\alpha}f)^4F(\Delta_\alpha f)^5,\]
where 
\[\abs{\tilde \Theta(x,\alpha)}\leq c\norm{g}_{L^\infty}\abs{\alpha}^{-1}.\]
We compute directly
	\[F(\Delta_{\alpha}h)-F(\Delta_\alpha f)=-\Delta_{\alpha}g(2\Delta_\alpha f+\Delta_{\alpha}g)F(\Delta_{\alpha}h)F(\Delta_{\alpha }f).\]
Then expanding the sum we obtain that
	\begin{equation*}
		\begin{split}
			\Big|-2(\Delta_\alpha f)^4F(\Delta_\alpha h)^4\big[F(\Delta_\alpha h)-F(\Delta_\alpha f)\big]F(\Delta_\alpha f)\Big|&\leq \big|2(\Delta_\alpha f)^5\Delta_{\alpha}g F(\Delta_\alpha h)^3 F(\Delta_\alpha f)^2\big|\\
&\hspace{-4cm}+\big|2(\Delta_\alpha f)^4(\Delta_{\alpha}g)^2 F(\Delta_\alpha h)^3 F(\Delta_\alpha f)^2\big|\\
&\hspace{-4cm}\leq c\,\norm{g}_{L^\infty}\abs{\alpha}^{-1}+c\,\norm{g}_{L^\infty}^2\abs{\alpha}^{-2}
		\end{split}
	\end{equation*}
and therefore 
\[\bigg|\int_{\abs{\alpha}>1}\frac{1}{\alpha}\Theta_1(x,\alpha)d\alpha\bigg|\leq c\,(1+\norm{g}_{L^\infty})^2,\]
which completes the proof.
\end{proof}
\section{Regularization}\label{regularization}
In this section we regularize the equation (\ref{equation_g}), via mollifiers.  We consider a function $\chi\in C_c^\infty(\mathbb R)$ that satisfies 
\[\int_{\mathbb R}\chi(x) \,dx=1,\quad \chi(\abs{x})=\chi(x)\quad\mbox{and}\quad \chi\geq 0.\] 
For every $\epsilon>0$ we define $\chi_\epsilon(x)=\epsilon^{-1}\chi(\epsilon^{-1}x).$  We denote  the convolution by 
    \[\chi_\epsilon g(x):=(\chi_\epsilon\ast g)(x)=\int_{\mathbb R}\chi_\epsilon(x-y)g(y)dy.\] 
Throughout the section  we use the next properties of mollifiers
	\begin{equation}\label{properties_mol}
		\begin{split}
			&\norm{\chi_\epsilon \partial_x^kg}_{L^\infty},\,\norm{\chi_\epsilon\partial_x^k g}_{L^2}\leq c(\epsilon)\norm{g}_{L^2},\\
			& \partial_x^s \chi_\epsilon g=\chi_\epsilon \partial_x^sg,\\
			& \norm{\chi_\epsilon g-g}_{H^{s-1}}\leq c\,\epsilon\norm{g}_{H^s}.
			\end{split}
	\end{equation}
\noindent Now  we define the regularized system as follows
    \begin{equation}\label{regularized_g}
        \begin{split}
        &M^\epsilon(g^\epsilon):=\chi_\epsilon\int_{\mathbb R}\partial_x\Delta_\alpha(\chi_\epsilon g^\epsilon)(x)\,K^\epsilon(x,\alpha)d\alpha\hspace{0cm}+\chi_\epsilon\int_{\mathbb R}\Delta_\alpha(\chi_\epsilon g^\epsilon)(x)\,G^\epsilon(x,\alpha)d \alpha,\\
        &g^\epsilon(x,0)=g_0(x),
        \end{split}
    \end{equation}
   where the regularized kernels are defined by
    \begin{equation}\label{defn_kers_reg}\setlength{\jot}{12pt}
        \begin{split}
        &K^\epsilon(x,\alpha):=\frac{1}{1+(\Delta_\alpha (\chi_\epsilon g^\epsilon)+\Delta_\alpha f)^2},\\
        &G^\epsilon(x,\alpha):= -\frac{2\Delta_\alpha f+\Delta_\alpha( \chi_\epsilon g^\epsilon)}{(1+(\Delta_\alpha (\chi_\epsilon g^\epsilon)+\Delta_{\alpha}f)^2)(1+(\Delta_\alpha f)^2)}.\\
      \end{split}
    \end{equation}
In the next lemma we apply the Picard theorem to the regularized system (\ref{regularized_g}), where we consider the open set  $\mathcal O\subset H^s(\mathbb R)$ defined by $\mathcal O=\{g\in H^s(\mathbb R):\norm{g}_{H^s}<c\}$ for $s\geq 3.$
    \begin{lemma}\label{lemma_picard_reg} 
        Let $\epsilon>0,$ then there exists  a  time $T_\epsilon>0$ and a  solution $g^\epsilon(x,t)\in C^1([ 0, T_\epsilon]:\mathcal O)$ to the regularized system  (\ref{regularized_g})  such that $g^\epsilon(x,0)=g_0(x)$ for $s\geq 3.$
    \end{lemma}

  \begin{proof} Take $g_1,g_2\in \mathcal O\subset H^s(\mathbb R).$ We define the  auxiliary operator 
            \[\mathfrak M^\epsilon(g)(x):=\int_{\mathbb R}\partial_x\Delta_\alpha(\chi_\epsilon g^\epsilon)(x)K^\epsilon(x,\alpha)d\alpha+\int_{\mathbb R}\Delta_\alpha(\chi_\epsilon g^\epsilon)(x)G^\epsilon(x,\alpha)d \alpha.\]
            We observe that  $M^\epsilon=\chi_\epsilon\ast\mathfrak M^\epsilon.$  By applying the triangle inequality we have 
    \begin{equation*}
        \begin{split}
            \norm{\mathfrak M^\epsilon(g_1)-\mathfrak M^\epsilon(g_2)}_{L^2}&\leq \norm{R_{1}}_{L^2} +\norm{R_2}_{L^2},
            \end{split}
    \end{equation*}
for 
	\begin{equation*}\setlength{\jot}{12pt}
		\begin{split}		
			R_1(x)&:=\int_{\mathbb R} \partial_x\Delta_{\alpha}(\chi_\epsilon g_1)K_1^\epsilon (x,\alpha)d\alpha-  \int_{\mathbb R} \partial_x\Delta_{\alpha}(\chi_\epsilon g_2)K_2^\epsilon (x,\alpha)d\alpha,\\
			R_2(x)&:= \int_{\mathbb R} \Delta_{\alpha}(\chi_\epsilon g_1)G_1^\epsilon (x,\alpha)d\alpha- \int_{\mathbb R} \Delta_{\alpha}(\chi_\epsilon g_2)G_2^\epsilon (x,\alpha)d\alpha,
		\end{split}
	\end{equation*}
    where $K_i^\epsilon(x,\alpha)$ and $G_i^\epsilon(x,\alpha) $ are the respective kernels for the functions $g_1$ and $g_2.$ For $R_1,$ we note that by  adding and subtracting $\partial_x\Delta_{\alpha}(\chi_\epsilon g_2)K_1^\epsilon(x,\alpha),$ we find that 
    \begin{equation*}
        \begin{split}
            R_1(x)&=\int_{\mathbb R}\Big[\partial_x\Delta_{\alpha}(\chi_\epsilon  g_1)-\partial_x\Delta_{\alpha}(\chi_\epsilon g_2)\Big]K_1^\epsilon (x,\alpha)d\alpha\hspace{0cm}-\int_{\mathbb R} \partial_x\Delta_{\alpha}(\chi_\epsilon g_2)\Big[K_2^\epsilon (x,\alpha)-K_1^\epsilon(x,\alpha)\Big]d\alpha.
        \end{split}
    \end{equation*}
We have  the following identities
    \begin{equation}\label{differences_reg}\setlength{\jot}{12pt}
        \begin{split}
 &           \partial_x\Delta_\alpha(\chi_\epsilon g_1)-\partial_x\Delta_\alpha(\chi_\epsilon g_2)=\frac{1}{\alpha}\chi_\epsilon(\partial_xg_1(x)-\partial_xg_2(x))-\frac{1}{\alpha}\chi_\epsilon(\partial_xg_1(x-\alpha)-\partial_xg_2(x-\alpha)),\\
&           \partial_x\Delta_\alpha (\chi_\epsilon g_2)\big[ K_2^\epsilon(x,\alpha)-K_1^\epsilon(x,\alpha)\big]=\Big[\Delta_\alpha(\chi_\epsilon g_1)-\Delta_\alpha(\chi_\epsilon g_2)\Big]B_\epsilon(x,\alpha),\\
&B_\epsilon(x,\alpha)=\partial_x\Delta_\alpha (\chi_\epsilon g_2)(2\Delta_\alpha f+\chi_\epsilon\Delta_{\alpha}g_1+\chi_\epsilon \Delta_{\alpha}g_2)K_1^\epsilon(x,\alpha)K_2^\epsilon(x,\alpha).
        \end{split}
    \end{equation}
Using the formulas (\ref{differences_reg}), we obtain the next decomposition
    \begin{equation*}\setlength{\jot}{12pt}
        \begin{split}
                R_1(x) & \hspace{0cm}=\chi_\epsilon[\partial_xg_1(x)-\partial_xg_2(x)]\int_{\mathbb R}\frac{1}{\alpha}K_1^\epsilon(x,\alpha)d\alpha+ \chi_\epsilon[g_1(x)-g_2(x)]\int_{\mathbb R}\frac{1}{\alpha}B_\epsilon(x,\alpha)d\alpha\\
             & \hspace{0cm}+\int_{\mathbb R}\frac{\chi_\epsilon[\partial_xg_1(x-\alpha)-\partial_xg_2(x-\alpha)]}{\alpha}K_1^\epsilon(x,\alpha)d\alpha\\
&+ \int_{\mathbb R}\frac{\chi_\epsilon[g_1(x-\alpha)-g_2(x-\alpha)]}{\alpha}B_\epsilon(x,\alpha)d\alpha\\
&:=T_1(x)+T_2(x)+T_3(x)+T_4(x).
        \end{split}
    \end{equation*}
   We use the  estimate (\ref{first_estimates_order_one}) in \cref{aux_main_lemma1} to get a bound for  $T_1$.
 Now, we use the properties (\ref{properties_mol})  to obtain   \[\norm{T_1}_{L^2}\leq c(\norm{g_1}_{L^2},\epsilon)\,\norm{g_1-g_2}_{L^2}.\]
For  $T_2$ we decompose  the  integral in the next way
	\begin{equation*}
		\begin{split}		
		PV\int_{\mathbb R}\frac{1}{\alpha}B_{\epsilon}(x,\alpha)d\alpha &:=Q_1(x)+Q_2(x)+Q_3(x),
		\end{split}
	\end{equation*}
for 
\begin{equation*}
	\begin{split}
&	Q_1(x)=2\int_{\mathbb R} \frac{1}{\alpha} \Delta_{\alpha}f\cdot\partial_x\Delta_\alpha (\chi_\epsilon g_2)\cdot K_1^\epsilon(x,\alpha)K_2^\epsilon(x,\alpha) d\alpha, \\
 &Q_2(x) =  \int_{\mathbb R}\frac{1}{\alpha} \Delta_{\alpha}(\chi_\epsilon g_1)\cdot \partial_x\Delta_\alpha (\chi_\epsilon g_2)\cdot K_1^\epsilon(x,\alpha)K_2^\epsilon(x,\alpha) d\alpha, \\
& Q_3(x)=  \int_{\mathbb R}\frac{1}{\alpha} \Delta_{\alpha}(\chi_\epsilon g_2)\cdot \partial_x\Delta_{\alpha}(\chi_\epsilon g_2)\cdot K_1^\epsilon(x,\alpha)K_2^\epsilon(x,\alpha) d\alpha.
		\end{split}
	\end{equation*}
Using the next estimates
	\begin{equation*}
		\begin{split}
			&\abs{\Delta_{\alpha}(\chi_\epsilon g_i) }\leq \norm{\chi_\epsilon g_i}_{L^\infty}\abs{\alpha}^{-1}\quad\mbox{for}\quad i=1,2,\\ 				&	\abs{\partial_x\Delta_{\alpha}(\chi_\epsilon  g_2)}\leq \norm{\chi_\epsilon \partial_xg_2}_{L^\infty}\abs{\alpha}^{-1} 
		\end{split}
	\end{equation*}
and  the next bound
\[\abs{\Delta_{\alpha }f\, K_{1}^\epsilon(x,\alpha)}\leq \abs{(\Delta_{\alpha}f+\Delta_{\alpha}(\chi_\epsilon g_1)) K_1^\epsilon(x,\alpha)}+\abs{ \Delta_{\alpha}(\chi_\epsilon g_1)\cdot K_1^\epsilon(x,\alpha)}\leq c\,(1+\norm{\chi_\epsilon \partial_x g_1}_{L^\infty})\]
we find that 
\[\big|Q_1(x)^{out}+Q_2(x)^{out}+Q_3(x)^{out}\big|\leq c\,\norm{\chi_\epsilon \partial_xg_2}_{L^\infty}(1+\norm{\chi_{\epsilon} g_1}_{L^\infty}+\norm{\chi_{\epsilon} g_2}_{L^\infty}+ \norm{\partial_x\chi_\epsilon g_1}_{L^\infty}).\]
For the \emph{in} part, we decompose  $Q_2(x)^{in}$ by adding and subtracting $\chi_\epsilon \partial_xg_1(x),\chi_\epsilon \partial_x^2g_2(x), K_1^\epsilon(x,0)$ and $K_2^\epsilon(x,0),$ then we obtain
\begin{equation*}\setlength{\jot}{1pt}
	\begin{split}	
		Q_2(x)^{in}&=\int_{\abs{\alpha}<1}\frac{1}{\alpha}\big[ \Delta_{\alpha}\chi_\epsilon g_1-\chi_{\epsilon}\partial_xg_1(x)\big]\partial_x\Delta_\alpha (\chi_\epsilon g_2)K_1^\epsilon(x,\alpha)K_2^\epsilon(x,\alpha)d\alpha\\
&+\chi_{\epsilon}\partial_xg_1(x)\int_{\abs{\alpha}<1}\frac{1}{\alpha}\big[\partial_x\Delta_\alpha (\chi_\epsilon g_2)-\chi_\epsilon \partial_x^2g_2(x)\big]K_1^\epsilon(x,\alpha)K_2^\epsilon(x,\alpha)d\alpha \\
&+\chi_{\epsilon}\partial_xg_1(x)\chi_\epsilon \partial_x^2g_2(x)\int_{\abs{\alpha}<1}\frac{1}{\alpha}\big[K_1^\epsilon(x,\alpha)-K_1^\epsilon(x,0)\big]K_2^\epsilon(x,\alpha)d\alpha\\
&+\chi_{\epsilon}\partial_xg_1(x)\chi_\epsilon \partial_x^2g_2(x)K_1^\epsilon(x,0)\int_{\abs{\alpha}<1}\frac{1}{\alpha}\big[K_2^\epsilon(x,\alpha)-K_2^\epsilon(x,0)\big]d\alpha,
	\end{split}
\end{equation*}
where the regularized kernels at zero are 
		\begin{equation*}
			\begin{split}
			&K_1^\epsilon(x,0)=\frac{1}{1+(\partial_xf(x)+\chi_\epsilon\partial_xg_1(x))^2},\\
			&K_2^\epsilon(x,0)=\frac{1}{1+(\partial_xf(x)+\chi_\epsilon\partial_xg_2(x))^2}.
		\end{split}
		\end{equation*}
In a similar way to (\ref{second_order_g}) and (\ref{second_order_h}) we have the following inequalities
	\begin{equation}\label{reg_eq_bounds}
		\begin{split}	
			&\abs{\Delta_{\alpha}\chi_{\epsilon}g_1-\chi_{\epsilon}\partial_xg_1(x)}\leq c\,\norm{\chi_{\epsilon}\partial_x^2g_1}_{L^\infty}\abs{\alpha},\\
&\abs{\partial_x\Delta_{\alpha}\chi_{\epsilon} g_2-\chi_{\epsilon}\partial_x^2g_2(x)}\leq c\,\norm{(\partial_x\chi_\epsilon)\partial_x^2g_2}_{L^\infty}\abs{\alpha},\\
  &\abs{K_1^\epsilon(x,\alpha)-K_1^\epsilon(x,0)}\leq c\,(1+\norm{\chi_\epsilon\partial_x^2g_1}_{L^\infty})\,\abs{\alpha},\\
	  &\abs{K_2^\epsilon(x,\alpha)-K_2^\epsilon(x,0)}\leq c\,(1+\norm{\chi_\epsilon\partial_x^2g_2}_{L^\infty})\,\abs{\alpha}.
		\end{split}
		\end{equation}
Hence, we deduce the following
\begin{equation*}
	\begin{split}	
	\big| Q_2(x)^{in}\big|&\leq c\,\Big(\norm{\chi_{\epsilon}\partial_x^2g_1}_{L^\infty}\norm{\chi_{\epsilon}\partial_x^2g_2}_{L^\infty}+\norm{\chi_{\epsilon}\partial_xg_1}_{L^\infty}\norm{(\partial_x\chi_\epsilon)\partial_x^2g_2}_{L^\infty}\\
&+\norm{\chi_{\epsilon}\partial_xg_1}_{L^\infty}\norm{\chi_{\epsilon}\partial_x^2g_2}_{L^\infty}(1+\norm{\chi_{\epsilon}\partial_x^2g_1}_{L^\infty}+\norm{\chi_{\epsilon}\partial_x^2g_2}_{L^\infty})\Big).
\end{split}
	\end{equation*}
Similarly to the last term, we derive that 
	\begin{equation*}
		\begin{split}	
			\big|Q_3(x)^{in}\big|&\leq  c\,\Big(\norm{\chi_{\epsilon}\partial_x^2g_2}_{L^\infty}\norm{\chi_{\epsilon}\partial_x^2g_2}_{L^\infty}+\norm{\chi_{\epsilon}\partial_xg_2}_{L^\infty}\norm{(\partial_x\chi_\epsilon)\partial_x^2g_2}_{L^\infty}\\
&+\norm{\chi_{\epsilon}\partial_xg_2}_{L^\infty}\norm{\chi_{\epsilon}\partial_x^2g_2}_{L^\infty}(1+\norm{\chi_{\epsilon}\partial_x^2g_1}_{L^\infty}+\norm{\chi_{\epsilon}\partial_x^2g_2}_{L^\infty})\Big).
		\end{split}
	\end{equation*}
We recall the definition of the auxiliary function 
\[F(x)=\frac{1}{1+x^2}\]
then we decompose  $Q_1(x)^{in}$ by adding and subtracting  $\chi_\epsilon\partial_x^2 g_2(x)$ and $F(\Delta_{\alpha}f).$ We take  $Q_1(x)^{in}:=  \mathfrak I_1(x)+ \mathfrak I_2(x)+\mathfrak I_3(x)+\mathfrak I_4(x)$ for 
	\begin{equation}\label{decomposition_in_B_epsilon}
	\begin{split}
		&\mathfrak I_1(x)=\int_{\abs{\alpha}<1}\frac{1}{\alpha}[\partial_x\Delta_{\alpha}\chi_\epsilon g_2- \chi_\epsilon\partial_x^2 g_2(x)]\Delta_{\alpha}fK_{1}^\epsilon(x,\alpha)K_2^\epsilon(x,\alpha)d\alpha,\\
		&\mathfrak I_2(x)=\chi_\epsilon\partial_x^2 g_2(x)\int_{\abs{\alpha}<1}\frac{1}{\alpha}\Delta_{\alpha}f\big[K_{2}^\epsilon(x,\alpha)-F(\Delta_\alpha f)\big]K_1^\epsilon(x,\alpha )d\alpha,\\
		&\mathfrak I_3(x)=\chi_\epsilon\partial_x^2 g_2(x)\int_{\abs{\alpha}<1}\frac{1}{\alpha}\Delta_{\alpha}fF(\Delta_{\alpha}f)\big[K_1^\epsilon(x,\alpha)-F(\Delta_\alpha f)\big]d\alpha,\\
		&\mathfrak I_4(x)=\chi_\epsilon\partial_x^2 g_2(x)\int_{\abs{\alpha}<1}\frac{1}{\alpha} \Delta_{\alpha}fF(\Delta_{\alpha}f)^2d\alpha.\\
	\end{split}
	\end{equation}
A direct computation yields to
	\begin{equation*}
	K_1^\epsilon(x,\alpha)-F(\Delta_{\alpha}f)=-\Delta_{\alpha}\chi_{\epsilon}g_1(2\Delta_{\alpha}f+\Delta_{\alpha}\chi_\epsilon g_1)K_1^\epsilon(x,\alpha)F(\Delta_{\alpha}f).\end{equation*}
Now, we decompose   $\mathfrak I_3(x)$  by adding and subtracting $\chi_\epsilon\partial_x g_1(x)$ and $K_1^\epsilon(x,0),$ then  we obtain that 
	\begin{equation*}
		\begin{split}
			\mathfrak I_3(x)&=-2\int_{\abs{\alpha}<1}\frac{1}{\alpha}
			(\Delta_{\alpha}f)^2\big[\Delta_\alpha\chi_\epsilon g_1-\chi_\epsilon\partial_x g_1(x)\big]K_1^\epsilon(x,\alpha) F(\Delta_\alpha f)^2d\alpha \\
&-2\chi_\epsilon\partial_x^2 g_2(x)\chi_\epsilon\partial_x g_1(x)\int_{\abs{\alpha}<1}\frac{1}{\alpha}
			(\Delta_{\alpha}f)^2\big[K_1^\epsilon(x,\alpha)-K_1^\epsilon(x,0)\big] F(\Delta_\alpha f)^2d\alpha \\
			&-2\chi_\epsilon\partial_x^2 g_2(x)\chi_\epsilon\partial_x g_1(x)K_1^\epsilon(x,0)\int_{\abs{\alpha}<1}\frac{1}{\alpha}(\Delta_{\alpha}f)^2F(\Delta_{\alpha}f)^2d\alpha\\
			&-\chi_\epsilon\partial_x^2 g_2(x)\int_{\abs{\alpha}<1}\frac{1}{\alpha}\Delta_{\alpha}f\big[\Delta_{\alpha}\chi_\epsilon g_1-\chi_\epsilon\partial_x g_1(x)\big]\Delta_{\alpha}\chi_\epsilon g_1K_1^\epsilon(x,						\alpha)F(\Delta_{\alpha}f)^2d\alpha\\
			&-\chi_\epsilon\partial_x^2 g_2(x) \chi_\epsilon\partial_x g_1(x)\int_{\abs{\alpha}<1}\frac{1}{\alpha}\Delta_{\alpha}f\big[\Delta_{\alpha}\chi_\epsilon g_1-\chi_\epsilon\partial_x g_1(x)\big] K_1^\epsilon(x,						\alpha)F(\Delta_{\alpha}f)^2 d\alpha\\
			&-\chi_\epsilon\partial_x^2 g_2(x)(\chi_\epsilon\partial_x g_1(x))^2\int_{\abs{\alpha}<1}\frac{1}{\alpha}\Delta_{\alpha}f\big[K_1^\epsilon(x,\alpha)-K_1^\epsilon(x,0)\big]F(\Delta_\alpha f)^2d\alpha \\
			&-\chi_\epsilon\partial_x^2 g_2(x)(\chi_\epsilon\partial_x g_1(x))^2K_1^\epsilon(x,0)\int_{\abs{\alpha}<1}\frac{1}{\alpha}\Delta_{\alpha}fF(\Delta_{\alpha}f)^2d\alpha.\\
		\end{split}
	\end{equation*}
Hence, using the estimates (\ref{reg_eq_bounds}) we find that 
		\begin{equation}\label{mixed_terms_b_21}\setlength{\jot}{11pt}
			\begin{split}
			\big|\mathfrak I_3(x)\big| &\leq c\,\Big(\norm{\chi_{\epsilon}\partial_x^2g_1}_{L^\infty}+(1+\norm{\chi_{\epsilon}\partial_x^2g_1}_{L^\infty})(\norm{\chi_{\epsilon}\partial_xg_1}_{L^\infty}+\norm{\chi_{\epsilon}\partial_xg_1}_{L^\infty}^2)\Big).\\
			\end{split}		
		\end{equation}
For the second  term in   (\ref{decomposition_in_B_epsilon})  we add and subtract $F(\Delta_\alpha f )$, and hence  
	\begin{equation}\label{bound_in_part_mixed_terms}\setlength{\jot}{11pt}
		\begin{split}
			\mathfrak I_2(x)&=\chi_\epsilon\partial_x^2 g_2(x)\int_{\abs{\alpha}<1}\frac{1}{\alpha}\Delta_{\alpha}f\big[ K_2^\epsilon (x,\alpha)-F(\Delta_{\alpha}f)\big] \big[K_1^\epsilon(x,\alpha)-F(\Delta_{\alpha}f)]d\alpha\\
			&+\chi_\epsilon\partial_x^2 g_2(x)\int_{\abs{\alpha}<1}\frac{1}{\alpha}\Delta_{\alpha}fF(\Delta_{\alpha}f)[K_2^\epsilon(x,\alpha)-F(\Delta_{\alpha}f)]d\alpha\\
& :=\mathfrak I_{2,1}(x)+\mathfrak I_{2,2}(x).
		\end{split}
	\end{equation}
The  term $\mathfrak I_{2,2}(x)$ in the last decomposition (\ref{bound_in_part_mixed_terms}) can be bounded in similar way to (\ref{mixed_terms_b_21}). While for the first one,  we observe that 
	\begin{equation}\label{long_decomposition}\setlength{\jot}{11pt}
		\begin{split}
		\mathfrak I_{2,1}(x)&=\chi_\epsilon\partial_x^2 g_2(x)\int_{\abs{\alpha}<1}\frac{1}{\alpha}\Delta_{\alpha}f K_1^\epsilon(x,\alpha)K_2^\epsilon(x,\alpha)d\alpha\\
& -\chi_\epsilon\partial_x^2 g_2(x)\int_{\abs{\alpha}<1}\frac{1}{\alpha}\Delta_{\alpha}f K_1^\epsilon(x,\alpha)F(\Delta_{\alpha}f)d\alpha\\
&-\chi_\epsilon\partial_x^2 g_2(x)\int_{\abs{\alpha}<1}\frac{1}{\alpha}\Delta_{\alpha}f K_2^\epsilon(x,\alpha)F(\Delta_{\alpha}f)d\alpha\\
& +\chi_\epsilon\partial_x^2 g_2(x)\int_{\abs{\alpha}<1}\frac{1}{\alpha}\Delta_{\alpha}f F(\Delta_{\alpha}f)^2d\alpha:= 	\mathfrak N_{1}(x)+\mathfrak N_{2}(x)+\mathfrak N_{3}(x)+\mathfrak N_{4}(x).
		\end{split}
	\end{equation}
The term $\mathfrak N_4(x)$ is bounded by lemma (\ref{hilbert_truncated}). For $\mathfrak N_2(x)$  we decompose 
by adding and subtracting $K_1^\epsilon(x,0)$  then we have 
	\begin{equation*}\setlength{\jot}{10pt}
		\begin{split}
	\mathfrak N_2(x)&=-\chi_\epsilon\partial_x^2 g_2(x)\int_{\abs{\alpha}<1}\frac{1}{\alpha}\Delta_{\alpha}f\big[K_1^\epsilon(x,\alpha)-K_1^\epsilon(x,0)\big]F(\Delta_\alpha f)d\alpha\\
&- \chi_\epsilon\partial_x^2 g_2(x) K_1^\epsilon(x,0)\int_{\abs{\alpha}<1}\frac{1}{\alpha}\Delta_{\alpha}fF(\Delta_{\alpha}f)d\alpha.\\
	\end{split}
	\end{equation*}
Using the estimates (\ref{reg_eq_bounds}) we find that 
\[\abs{\mathfrak N_2(x)}\leq c\,\norm{\chi_\epsilon \partial_x^2g_2}_{L^\infty}(1+\norm{\chi_\epsilon \partial_x^2g_1}_{L^\infty}).\]
Similarly we get
\[\abs{\mathfrak N_3(x)}\leq c\,\norm{\chi_\epsilon \partial_x^2g_2}_{L^\infty}(1+\norm{\chi_\epsilon \partial_x^2g_2}_{L^\infty}).\]
For the remaining  term in (\ref{long_decomposition}) we add and subtract ${F(\Delta_{\alpha}f), K_1^\epsilon(x,0),K_2^\epsilon(x,0)}$  and $\chi_\epsilon\partial_x g_1(x).$ We find that 
	\begin{equation*}
		\begin{split}
\mathfrak N_1(x)&=-2\chi_\epsilon\partial_x^2 g_2(x)\int_{\abs{\alpha}<1}\frac{1}{\alpha}\Delta_{\alpha}f\big[\Delta_{\alpha}\chi_\epsilon g_1-\chi_\epsilon \partial_xg_1(x)\big]K_1^\epsilon(x,\alpha)K_2^\epsilon(x,\alpha) F(\Delta_{\alpha}f)d\alpha\\
&-2\chi_\epsilon\partial_x^2 g_2(x)\chi_\epsilon\partial_x g_1(x)\int_{\abs{\alpha}<1}\frac{1}{\alpha}\Delta_{\alpha}f\big[K_1^\epsilon(x,\alpha)-K_1^\epsilon(x,0)\big]K_2^\epsilon(x,\alpha)F(\Delta_{\alpha}f)d\alpha \\
&-2\chi_\epsilon\partial_x^2 g_2(x)\chi_\epsilon\partial_x g_1(x)K_1^\epsilon(x,0)\int_{\abs{\alpha}<1}\frac{1}{\alpha}\Delta_{\alpha}f\big[K_2^\epsilon(x,\alpha)-K_2^\epsilon(x,0)\big]F(\Delta_{\alpha}f)d\alpha \\
& -2\chi_\epsilon\partial_x^2 g_2(x)\chi_\epsilon\partial_x g_1(x)K_1^\epsilon(x,0) K_2^\epsilon(x,0)\int_{\abs{\alpha}<1}\frac{1}{\alpha}\Delta_{\alpha}f F(\Delta_{\alpha}f)d\alpha\\
&-\chi_\epsilon\partial_x^2 g_2(x)\int_{\abs{\alpha}<1}\frac{1}{\alpha}\Delta_{\alpha}f\big[\Delta_{\alpha}\chi_\epsilon g_1-\chi_\epsilon\partial_x g_1(x)\big]\Delta_{\alpha}\chi_\epsilon g_1 K_1^\epsilon(x,\alpha)K_2^\epsilon(x,\alpha) F(\Delta_{\alpha}f)d\alpha\\
&-\chi_\epsilon\partial_x^2 g_2(x)\chi_\epsilon \partial_x g_1(x)\int_{\abs{\alpha}<1}\frac{1}{\alpha}\Delta_{\alpha}f[\Delta_{\alpha}\chi_\epsilon g_1-\chi_\epsilon \partial_xg_1(x)\big] K_1^\epsilon(x,\alpha)K_2^\epsilon(x,\alpha) F(\Delta_{\alpha}f)d\alpha\\
&-\chi_\epsilon\partial_x^2 g_2(x)(\chi_\epsilon \partial_x g_1(x))^2\int_{\abs{\alpha}<1}\frac{1}{\alpha}\Delta_{\alpha}f\big[K_1^\epsilon(x,\alpha)-K_1^\epsilon(x,0)\big]K_2^\epsilon(x,\alpha)F(\Delta_{\alpha}f)d\alpha\\
&-\chi_\epsilon\partial_x^2 g_2(x)(\chi_\epsilon \partial_x g_1(x))^2K_1^\epsilon(x,0)\int_{\abs{\alpha}<1}\frac{1}{\alpha}\Delta_{\alpha}f\big[K_2^\epsilon(x,\alpha)-K_2^\epsilon(x,0)\big]F(\Delta_{\alpha}f)d\alpha\\
&-\chi_\epsilon\partial_x^2 g_2(x)(\chi_\epsilon\partial_x g_1(x))^2K_1^\epsilon(x,0)K_2^\epsilon(x,0)\int_{\abs{\alpha}<1}\frac{1}{\alpha}\Delta_{\alpha}fF(\Delta_{\alpha}f)d\alpha\\
&+\chi_\epsilon\partial_x^2 g_2(x)\int_{\abs{\alpha}<1}\frac{1}{\alpha}\Delta_{\alpha}f F(\Delta_{\alpha}f)\big[K_1^\epsilon(x,\alpha)-K_1^\epsilon(x,0)\big]d\alpha\\
&+ \chi_\epsilon\partial_x^2 g_2(x)K_1^\epsilon(x,0)\int_{\abs{\alpha}<1}\frac{1}{\alpha}\Delta_{\alpha}f F(\Delta_{\alpha}f)d\alpha.\\
		\end{split}
	\end{equation*}
Using  the bounds (\ref{reg_eq_bounds}) we deduce the next estimate
\begin{equation*}\setlength{\jot}{12pt}
	\begin{split}
		\big|\mathfrak N_1(x)\big|
&\hspace{0cm}\leq c\,\Big\{1+\norm{\chi_{\epsilon}\partial_x^2g_1}_{L^\infty}+\norm{\chi_{\epsilon}\partial_xg_2}					_{L^\infty}\norm{\chi_{\epsilon}\partial_x^2g_1}_{L^\infty}\\
&\hspace{0cm}+(\norm{\chi_{\epsilon}\partial_x^2g_1}_{L^\infty}+\norm{\chi_{\epsilon}\partial_xg_1}_{L^\infty}^2)(1+\norm{\chi_{\epsilon}\partial_x^2g_2}					_{L^\infty}+\norm{\chi_{\epsilon}\partial_x^2g_1}_{L^\infty})\Big\}\norm{\chi_\epsilon\partial_x^2g_2}_{L^\infty}.
	\end{split}
\end{equation*}
The last inequality completes the estimate for the \emph{in} part $Q_1(x)^{in}$. Now, we use the   properties of mollifiers  (\ref{properties_mol}) and we  conclude that
\[\big|Q_1(x)^{in} \big|\leq c(\epsilon)\,(1+\norm{g_1}_{L^2})^3(1+\norm{g_2}_{L^2})^3\norm{g_2}_{L^2}.\] 
Therefore \[\norm{T_2}_{L^2}\leq c\,(\norm{g_1}_{L^2},\norm{g_2}_{L^2},\epsilon)\,\norm{g_1-g_2}_{L^2}.\]
Now we move to $T_3,$  for the  \emph{out} part  using the Cauchy-Schwarz inequality with respect to $\alpha,$ we find the following bound
	\begin{equation*}\setlength{\jot}{12pt}
		\begin{split}
			\norm{T_3^{out}}_{L^2}&\leq \norm{\chi_\epsilon(\partial_x g_1-\partial_xg_2)}_{L^2}\bigg(\int_{\abs{\alpha}>1}\frac{1}{\alpha^2}\int_{\mathbb R}K_1^\epsilon(x,\alpha)^2dxd\alpha\bigg)^{1/2}
		\end{split}
	\end{equation*}
which is enough to control the \emph{out} part. For the \emph{in} part we  add and subtract the term $K_1^\epsilon(x,0).$ This leads to the next decomposition
\begin{equation*}
		\begin{split}	
			\int_{\abs{\alpha}<1}\frac{\chi_\epsilon(\partial_xg_1(x-\alpha)-\partial_xg_2(x-\alpha))}{\alpha} K_1^\epsilon(x,							\alpha)d\alpha&\\
			&\hspace{-3cm}= \int_{\abs{\alpha}<1}\chi_\epsilon(\partial_xg_1(x-\alpha)-\partial_xg_2(x-\alpha))\frac{1}{\alpha}						\bigg[K_1^\epsilon(x,\alpha)-K_1^\epsilon(x,0)\bigg]d\alpha\\
			&\hspace{-3cm}+K_1^\epsilon(x,0)\int_{\abs{\alpha}<1}\frac{\chi_\epsilon(\partial_xg_1(x-\alpha)-\partial_xg_2(x-\alpha))}{\alpha}d\alpha.
\end{split}
\end{equation*}
From the above, a truncated Hilbert transform  arises. Applying the Minkowski's integral inequality and using the estimates (\ref{reg_eq_bounds}) we obtain that  
	\begin{equation*}\setlength{\jot}{12pt}
		\begin{split}
			\norm{T_{3}^{in}}_{L^2}&\hspace{0cm}\leq \bigg\lVert\int_{\abs{\alpha}<1} \chi_\epsilon(\partial_xg_1(x-\alpha)-\partial_xg_2(x-\alpha))\frac{ K_1^\epsilon(x,\alpha)-K_1^\epsilon(x,0)}{\alpha}d\alpha\bigg\rVert_{L^2}\\
&+\big\lVert K_1^\epsilon(x,0) H_{\abs{\alpha}<1} \chi_\epsilon(\partial_xg_1-\partial_xg_2)(x)\big\rVert_{L^2}\\
&\leq c\int_{\abs{\alpha}<1}(1+\norm{\chi_\epsilon\partial_x^2g_1}_{L^\infty})\bigg(\int_{\mathbb R} \chi_\epsilon\big[\partial_xg_1(x-\alpha)-\partial_xg_2(x-\alpha)\big]^2dx \bigg)^{1/2}d\alpha\\
&+\norm{K_1^\epsilon(x,0)}_{L^\infty}\norm{\chi_\epsilon(\partial_xg_1-\partial_xg_2)}_{L^2}\\
&\leq c\,(\norm{\chi_\epsilon \partial_x^2g_1}_{L^\infty},\epsilon ) \norm{\chi_\epsilon(\partial_xg_1-\partial_xg_2)}_{L^2}.
		\end{split}
	\end{equation*}
We use  the properties of mollifiers (\ref{properties_mol})  to conclude that  
\[\norm{T_3}_{L^2}\leq c\,(\norm{g_1}_{L^2},\epsilon)\norm{g_1-g_2}_{L^2}.\]
For $T_4$ we expand the sum in $B_\epsilon(x,\alpha),$ see the definitions (\ref{differences_reg}), and we repeat the argument  used in $T_3.$ We have the next decomposition
	\begin{equation*}
		\begin{split}
		B_{\epsilon}(x,\alpha)&=2\Delta_{\alpha}f K_{1}^\epsilon(x,\alpha)K_2^\epsilon(x,\alpha)\partial_x\Delta_{\alpha}(\chi_\epsilon g_2)+\Delta_{\alpha}(\chi_{\epsilon}g_1) K_{1}^\epsilon(x,\alpha)K_2^\epsilon(x,\alpha)\partial_x\Delta_{\alpha}(\chi_\epsilon g_2)\\
&\hspace{5cm}+\Delta_{\alpha}(\chi_{\epsilon}g_2 )K_{1}^\epsilon(x,\alpha)K_2^\epsilon(x,\alpha)\partial_x\Delta_{\alpha}(\chi_\epsilon g_2).
		\end{split}
	\end{equation*}
For the second term in $B_\epsilon(x,\alpha)$ we add and subtract the terms $\chi_\epsilon \partial_xg_1(x),\chi_\epsilon \partial_x^2g_2(x), K_1^\epsilon(x,0)$ and $K_2^\epsilon(x,0)$ in order to obtain
	\begin{equation*}
		\begin{split}
\partial_x\Delta_{\alpha}(\chi_\epsilon g_2)\Delta_{\alpha}(\chi_{\epsilon}g_1) K_{1}^\epsilon(x,\alpha)K_2^\epsilon(x,\alpha)&=\big[\partial_x\Delta_{\alpha}\chi_\epsilon g_2- \chi_\epsilon \partial_x^2g_2(x)\big]\Delta_{\alpha}g_1 K_{1}^\epsilon(x,\alpha)K_2^\epsilon(x,\alpha)\\			
			&+\chi_\epsilon \partial_x^2g_2(x)\big[\Delta_{\alpha}\chi_{\epsilon}g_1 -\chi_\epsilon \partial_xg_1(x)\big]K_{1}^\epsilon(x,\alpha)K_2^\epsilon(x,\alpha)\\
			&+\chi_\epsilon \partial_x^2g_2(x)\chi_\epsilon \partial_xg_1(x)\big[K_{1}^\epsilon(x,\alpha)-K_1^\epsilon(x,0)]K_2^\epsilon(x,\alpha)\\
			&+\chi_\epsilon \partial_x^2g_2(x)\chi_\epsilon \partial_xg_1(x)K_1^\epsilon(x,0)\big[K_{2}^\epsilon(x,\alpha)-K_2^\epsilon(x,0)]\\
			&+\chi_\epsilon \partial_x^2g_2(x)\chi_\epsilon \partial_xg_1(x)K_1^\epsilon(x,0)K_2^\epsilon(x,0).
	\end{split}
	\end{equation*}
Now, we use the last decomposition and the estimates (\ref{reg_eq_bounds}) together with the Minkowski's integral inequality to obtain that
	\begin{equation*}
		\begin{split}
			\bigg(\int_{\mathbb R}\bigg|\int_{\abs{\alpha}<1}\frac{\chi_\epsilon(g_1(x-\alpha)-g_2(x-\alpha))}{\alpha} \partial_x\Delta_{\alpha}(\chi_\epsilon g_2)\Delta_{\alpha}(\chi_{\epsilon}g_1) K_{1}^\epsilon(x,\alpha)K_2^\epsilon(x,\alpha)			d\alpha \bigg|^2dx\bigg)^{1/2}&\\
&\hspace{-14cm}\leq \int_{\abs{\alpha}<1}\norm{(\partial_x\chi_\epsilon) \partial_x^2g_2}_{L^\infty}\norm{\chi_\epsilon \partial_xg_1}_{L^\infty}\bigg(\int_{\mathbb R}\chi_\epsilon (g_1(x-\alpha)-g_2(x-\alpha))^2dx\bigg)^{1/2}d\alpha\\
&\hspace{-14cm}\leq \int_{\abs{\alpha}<1}\norm{\chi_\epsilon\partial_x^2g_2}_{L^\infty}\norm{\chi_\epsilon \partial_x^2g_1}_{L^\infty}\bigg(\int_{\mathbb R}\chi_\epsilon (g_1(x-\alpha)-g_2(x-\alpha))^2dx\bigg)^{1/2}d\alpha\\
&\hspace{-14cm}+\int_{\abs{\alpha}<1}\norm{\chi_\epsilon\partial_x^2g_2}_{L^\infty}\norm{\chi_\epsilon \partial_xg_1}_{L^\infty}(1+\norm{\chi_\epsilon \partial_x^2g_1}_{L^\infty})\bigg(\int_{\mathbb R}\chi_\epsilon (g_1(x-\alpha)-g_2(x-\alpha))^2dx\bigg)^{1/2}d\alpha\\
&\hspace{-14cm}+\int_{\abs{\alpha}<1}\norm{\chi_\epsilon\partial_x^2g_2}_{L^\infty}\norm{\chi_\epsilon \partial_xg_1}_{L^\infty}(1+\norm{\chi_\epsilon \partial_x^2g_2}_{L^\infty})\bigg(\int_{\mathbb R}\chi_\epsilon (g_1(x-\alpha)-g_2(x-\alpha))^2dx\bigg)^{1/2}d\alpha\\
&\hspace{-14cm}+\norm{\chi_\epsilon\partial_x^2g_2}_{L^\infty}\norm{\chi_\epsilon \partial_xg_1}_{L^\infty}\norm{H_{\abs{\alpha}<1}\chi_\epsilon(g_1-g_2)}_{L^2}\\
&\hspace{-14cm}\leq c\,(\norm{g_1}_{L^2},\norm{g_2}_{L^2},\epsilon)\norm{g_1-g_2}_{L^2}.
		\end{split}
	\end{equation*}
Analogously, we obtain a similar bound for the third term in $B_\epsilon(x,\alpha).$  For the first term in $B_\epsilon(x,\alpha)$ we decompose 
\[2\Delta_{\alpha}fK_1^\epsilon(x,\alpha)K_2^\epsilon(x,\alpha)=2(\Delta_{\alpha}f+\chi_\epsilon g_1)K_1^\epsilon(x,\alpha)K_2^\epsilon(x,\alpha)-2\chi_\epsilon \Delta_{\alpha}g_1K_1^\epsilon(x,\alpha)K_2^\epsilon(x,\alpha),\]
and  repeat the previous argument. Thus, we conclude that
\[\norm{T_4}_{L^2}\leq c\,(\norm{g_1}_{L^2},\norm{g_2}_{L^2},\epsilon)\norm{g_1-g_2}_{L^2}.\]
 By joining the estimates for  $T_1,T_2,T_3$ and $T_4$ we obtain the bound for $R_1.$  For $R_2$ we observe from the definitions (\ref{defn_kers_reg}) and  (\ref{differences_reg}) the following
    \begin{equation*}
        \begin{split}
 R_2(x)=2\int_{\mathbb R}\big[K_1^\epsilon(x,\alpha)-K_2^\epsilon(x,\alpha)\big]d\alpha=2\int_{\mathbb R}\big[\Delta_\alpha(\chi_\epsilon g_1)-\Delta_\alpha(\chi_\epsilon g_2)\big]B_\epsilon(x,\alpha)d\alpha.
        \end{split}
    \end{equation*}
Thus, similarly to $R_1$ we  obtain the next estimate
    \[\norm{R_2}_{L^2}\leq c(\norm{g_1}_{L^2},\norm{g_2}_{L^2},\epsilon)\norm{g_1-g_2}_{L^2}.\]
Therefore using the properties of mollifiers (\ref{properties_mol}) together with the bounds for $R_1$ and $R_2,$ we deduce that
    \[\norm{M^\epsilon(g_1)-M^\epsilon(g_2)}_{H^s}\leq c\,\epsilon^{-s}\norm{\mathfrak M^\epsilon(g_1)-\mathfrak M^\epsilon(g_2)}_{L^2}\leq c(\norm{g_1}_{L^2},\norm{g_2}_{L^2},\epsilon)\norm{g_1-g_2}_{L^2}.\]
   Finally, we conclude 
    \[\norm{M^\epsilon(g_1)-M^\epsilon(g_2)}_{H^s}\leq c( \norm{g_1}_{L^2},\norm{g_2}_{L^2},\epsilon)\norm{g_1-g_2}_{H^s}.\]
Thus the operator $M^\epsilon$ is locally Lipschitz on the open set $\mathcal O.$ The Picard theorem implies that there exists an unique solution $g^\epsilon\in C^1([0,T_\epsilon]: \mathcal O )$ of (\ref{regularized_g}) which completes the proof. 
\end{proof}

     Due to the properties of mollifiers (\ref{properties_mol}) we use the energy estimate obtained in section \ref{energy} and the time of existence $T_\epsilon >0$ can be changed for a time that  depends only on the initial data $g_0\in H^s(\mathbb R).$ That is
     \begin{equation}\label{uniform_bound}
        \norm{g^\epsilon(t)}_{H^3}\leq \frac{\norm{g_0}_{H^3}}{\Big(1-c[\phi(0)]^3t\Big)^{1/3}},
     \end{equation}
    and it follows that $g^\epsilon(\,\cdot\,, t)\in H^3(\mathbb R)$ when 
    $t <T^\star.$ The next step is to prove that the regularized system forms a  Cauchy sequence  with respect to the norm $L^2(\mathbb R)$ which is  the next lemma where we choose $T_0<T^\ast$.
    \begin{lemma}\label{cauchy_sequence} The sequence of regularized solutions forms  a Cauchy sequence  in $C([0,T_0]:L^2(\mathbb R))$ and we have the estimate 
        \[\norm{g^\epsilon-g^{\epsilon'}}_{L^2}(t)\leq c\,(T_0)\,( \epsilon+\epsilon'),\]
    for $\epsilon\neq \epsilon'$ and therefore there exists a limit function $g\in C([0,T_0):L^2(\mathbb R))$ such that $g^\epsilon \to g.$
\end{lemma}

\begin{proof}
Taking the $L^2(\mathbb R)$ product  and using the Cauchy-Schwarz inequality we obtain 
    \begin{equation*}
        \begin{split}
            \frac{1}{2}\frac{d}{dt}\norm{g^\epsilon-g^{\epsilon'}}_{L^2}^2&=\int_{\mathbb R}(g^\epsilon-g^{\epsilon'})(M^\epsilon (g^\epsilon)-M^{\epsilon'}(g^{\epsilon'}))dx\\
            &\leq \norm{g^\epsilon-g^{\epsilon'}}_{L^2}\norm{M^\epsilon(g^\epsilon)-M^{\epsilon'}(g^{\epsilon'})}_{L^2}.\\
        \end{split}
    \end{equation*}
We add and subtract $M^{\epsilon'}(g^\epsilon)$ and by using the \cref{lemma_picard_reg}, it follows
    \begin{equation*}
        \begin{split}
            \norm{M^\epsilon(g^\epsilon)-M^{\epsilon'}(g^{\epsilon'})}_{L^2}&\leq \norm{M^\epsilon(g^\epsilon)-M^{\epsilon'}(g^{\epsilon})}_{L^2}+\norm{M^{\epsilon'}(g^{\epsilon})-M^{\epsilon'}(g^{\epsilon'})}_{L^2}\\
            &\leq \norm{M^{\epsilon}(g^{\epsilon})-M^{\epsilon'}(g^{\epsilon})}_{L^2}+ c(T_0)\norm{g^\epsilon-g^{\epsilon'}}_{L^2}.
        \end{split}
    \end{equation*}
For the first term in the last inequality we add and subtract  $\mathfrak M^\epsilon(g^\epsilon)$ and  $\mathfrak M^{\epsilon'}(g^\epsilon),$ then we get 
        \begin{equation*}
                \begin{split}
                   \norm{M^{\epsilon}(g^{\epsilon})-M^{\epsilon'}(g^{\epsilon})}_{L^2}&  \leq \norm{\chi_\epsilon \mathfrak M^\epsilon(g^\epsilon)-\mathfrak M^\epsilon(g^\epsilon)}_{L^2}+\norm{\chi_{\epsilon'} \mathfrak M^{\epsilon'}(g^\epsilon)-\mathfrak M^{\epsilon'}(g^\epsilon)}_{L^2}\\
&+\norm{\mathfrak M^\epsilon(g^\epsilon)-\mathfrak M^{\epsilon'}(g^\epsilon)}_{L^2}.\\
 \end{split}
\end{equation*}
Using  the  properties of mollifiers  (\ref{properties_mol}) we deduce that 
		\begin{equation}\setlength{\jot}{10pt}\label{l2difference_cauchy_sec}
			\begin{split}
				 \norm{M^{\epsilon}(g^{\epsilon})-M^{\epsilon'}(g^{\epsilon})}_{L^2}&\leq c\,\epsilon\norm{\mathfrak M^{\epsilon}(g^\epsilon)}_{H^1}+c\,\epsilon'\norm{\mathfrak M^{\epsilon'}(g^\epsilon)}_{H^1}+\norm{\mathfrak M^\epsilon(g^\epsilon )-\mathfrak M^{\epsilon'}(g^\epsilon )}_{L^2}.
                \end{split}
        \end{equation}
The bound for the last term in (\ref{l2difference_cauchy_sec}) is obtained by  applying the  \cref{lemma_picard_reg} with  $g_1=\chi_\epsilon g^\epsilon$   and $g_2=\chi_{\epsilon'}g^{\epsilon},$ that is
        \begin{equation*}\setlength{\jot}{12pt}
            \begin{split}
                \norm{\mathfrak M^{\epsilon}( g^\epsilon )-\mathfrak M^{\epsilon'}( g^\epsilon )}_{L^2}&\leq c(T_0)\norm{\chi_{\epsilon}g^\epsilon-\chi_{\epsilon'}g^\epsilon}_{L^2}.\\
            \end{split}
        \end{equation*}
Hence by adding and subtracting $g^\epsilon$ and using the properties of mollifiers (\ref{properties_mol}) we find that 
        \begin{equation*}\setlength{\jot}{12pt}
            \begin{split}
      \norm{\chi_{\epsilon}g^\epsilon-\chi_{\epsilon'}g^\epsilon}_{L^2}          &= \norm{\chi_{\epsilon}g^\epsilon- g^\epsilon+ g^\epsilon -\chi_{\epsilon'}g^\epsilon}_{L^2}\\
                &\leq \norm{\chi_{\epsilon}g^\epsilon- g^\epsilon}_{L^2}+\norm{\chi_{\epsilon'}g^\epsilon-g^\epsilon}_{L^2}\\
                &\leq c\,\epsilon \norm{g^\epsilon}_{H^1}+c\,\epsilon' \norm{g^\epsilon}_{H^1}.
            \end{split}
        \end{equation*}
Because the solutions  $g^\epsilon$ are uniformly bounded by relation (\ref{uniform_bound}), we obtain the following 
         \begin{equation*}
			\begin{split}
				\frac{1}{2}\frac{d}{dt}\norm{g^\epsilon-g^{\epsilon'}}_{L^2}^2&\leq c\,(T_0)\norm{g^\epsilon-							g^{\epsilon'}}_{L^2}^2+c\,(T_0)(\epsilon+\epsilon')\norm{g^\epsilon-g^{\epsilon'}}_{L^2}.							\end{split}		
		\end{equation*}
Hence
		\[		\frac{1}{2}\frac{d}{dt}\norm{g^\epsilon-g^{\epsilon'}}_{L^2}\leq c\,(T_0)\big[\epsilon+\epsilon' +\norm{g^{\epsilon}-g^{\epsilon'}}_{L^2}\big].\]
Finally, we integrate  with respect to $t$ to conclude  that 
                \[\norm{g^\epsilon-g^{\epsilon'}}_{L^2}(t)\leq c\,(T_0)(\epsilon+\epsilon'),\]
                and this completes the proof.
\end{proof}

 \noindent Now we prove the main result
\begin{proof}[Proof of the main \cref{main_theorem}] By \cref{lemma_picard_reg}, there exists  solutions  $\{g^\epsilon\}$ of the regularized problem and from the energy estimate   they are uniformly bounded in $H^3(\mathbb R),$ These solutions can be continued for all time, see  theorem 3.3 in  \cite{majda2002vorticity}. By \cref{cauchy_sequence} the solutions $\{g^\epsilon\}$ forms a Cauchy sequence in  $C([0,T_0]:L^2(\mathbb R))$ hence $\{g^\epsilon\}$  converges to a function $g\in C([0,T_0]:L^2(\mathbb R)).$  Now we use Sobolev interpolation,  for any $0<s<3,$ there exists a constant $c_s>0$ such that 
        \[\norm{f}_{H^{s}}\leq c_s\norm{f}_{L^2}^{1-s/3}\norm{f}_{H^3}^{s/3}\quad\mbox{for all}\quad f\in H^3(\mathbb R). \]
We apply the previous inequality  to  the difference $g^\epsilon-g^{\epsilon'}$ to derive the following
    \begin{equation*}\setlength{\jot}{12pt}
        \begin{split}
            \norm{g^\epsilon-g^{\epsilon'}}_{H^{s}}&\leq c_s\,\norm{g^\epsilon-g^{\epsilon'}}_{L^2}^{1-s/3}\norm{g^\epsilon-g^{\epsilon'}}_{H^3}^{s/3}\\
            &\leq c\,(s,T_0)(\epsilon+\epsilon')^{1-s/3}\norm{g^\epsilon-g^{\epsilon'}}_{H^3}^{s/3}\\
            &\leq c\,(s,T_0)(\epsilon+\epsilon')^{1-s/3}.
        \end{split}
    \end{equation*}
Therefore $\{g^\epsilon\}$ forms  a Cauchy sequence in $H^{s}(\mathbb R),$ and this implies strong convergence in the space $C([0,T_0]:H^{s}(\mathbb R))$ for $s<3$ and the limit function $g$ satisfies the equation (\ref{equation_g}).\\

 \noindent For the rest of the proof  we follow several steps.\\

    \noindent{\color{blue}{\emph{Step 1}}}: Fix $t\in [0,T_0],$  we use  the energy estimate to obtain that $\{g^{\epsilon}(\cdot, t)\}$ is  a sequence uniformly bounded  in $H^3(\mathbb R).$  The  Banach-Alaoglu theorem implies that there exists a subsequence $\{g^{\epsilon}(\cdot,t)\}$  that  converges weakly to some function $\tilde g(\cdot,t)\in H^3(\mathbb R).$ \\
    
  \noindent{\color{blue}{\emph{Step 2}}}: The weak limit and the strong limit are equal pointwise in time, that is, $\tilde g(\cdot,t)=g(\cdot,t),$ where  $g$ is the function of the strong convergence in $H^{s}(\mathbb R)$ for all $t\in[0,T_0].$ We take $\varphi\in H^{-s}(\mathbb R)$ and for $g\in H^s(\mathbb R)$ we denote $\langle g,\varphi\rangle_{s}$ as the dual pairing of $H^{s}(\mathbb R)$ and $H^{-s}(\mathbb R)$ through the $L^2(\mathbb R)$ product.
    Using the  weak convergence
    \[\langle g^{\epsilon} (\cdot,t), \varphi\rangle_{3} \to \langle \tilde g(\cdot,t), \varphi \rangle_{3}, \quad\mbox{as}\quad\epsilon\to 0 \quad\mbox{for all}\quad\varphi\in H^{-3}(\mathbb R),\]
and  the inclusion $L^2(\mathbb R)\subset H^{-3}(\mathbb R),$ we see that 
    \[
    \int_{\mathbb R}\big[g^{\epsilon}(x,t)-\tilde g(x,t)\big]\varphi(x) dx\to 0, \quad\mbox{as}\quad \epsilon\to 0\quad\mbox{for all}\quad\varphi\in L^2(\mathbb R). \]
The strong convergence in $H^{s}(\mathbb R)$ implies weak convergence in $H^{s}(\mathbb R),$ thus  for the same function $\varphi\in L^2(\mathbb R)$ we have 
    \[\langle g(\cdot,t)^{\epsilon}- g(\cdot,t),\varphi\rangle_{s}\to 0, \quad\mbox{as}\quad \epsilon\to 0. \]
Therefore if $\tilde g(\cdot, t)\neq g(\cdot,t)$  we get
\[\langle g(\cdot,t)-\tilde g(\cdot,t), \varphi\rangle_0=\langle g(\cdot,t)-g^\epsilon(\cdot,t), \varphi\rangle_0+\langle g^\epsilon(\cdot,t)-\tilde g(\cdot,t), \varphi\rangle_0\to 0\]
and we have a contradition, therefore the weak limit  $\tilde g(\cdot,t)$ is equal pointwise in time to the strong limit $g(\cdot,t).$ Hence $g(\cdot,t)\in H^3(\mathbb R)$ for every $t\in[0,T_0].$\\

 \noindent{\color{blue}{\emph{Step 3}}}:  The limit function  $g\in C_\mathrm{w} ([0,T_0]:H^{3}(\mathbb R)).$ Using that  $H^{-s}(\mathbb R)$ is dense in $ H^{-3}(\mathbb R)$ for $s<3,$ we  take $\varphi\in H^{-3}(\mathbb R)$ and   $\epsilon>0,$ then  there exists $\varphi'\in H^{-s}(\mathbb R)$ such that 
    \[\norm{\varphi-\varphi'}_{H^{-3}}<\epsilon.\] 
    The uniform bound for $g^\epsilon$ together with the  triangle inequality and the Cauchy-Schwarz inequality implies  that
    \begin{equation*}\setlength{\jot}{12pt}
        \begin{split}
            \abs{\langle g^{\epsilon}(\cdot,t) -g(\cdot,t),\varphi \rangle_{3}}   &\leq \abs{\langle g^{\epsilon}(\cdot,t)-g(\cdot,t),  \varphi-\varphi'\rangle_{3}} +\abs{\langle (g^{\epsilon}-g)(\cdot,t),\varphi' \rangle_{3}}\\
            &\leq 2c\,(T_{0})\norm{\varphi-\varphi'}_{H^{-3}}+\norm{\varphi'}_{H^{-s}}\norm{g^{\epsilon}(t)-g(t)}_{H^s}.
        \end{split}
    \end{equation*}
Using the strong convergence in $H^{s}(\mathbb R)$ we have
    \[\abs{\langle g^{\epsilon}(\cdot,t) -g(\cdot,t),\varphi \rangle_{3}} \leq \epsilon \, c\,(T_0).\]
The last inequality implies that 
\[\langle g^{\epsilon}(\cdot,t),\varphi \rangle_{3}\to \langle g(\cdot,t),\varphi\rangle_{3}\]
as $\epsilon\to 0$ uniformly, therefore the limit $\langle g(\cdot,t),\phi\rangle_3$ is a continuous function in time over $[0,T_0],$ and the arbitrary choice of $\varphi\in H^{-3}(\mathbb R)$ implies that  $g\in C_{\mathrm{w}}([0,T_0]:H^3(\mathbb R)).$\\
   \end{proof}
\begin{remark}The limit solution belongs  to $H^3(\mathbb R)$ for every $t\in[0,T_0]$ and  we have 
\[g\in L^\infty([0,T_0]: H^3(\mathbb R)).\]
We observe that this argument is not sufficient to prove the continuity in time of the limit solution, due to the loss of parabolicity in the equation.
\end{remark}

\begin{acknowledgements}\emph{The author thanks \'Angel Castro and Daniel Faraco, 
his advisors,  for their suggestions to address this problem and  for their valuable comments and  discussions  in the preparation of this work. The author acknowledges financial support from the Spanish Ministry of Economy under the ICMAT Severo Ochoa grant PRE2018-084508, the Severo Ochoa Programme for Centres of Excellence Grant CEX2019-000904-S funded by MCIN/AEI/10.13039/501\\
100011033, the Grant PID2020-114703GBI00 funded by MCIN/AEI/10.13039/501100011033, the grant PI2021-124-195NB-C32 funded by MCIN/AEI/10.13039/501100011033 and the ERC Advanced Grant 834728.   Finally the author acknowledges financial support from Grants RED20\\
22-134784-T and RED2018-102650-T funded by MCIN/AEI/10.13039/501100011033.}
\end{acknowledgements}

\bibliographystyle{plain} 
\bibliography{references}
\hspace{2cm}

\noindent  Departamento de Matem\'aticas, Universidad Aut\'onoma de Madrid, 28049 Madrid, Spain\\
Instituto de Ciencias Matem\'aticas (CSIC-UAM-UC3M-UCM), 28049 Madrid, Spain\\
\textit{E-mail address:} \href{mailto:omar.sanchez@uam.es}{omar.sanchez@uam.es}\\

\end{document}